\documentclass[a4paper,10pt]{article}
\usepackage{latexsym, amssymb, amsmath, amsthm, bbm}
\usepackage[all,cmtip]{xy}
\usepackage{amsmath}
\usepackage{amsthm}
\usepackage{amssymb}
\usepackage{amsfonts}
\usepackage{amsrefs}
\usepackage{color,authblk}
\usepackage{tikz-cd}
\usepackage{amscd} 
\usepackage{comment}
\usepackage{mathrsfs}
\usepackage[all]{xy}
\usepackage{graphicx}
\usepackage{authblk}
\usepackage{hyperref}
\usepackage{stackrel}
\usepackage{fullpage}
\usepackage[all,cmtip]{xy}
\usepackage{amsmath,amscd}
\usepackage[title]{appendix}
\newtheorem{thm}{Theorem}

\newtheorem{theorem}{Theorem}[section]
\newtheorem{lemma}[theorem]{Lemma}
\newtheorem{example}[theorem]{Example}

\newtheorem{prop}[theorem]{Proposition}
\newtheorem{corollary}[theorem]{Corollary}

\newtheorem{definition}[theorem]{Definition}

\newcommand{\leftrarrows}{\mathrel{\raise.75ex\hbox{\oalign{%
  $\scriptstyle\leftarrow$\cr
  \vrule width0pt height.5ex$\hfil\scriptstyle\relbar$\cr}}}}
\newcommand{\lrightarrows}{\mathrel{\raise.75ex\hbox{\oalign{%
  $\scriptstyle\relbar$\hfil\cr
  $\scriptstyle\vrule width0pt height.5ex\smash\rightarrow$\cr}}}}
\newcommand{\Rrelbar}{\mathrel{\raise.75ex\hbox{\oalign{%
  $\scriptstyle\relbar$\cr
  \vrule width0pt height.5ex$\scriptstyle\relbar$}}}}

\makeatletter
\def\leftrightarrowsfill@{\arrowfill@\leftrarrows\Rrelbar\lrightarrows}
\newcommand{\xleftrightarrows}[2][]{\ext@arrow 3399\leftrightarrowsfill@{#1}{#2}}
\makeatother

\numberwithin{equation}{section}

\title{Entwined modules over linear categories and Galois extensions}
\author{Mamta Balodi\footnote{Email: mamta.balodi@gmail.com} $\quad$ Abhishek Banerjee\footnote{Email: abhishekbanerjee1313@gmail.com} $\quad$ Samarpita Ray\footnote{Email: ray.samarpita31@gmail.com}\\ \small \emph{Department of Mathematics, Indian Institute of Science, Bangalore - 560012, India.} }
\date{}

\begin{document}
\maketitle

\begin{abstract}
In this paper, we study modules over quotient spaces of certain categorified fiber bundles.  These are understood as modules over entwining structures involving a small $K$-linear category 
$\mathcal D$ and a $K$-coalgebra $C$. 
We obtain  Frobenius and separability conditions for functors on entwined modules. We also introduce the notion of a $C$-Galois extension $\mathcal E\subseteq \mathcal D$ of categories. Under suitable conditions, we show that entwined modules over a $C$-Galois extension may be described as modules over the subcategory $\mathcal E$ of 
$C$-coinvariants of $\mathcal D$.

\end{abstract}

\smallskip

\smallskip

\smallskip
{\emph{\bf MSC(2010) Subject Classification:} 16W30, 18E05}


\smallskip

\section{Introduction}

The purpose of this paper is to study a theory of  modules over quotient spaces of certain categorified fiber bundles. Suppose that $X$ is an affine scheme over a field $K$ and let $G$ be an affine algebraic group scheme with a free action $\sigma: X\times G\longrightarrow X$ on $X$. Let $Y$ be the quotient given by the coequalizer
\begin{equation}\label{intro1}
 \xymatrix{X\times G\ar@<-.5ex>[r]_(0.6){pr} \ar@<.5ex>[r]^(0.6){\sigma}& X}\overset{p}{\longrightarrow} Y
\end{equation} If $X\longrightarrow Y$ is faithfully flat and the canonical map $can: X\times G\longrightarrow X\times_YX$ is an isomorphism, then $X$ is said to be (see, for instance, \cite{Mum}, \cite{HJS})
a principal fiber bundle over $Y$ with group $G$. 

\smallskip
The algebraic counterpart of \eqref{intro1} consists of an algebra $A$, a  Hopf algebra $H$ and a coaction $\rho: A\longrightarrow A\otimes H$ that makes $A$ into a right $H$-comodule algebra. Let $B:=A^{coH}=\{\mbox{$a\in A$ $\vert$ $\rho(a)=a\otimes 1_H$}\}$ be the algebra of coinvariants of $A$, i.e., $B$ is given by the equalizer
\begin{equation}\label{intro2}
 B\longrightarrow \xymatrix{A\ar@<-.5ex>[r]_(0.4){in} \ar@<.5ex>[r]^(0.4){\rho}& A\otimes H} 
\end{equation} In this case, there is a canonical map $can: A\otimes_BA\longrightarrow A\otimes H$ determined by setting $can(x\otimes y)=x\cdot \rho(y)$. If the Hopf algebra $H$
has bijective antipode, $B\longrightarrow A$ is a faithfully flat extension and $can: A\otimes_BA\longrightarrow A\otimes H$  is an isomorphism, it was shown by Schneider \cite{HJS} that
modules over $B$ may be recovered as the category of ``$(A,H)$-Hopf modules.'' 

\smallskip
We work with a small $K$-linear category $\mathcal D$, a $K$-coalgebra $C$ and an ``entwining structure'' $\psi$ consisting of a collection of morphisms
\begin{equation*}\psi=\{\psi_{XY}:C \otimes Hom_\mathcal{D}(X,Y) \longrightarrow Hom_\mathcal{D}(X,Y) \otimes C\}_{(X,Y)\in Ob(\mathcal D)^2}
\end{equation*}
satisfying   conditions that we lay out in Section 2. We consider the category ${\mathscr{M}(\psi)}_\mathcal{D}^C$ of modules over the entwining structure
$(\mathcal D,C,\psi)$ (see Definition \ref{ddde}). These may be seen as modules over a ``categorical quotient space'' of $\mathcal D$ with respect to the coalgebra $C$ and the entwining
$\psi$. 

\smallskip The notion of a $C$-Galois extension $\mathcal E\subseteq \mathcal D$ of categories is introduced in Section 4. Additionally, a $C$-Galois
extension gives rise to a canonical entwining structure on $\mathcal D$. Under certain conditions, we show that modules over the category $\mathcal E$ of $C$-coinvariants of $\mathcal D$ may be described as modules 
over the canonical entwining structure.

\smallskip Entwining structures for algebras were introduced by Brzezi\'{n}ski and Majid in \cite{BrMj} and it was realized in Brzezi\'{n}ski \cite{Brz00} that
entwined modules provide a unifying formalism for studying diverse concepts such as relative Hopf modules, Doi-Hopf and Yetter-Drinfeld modules as well as coalgebra Galois extensions. In fact, the study of entwining structures for algebras and entwined modules over them is  well developed in the literature and we refer the reader, for instance, to \cite{Abu}, \cite{Brz00}  \cite{Br02},  \cite{BCT1}, \cite{BCT}, \cite{CG},  \cite{Jia}, \cite{Sch} for more on this subject.

\smallskip
Our notion of modules over an entwining structure $(\mathcal D,C,\psi)$ builds on the analogy of Mitchell \cite{Mit1} which says that a small $K$-linear category
should be seen as a ``$K$-algebra with several objects.'' In particular, the category ${\mathscr{M}(\psi)}_\mathcal{D}^C$ also generalizes the ``relative $(\mathcal D,H)$-Hopf modules'' studied
in our previous work in \cite{BBR}, where $H$ is a Hopf algebra and $\mathcal D$ is an $H$-comodule category in the sense of Cibils and Solotar \cite{CiSo}. In other words, 
$\mathcal D$ is a small $K$-linear category whose morphism spaces are equipped with a coaction of $H$ that is compatible with composition. When $\mathcal D$ has a single object, it reduces
to an ordinary $H$-comodule algebra and  the relative $(\mathcal D,H)$-Hopf modules
reduce to the usual notion of relative Hopf modules (see Takeuchi \cite{Take}).

\smallskip
For Doi-Hopf modules, Frobenius and separability conditions were studied extensively in a series of papers \cite{CMIZ}, \cite{CMZ1}, \cite{CMZ}. Later, 
 Brzezi\'{n}ski   studied Frobenius and Maschke type theorems for entwined modules in \cite{Brz}. In this paper, we proceed in a manner analogous
 to the unified approach of Brzezi\'{n}ski, Caenepeel, Militaru and Zhu \cite{uni} for  studying Frobenius and separability conditions for entwined modules over $(\mathcal D,C,\psi)$.  
 
 \smallskip
 The idea is as follows: the ``categorical quotient space'' of $\mathcal D$ with respect to $C$ and $\psi$ may be thought of as a subcategory of $\mathcal D$ and ${\mathscr{M}(\psi)}_\mathcal{D}^C$ plays the role of modules over this subcategory. Although this ``subcategory'' of $\mathcal D$ need not exist in an explicit sense, we would like to study the properties of this extension
 of categories. In particular, we would like to know if it behaves like a separable, split or Frobenius extension of small $K$-linear categories. For this, we turn to a pair of functors
 \begin{equation*}
 \begin{array}{c}
 \mathscr F: {\mathscr{M}(\psi)}_\mathcal{D}^C\longrightarrow Mod\text{-}\mathcal D\qquad
 \mathscr G:Mod\text{-}\mathcal D\longrightarrow {\mathscr{M}(\psi)}_\mathcal{D}^C\\
 \end{array}
 \end{equation*} Here $\mathscr F$ is the left adjoint and behaves like an ``extension of scalars'' whereas its right adjoint $\mathscr G$ behaves like a ``restriction of scalars.'' We recall here (see \cite[Theorem 1.2]{uni}) that in the classical case of an extension $R\longrightarrow S$ of rings inducing the pair of adjoint functors $Mod\text{-}R  
\xleftrightarrows[\text{$\qquad F\qquad $}]{\text{$\qquad G\qquad$}} Mod\text{-}S
$ given
 by extension and restriction of scalars, we have:
 \begin{equation*}
 \begin{array}{ccc}
 \mbox{$R\longrightarrow S$ is split extension} &\qquad \Leftrightarrow \qquad & \mbox{Left adjoint $F:Mod\text{-}R\longrightarrow Mod\text{-}S$ is separable}\\ && \\
  \mbox{$R\longrightarrow S$ is separable extension} &\qquad \Leftrightarrow \qquad & \mbox{Right adjoint $G:Mod\text{-}S\longrightarrow Mod\text{-}R$ is separable}\\ && \\
   \mbox{$R\longrightarrow S$ is Frobenius extension} &\qquad \Leftrightarrow \qquad & \mbox{$(F,G)$ is Frobenius pair of functors}\\
 \end{array}
 \end{equation*} It is therefore natural to study criteria for the separability of the functors $\mathscr F$ and $\mathscr G$ as well as conditions for $(\mathscr F,\mathscr G)$
 to be a Frobenius pair of functors.
 
 \smallskip 
 In this paper, we will always use the following convention: for $f\in Hom_{\mathcal D}(Y,X)$ and $c\in C$, we write $\psi_{YX}(c\otimes f)=f_\psi\otimes c^\psi\in Hom_{\mathcal D}(Y,X)\otimes C$ with the summation omitted. We write $h:\mathcal D^{op}\otimes \mathcal D\longrightarrow Vect_K$ for the canonical $\mathcal D\text{-}\mathcal D$-bimodule $h(Y,X)=Hom_{\mathcal D}(Y,X)$. The entwining structure makes $h\otimes C$ into a $\mathcal D\text{-}\mathcal D$-bimodule by setting 
 \begin{equation*}
(h \otimes C)(Y,X):=Hom_\mathcal{D}(Y,X) \otimes C \qquad \big((h \otimes C)(\phi)\big)(f \otimes c):=\phi''f\phi'_\psi \otimes c^\psi
\end{equation*}
for any $(Y,X) \in Ob(\mathcal{D}^{op} \otimes \mathcal{D})$, $\phi:=(\phi',\phi'') \in Hom_{\mathcal{D}^{op} \otimes \mathcal{D}}\big((Y,X),(Y',X')\big)$, $f \in Hom_\mathcal{D}(Y,X)$ and $c \in C$.  We consider a collection $\theta:=\{\theta_{X}:  C \otimes C \longrightarrow End_\mathcal{D}(X)\}_{X \in Ob(\mathcal{D})}$ of $K$-linear maps  satisfying the following conditions:
\begin{equation*}
\left( \theta_{X}(c \otimes d)\right)\circ f= {f_\psi}_\psi\circ \theta_{Y}\left(c^\psi \otimes d^\psi \right) \qquad 
\theta_{X}(c \otimes d_1) \otimes d_2 = {\left(\theta_{X}(c_2 \otimes d)\right)}_\psi \otimes {c_1}^\psi  
\end{equation*}
for any $f \in Hom_{\mathcal{D}}(Y,X)$. Let $V_1$ be the $K$-space consisting of all such $\theta$.  
 Our first   result gives conditions for the functors $\mathscr F$ and $\mathscr G$ to be separable.
 
 \begin{thm}\label{ThmI} (see \ref{V=V1}, \ref{ThmIx}, \ref{W=W1} and \ref{ThmIIx}) Let $\mathcal D$ be a small $K$-linear category, $(C,\Delta_C,\varepsilon_C)$ be a $K$-coalgebra and let $(\mathcal D,C,\psi)$ be a right-right entwining
 structure.  
 
 \smallskip
 \begin{itemize}
\item[(a)]  Let $V=Nat(\mathscr G\mathscr F,1_{{\mathscr{M}(\psi)}_\mathcal{D}^C})$ be the space of natural transformations from $\mathscr G\mathscr F$ to $1_{{\mathscr{M}(\psi)}_\mathcal{D}^C}$. Then:
 
 \begin{itemize}
 \item[(1)] There is an isomorphism $V\cong V_1$ of $K$-vector spaces.
 
 \item[(2)]  The functor $\mathscr{F}$ is separable if and only if there exists $\theta \in V_1$ such that
\begin{equation*}
\theta_{X} \circ \Delta_C =  \varepsilon_C \cdot id_X \qquad \forall X \in Ob(\mathcal{D})
\end{equation*}
 \end{itemize} 
 
 \end{itemize}
 
 \begin{itemize}\item[(b)] Let $W=Nat(1_{Mod\text{-}\mathcal D},\mathscr F\mathscr G)$ be the space of natural transformations from  $1_{Mod\text{-}\mathcal D}$ to $\mathscr F\mathscr G$. Then:
 
 \begin{itemize} 
 
 \item[(1)] There is an isomorphism of $K$-vector spaces from $W$ to $W_1=Nat(h,h\otimes C)$. 
 
 \item[(2)] The functor  $\mathscr{G}$ is separable if and only if there exists $\eta \in W_1=Nat(h, h \otimes C)$ such that
$
(id_h \otimes \varepsilon_C)\eta=id_h
$.
 
 \end{itemize}
 
 \end{itemize}
 
 \end{thm} 
 
 The next result gives conditions for $(\mathscr F,\mathscr G)$ to be a Frobenius pair. 
 
 \begin{thm}\label{iFrobcondition} (see \ref{Frobcondition}) Let $\mathcal D$ be a small $K$-linear category, $(C,\Delta_C,\varepsilon_C)$ be a $K$-coalgebra and let $(\mathcal D,C,\psi)$ be a right-right entwining
 structure.  
 Then, $(\mathscr{F},\mathscr{G})$ is a Frobenius pair if and only if there exist $\theta \in V_1$ and $\eta \in W_1$ such that the following conditions hold:
\begin{equation*}
\varepsilon_C(d)f=\sum \hat{f} \circ \theta_{X}(c_f \otimes d) \qquad 
\varepsilon_C(d)f = \sum \hat{f}_\psi \circ \theta_{X}(d^\psi \otimes c_f) 
\end{equation*}
for any $f \in Hom_\mathcal{D}(X,Y)$, $d \in C$ and $\eta(X,Y)(f)=\sum \hat{f} \otimes c_f$.
\end{thm}

More generally, the $\mathcal D\text{-}\mathcal D$-bimodule $h\otimes C$ may be treated as a functor  $h\otimes C:\mathcal D\longrightarrow {\mathscr{M}(\psi)}_{\mathcal{D}}^C$  by setting (see Lemma \ref{lem 6.2})
\begin{equation*}
\begin{array}{l}
(h \otimes C)(Y):=Hom_{\mathcal D}(-,Y)\otimes C\qquad 
(h\otimes C)(f)(Z)(g\otimes c):= fg\otimes c
\end{array}
\end{equation*}
for $f\in Hom_\mathcal{D}(Y,X)$ and $g \otimes c \in Hom_{\mathcal D}(Z,Y) \otimes C$. Additionally, let $C$ be a finite dimensional coalgebra and let $C^*=Hom(C,K)$ be the linear dual of $C$. Then, we show that there is a functor
$C^*\otimes h: \mathcal D\longrightarrow {\mathscr{M}(\psi)}_{\mathcal{D}}^C$. 

\begin{thm}(see \ref{ThmIIIx})
Let $(\mathcal{D},C,\psi)$ be an entwining structure and let $C$ be a finite dimensional coalgebra.  Then, the following statements are equivalent:

\smallskip
(i) $(\mathscr{F},\mathscr{G})$ is a Frobenius pair.

\smallskip
(ii) $C^*\otimes h$ and $h\otimes C$ are isomorphic as   functors from $\mathcal{D}$ to ${\mathscr{M}(\psi)}_\mathcal{D}^C$.
\end{thm}

In the final part of this paper, we study coalgebra Galois extensions of categories in a manner analogous to Brzezi\'{n}ski \cite{Brz00}, Brzezi\'{n}ski and Hajac \cite{BH}  and Caenepeel \cite{Caeni}. For this, we suppose that
every morphism space $Hom_{\mathcal D}(X,Y)$ carries the structure of a $C$-comodule $\rho_{XY}:Hom_\mathcal{D}(X,Y)\longrightarrow Hom_\mathcal{D}(X,Y)\otimes C$, $f\mapsto \sum f_0\otimes f_1$. This allows us to define a category $\mathcal E$ of $C$-coinvariants of $\mathcal D$ (see 
Definition \ref{coinv*}). Further, we say that $\mathcal D$ is a $C$-Galois extension of $\mathcal E$ if the canonical map
\begin{equation*} can_X:h\otimes_{\mathcal E}{Hom_{\mathcal D}(X,-)}\longrightarrow {Hom_{\mathcal D}(X,-)}\otimes C
\end{equation*} is an isomorphism for each $X\in Ob(\mathcal D)$ (see Definition \ref{4.5tr}). We show that a $C$-Galois extension leads to a canonical entwining structure.

\begin{thm}\label{iGal-ent} (see \ref{Gal-ent})
Let $\mathcal{D}$ be a $C$-Galois extension of $\mathcal{E}$. 
Then, 
there exists a unique right-right entwining structure $(\mathcal{D},C,\psi)$ which makes $Hom_{\mathcal D}(-,Y)$ an object in ${\mathscr{M}(\psi)}_\mathcal{D}^C$ for every $Y\in Ob(\mathcal{D})$ with its canonical $\mathcal{D}$-module structure and right $C$-coactions $\{\rho_{XY}\}_{X\in Ob(\mathcal{D})}$.
\end{thm}

Conversely, under suitable conditions, an entwining structure $(\mathcal D,C,\psi)$ may be used to express $\mathcal D$ as a $C$-Galois extension. In that case, the category
${\mathscr{M}(\psi)}_{\mathcal{D}}^C$ reduces to the category of modules over the $C$-coinvariants of $\mathcal D$. 

\begin{thm}\label{ith4.11} (see \ref{th4.11} and \ref{entlast})
Let $C$ be a $K$-coalgebra and $\mathcal{D}$ be a small $K$-linear category such that $Hom_\mathcal{D}(X,Y)$ has a right $C$-comodule structure $\rho_{XY}$ for every $X,Y \in Ob(\mathcal{D})$.   Let $\mathcal{E}$ be the subcategory of $C$-coinvariants of $\mathcal{D}$. If there exists a convolution invertible collection $\Phi=\{\Phi_{XY}:C\longrightarrow Hom_\mathcal{D}(X,Y)\}_{X,Y\in Ob(\mathcal{D})}$ of right
$C$-comodule maps, then the following are equivalent:

(i) $\mathcal{D}$ is a $C$-Galois extension of $\mathcal{E}$.

(ii) There exists a right-right entwining structure $(\mathcal{D},C,\psi)$ such that $Hom_{\mathcal D}(-,Y)$ is an object in ${\mathscr{M}(\psi)}_\mathcal{D}^C$ for every $Y\in Ob(\mathcal{D})$ with its canonical $\mathcal{D}$-module structure and right $C$-coactions $\{\rho_{XY}\}_{X \in Ob(\mathcal{D})}$.

(iii) For any $f\in Hom_\mathcal{D}(X,Y)$, the morphism $\sum f_0\circ\Phi'_{ZX}(f_1) \in Hom_\mathcal{E}(Z, Y)$ for every $Z \in Ob(\mathcal{D})$, where $\Phi'$ is the convolution  inverse of $\Phi$. 

\smallskip
In this case,  the categories ${\mathscr{M}(\psi)}_\mathcal{D}^C $ and Mod-$\mathcal{E}$ are equivalent.
\end{thm}

\smallskip
\textbf{Notations:} Throughout the paper, $K$ is a field, $C$ is a $K$-coalgebra with comultiplication $\Delta_C$ and counit $\varepsilon_C$. We shall use Sweedler's notation for the coproduct $\Delta_C(c)=  c_1 \otimes c_2,$ and for a coaction $\rho_M:M \longrightarrow M \otimes C$, $\rho_M(m)= m_0 \otimes m_1$ with the summation omitted. We denote by $C^*$ the linear dual of $C$. Sometimes when the coaction is clear from context, we will omit the subscript. 

\section{Entwining structures}\label{entwstr}

 In this section, we introduce a categorical generalization of entwining structures and entwined modules. We prove that the category of entwined modules is a Grothendieck category.
We begin by recalling the definition of modules over a category (see, for instance, \cite{ Sten,Mit2 }). 

\begin{definition}
 A right module  over a small $K$-linear category $\mathcal{D}$ is  a  $K$-linear functor $\mathcal{D}^{op} \longrightarrow Vect_K$, where $Vect_K$ denotes the category of  $K$-vector spaces. Similarly, a left module over $\mathcal{D}$ is a  $K$-linear functor $\mathcal{D} \longrightarrow Vect_K$. The category of all right (resp. left) modules over $\mathcal{D}$ will be denoted by $Mod$-$\mathcal D$  (resp. $\mathcal D$-$Mod$). 
\end{definition}

For each  $X\in Ob(\mathcal D)$, the representable functors ${\bf h}_X:=Hom_{\mathcal D}(-,X)$ and $_X{\bf h}:=Hom_{\mathcal D}(X,-)$  are examples of right and left modules over $\mathcal{D}$ respectively. Unless otherwise mentioned,
by a $\mathcal D$-module we will always mean a right  $\mathcal D$-module. 

\smallskip
Let $C$ be a $K$-coalgebra and let $\mathcal{D}$ be a small $K$-linear category. Suppose that we have a collection of $K$-linear maps 
\begin{equation*}\psi=\{\psi_{XY}:C \otimes Hom_\mathcal{D}(X,Y) \longrightarrow Hom_\mathcal{D}(X,Y) \otimes C\}_{(X,Y)\in Ob(\mathcal D)^2}
\end{equation*} We use the notation $\psi_{XY}(c\otimes f)=  f_\psi\otimes c^\psi$ for $c\in C$ and $f \in Hom_\mathcal{D}(X,Y)$. We will say that the tuple $(\mathcal{D},C,\psi)$ is a  (right-right) entwining structure if the following conditions hold:
\begin{align}
  (gf)_\psi \otimes c^\psi &= g_\psi f_\psi \otimes {c^\psi}^\psi \label{eq 6.1}\\
 \varepsilon_C(c^\psi)(f_\psi) &= \varepsilon_C(c)f \label{eq 6.2}\\
 f_\psi \otimes \Delta_C(c^\psi) &=  {f_\psi}_\psi \otimes {c_1}^\psi \otimes {c_2}^\psi \label{eq 6.3}\\
\psi_{XX}(c \otimes id_X)&= id_X \otimes c \label{eq 6.4}
\end{align}
for each $f \in Hom_\mathcal{D}(X,Y)$, $g \in Hom_\mathcal{D}(Y,Z)$ and $c \in C$.
Throughout this  paper, $(\mathcal{D},C,\psi)$ will always be an entwining structure. A morphism between entwining structures $(\mathcal{D'},C',\psi')$ and $(\mathcal{D},C,\psi)$ is a pair $(\mathscr{F},\sigma)$ where $\mathscr{F}:\mathcal{D'}\longrightarrow \mathcal{D}$ is a functor and $\sigma: C'\longrightarrow C$ is a counital coalgebra map such that $\mathscr{F}({f'}_{\psi'})\otimes \sigma({c'}^{\psi'}) = \mathscr{F}(f')_\psi \otimes \sigma(c')^\psi$ for any $c'\otimes f' \in C'\otimes Hom_\mathcal{D'}(X',Y')$ where $X',Y' \in Ob(\mathcal{D'})$.

\begin{definition}\label{ddde}
Let $\mathcal{M}$ be a right $\mathcal{D}$-module with a given right $C$-comodule structure $\rho_{\mathcal M(Y)}:\mathcal M(Y)\longrightarrow \mathcal M(Y)\otimes C$ on $\mathcal{M}(Y)$ for each $Y\in Ob(\mathcal{D})$. Then, $\mathcal{M}$ is said to be an entwined module over  $(\mathcal{D},C,\psi)$ if the following compatibility condition holds:
\begin{equation}\label{comp 2}
 \rho_{\mathcal{M}(Y)}(\mathcal{M}(f)(m))= \big({\mathcal{M}(f)(m)}\big)_0 \otimes \big({\mathcal{M}(f)(m)}\big)_{1}= \mathcal{M}(f_\psi)(m_{0}) \otimes  {m_1}^\psi
\end{equation}
for every $f \in Hom_\mathcal{D}(Y,X)$  and $m \in \mathcal{M}(X).$ We denote by ${\mathscr{M}(\psi)}_\mathcal{D}^C$ the category whose objects are entwined modules over  $(\mathcal{D},C,\psi)$ and whose morphisms are given by
$$Hom_{\mathscr{M}(\psi)_\mathcal{D}^C}(\mathcal{M},\mathcal{N}):=\{\eta \in Hom_{Mod\text{-}\mathcal{D}}(\mathcal{M},\mathcal{N})\ |\ \eta(X):\mathcal{M}(X) \longrightarrow \mathcal{N}(X)~ \text{is $C$-colinear}~ \forall X \in Ob(\mathcal{D}) \}$$
\end{definition}

We now give an important example of entwining structures. 

\begin{example}\label{2.3fed}
Let $\mathcal{D}$ be a right co-$H$-category (see \cite{CiSo} or the description in \cite[Definition 2.4]{BBR} ) and $C$ be a right $H$-module coalgebra. Then, the triple $(\mathcal{D},C,\psi)$ is an entwining structure, where $\psi$ is given by:
$$\psi_{XY}:C \otimes Hom_\mathcal{D}(X,Y) \longrightarrow C \otimes Hom_\mathcal{D}(X,Y) \otimes H \overset{\cong}{\longrightarrow} Hom_\mathcal{D}(X,Y) \otimes C \otimes H \longrightarrow Hom_\mathcal{D}(X,Y) \otimes C$$
Explicitly, we have $\psi_{XY}(c \otimes f):= f_0 \otimes cf_1$ for any $f \in Hom_\mathcal{D}(X,Y)$ and $c \in C$. In this case, an entwined module is precisely a right $\mathcal{D}$-module with a given right $C$-comodule structure on  $\mathcal{M}(X)$ for each $X \in Ob(\mathcal{D})$ and satisfying the following compatibility condition
\begin{equation*}\label{comp}
 \big({\mathcal{M}(f)(m)}\big)_0 \otimes \big({\mathcal{M}(f)(m)}\big)_{1}= \mathcal{M}(f_0)(m_{0}) \otimes  m_1f_1
\end{equation*}
We will refer to these modules as (right-right) Doi-Hopf modules and their  category will be denoted by $\mathscr{M}_\mathcal{D}^C$.  If $\mathcal{D}$ is a right co-$H$-category with a single object, i.e., an $H$-comodule algebra, then  $\mathscr{M}_\mathcal{D}^C$  recovers the classical notion of Doi-Hopf modules (see \cite{Doi}). In the particular case where $C=H$, the right-right 
Doi-Hopf modules have been referred to as relative Hopf modules in \cite[$\S$ 5]{BBR}. 
\end{example}

\begin{lemma}\label{lem 6.2}
Let $(\mathcal{D},C,\psi)$ be an entwining structure and let $\mathcal{M}$ be a right $\mathcal{D}$-module. Then, we may obtain an object $\mathcal{M}\otimes C \in {\mathscr{M}(\psi)}_\mathcal{D}^C$ by setting
\begin{equation*}\label{rightaction1}
\begin{array}{c}
(\mathcal{M}\otimes C)(X):=\mathcal{M}(X)\otimes C\\
(\mathcal{M}\otimes C)(f)(m \otimes c) :=   \mathcal{M}(f_\psi)(m)\otimes c^\psi
\end{array}
\end{equation*} 
for $X \in \text{Ob}(\mathcal{D}), f \in Hom_{\mathcal{D}}(Y,X)$ and $m \otimes c \in \mathcal{M}(X) \otimes C$. In fact, this determines a functor
from $Mod\text{-}\mathcal D$ to $ {\mathscr{M}(\psi)}_\mathcal{D}^C$ . \end{lemma}

\begin{proof}
The fact that $\mathcal{M}\otimes C$ is a right $\mathcal{D}$-module follows from \eqref{eq 6.1}. For each $X \in Ob(\mathcal{D})$, it may be verified that $\mathcal{M}(X)\otimes C$ has a right $C$-comodule structure given by
\begin{equation}\label{rightcomod1}
\pi^r_{\mathcal{M}(X)\otimes C}(m\otimes c):= (id_{\mathcal{M}(X)}\otimes \Delta_C)(m\otimes c)=   m\otimes c_1\otimes c_2
\end{equation}
It remains to check the compatibility condition in \eqref{comp 2}. By definition, we have
\begin{equation*}
\begin{array}{ll}
(\mathcal{M}(f_\psi)(m)\otimes c^\psi)_0 \otimes (\mathcal{M}(f_\psi)(m)\otimes c^\psi)_1 &= \mathcal{M}(f_\psi)(m)\otimes {(c^\psi)}_1\otimes {(c^\psi)}_2\\
&= \mathcal{M}({f_\psi}_\psi)(m)\otimes {c_1}^\psi \otimes {c_2}^\psi \quad (\text{using}~ \eqref{eq 6.3})\\
&= (\mathcal{M}\otimes C)(f_\psi)(m\otimes c_1)\otimes {c_2}^\psi
\end{array}
\end{equation*}
\end{proof}

\begin{lemma}\label{lem 6.3}
Let $(\mathcal{D},C,\psi)$ be an entwining structure and $N$ be a right $C$-comodule. Then, for each $X\in Ob(\mathcal{D})$ we may obtain an object $N\otimes {\bf h}_X \in {\mathscr{M}(\psi)}_\mathcal{D}^C$ by setting
\begin{align}
(N\otimes {\bf h}_X)(Y): = N\otimes {\bf h}_X(Y)\\
(N\otimes {\bf h}_X)(f)(n \otimes g) :=n\otimes gf \label{rightaction2}
\end{align}
for $Y\in Ob(\mathcal{D})$, $f \in Hom_{\mathcal{D}}(Z,Y)$, $n\otimes g\in N\otimes {\bf h}_X(Y)$.  In fact, this determines a functor
from $Comod\text{-}C$ to $ {\mathscr{M}(\psi)}_\mathcal{D}^C$ . 
\end{lemma} 
\begin{proof}
By definition, it follows that $N\otimes {\bf h}_X$ is a right $\mathcal{D}$-module. Further, for each $Y \in Ob(\mathcal{D})$, we define a $K$-linear map $\sigma^r_{N\otimes {\bf h}_X(Y)}: N\otimes {\bf h}_X(Y)\longrightarrow N\otimes {\bf h}_X(Y) \otimes C$ as follows
\begin{align}\label{rightcomod2}
\sigma^r_{N\otimes {\bf h}_X(Y)}(n\otimes g):=  n_0\otimes g_\psi \otimes {n_1}^\psi
\end{align}
 We now verify that the map defined in \eqref{rightcomod2} makes $N\otimes {\bf h}_X(Y)$ a right $C$-comodule. We have
\begin{equation*}
\begin{array}{ll}
(\sigma^r \otimes id_C)\sigma^r(n \otimes g) =  (\sigma^r \otimes id_C)(n_0\otimes g_\psi \otimes {n_1}^\psi)
&=  n_0 \otimes {g_\psi}_\psi \otimes {n_1}^\psi \otimes {n_2}^\psi\\
&=   n_0 \otimes {g_\psi} \otimes \Delta_C({n_1}^\psi) \quad~~~~~~~~~~~~~ (\text{by}~ \eqref{eq 6.3})\\
&= (id_{N\otimes {\bf h}_X(Y)} \otimes \Delta_C)\sigma^r(n \otimes g) 
\end{array}
\end{equation*}
Moreover, using \eqref{eq 6.2} we have
\begin{equation*}
\begin{array}{ll}
(id_{N\otimes {\bf h}_X(Y)} \otimes \varepsilon_C)\sigma^r(n \otimes g)&=  (id_{N\otimes {\bf h}_X(Z)} \otimes \varepsilon_C)(n_0\otimes g_\psi \otimes {n_1}^\psi)\\
&=   n_0\otimes \varepsilon_C({n_1}^\psi)g_\psi =  n_0\otimes \varepsilon_C(n_1)g=n \otimes g
\end{array}
\end{equation*}
It remains to verify the condition in \eqref{comp 2}. We have
\begin{equation*}
\begin{array}{ll}
  \big((N\otimes {\bf h}_X)(f)(n \otimes g)\big)_0 \otimes \big((N\otimes {\bf h}_X)(f)(n \otimes g)\big)_1 &=   n_0 \otimes (gf)_\psi \otimes {n_1}^\psi\\
&=  n_0 \otimes g_\psi f_\psi \otimes {{n_1}^\psi}^\psi \quad (\text{by}~ \eqref{eq 6.1})\\
&=  (N\otimes {\bf h}_X)(f_\psi)(n_0 \otimes g_\psi) \otimes {{n_1}^\psi}^\psi\\
&=  (N\otimes {\bf h}_X)(f_\psi)({(n \otimes g)}_0) \otimes {{(n \otimes g)}_1}^\psi 
\end{array}
\end{equation*}
\end{proof}

It follows from Lemma \ref{lem 6.2} and Lemma \ref{lem 6.3} that both  ${\bf h}_Y\otimes C$ and $C\otimes {\bf h}_Y$ are objects in ${\mathscr{M}(\psi)}_\mathcal{D}^C$ for every $Y\in Ob(\mathcal{D})$.

\begin{lemma}\label{6.6}
Let $(\mathcal{D},C,\psi)$ be an entwining structure. Then, for each $Y\in Ob(\mathcal{D})$, we get a morphism $\Psi_Y: C\otimes {\bf h}_Y \longrightarrow {\bf h}_Y\otimes C$ in ${\mathscr{M}(\psi)}_\mathcal{D}^C$ given by $\Psi_Y(X):=\psi_{XY}$.
\end{lemma}
\begin{proof}
First we verify that $\Psi_Y$ is a morphism of right $\mathcal{D}$-modules. For any $f \in Hom_\mathcal{D}(X',X)$, $ g \in Hom_\mathcal{D}(X,Y)$ and $c\in C$, we have
\begin{align*}
\left(({\bf h}_Y\otimes C)(f)\right)\psi_{XY}(c \otimes g)&= ({\bf h}_Y\otimes C)(f)(g_\psi \otimes c^\psi)=  g_\psi f_\psi \otimes {c^\psi}^\psi=   {(gf)}_\psi \otimes c^\psi\\
&=\psi_{X'Y}(c \otimes gf)=\psi_{X'Y}(C\otimes {\bf h}_Y)(f)(c \otimes g)
\end{align*}
Next, we will show that $\Psi_Y(X)$ is $C$-colinear for every $X \in Ob(\mathcal{D})$. We have
\begin{equation*}
\begin{array}{lll}
  {\big(\psi_{XY}(c \otimes g)\big)}_0 \otimes {\big(\psi_{XY}(c \otimes g)\big)}_1 &=  g_\psi \otimes \Delta_C(c^\psi)&\\
&=  {g_\psi}_\psi \otimes {c_1}^\psi \otimes {c_2}^\psi& (\text{by}~ \eqref{eq 6.3})\\
&=  (\psi_{XY} \otimes id){(c \otimes g)}_0 \otimes {(c \otimes g)}_1 &(\text{by}~ \eqref{rightcomod2}) 
\end{array}
\end{equation*}
\end{proof}

We now recall from \cite[$\S$ 3]{Mit1} and \cite{Mit2} the notion of a finitely generated module over a category. Given $\mathcal M\in Mod$-$\mathcal D$, we set $el(\mathcal M):=\underset{X\in Ob(\mathcal D)}{\coprod}
\mathcal M(X)$ to be the collection of all elements of $\mathcal M$. Since $\mathcal D$ is small, we note that 
$el(\mathcal M)$ is a set. If $m\in el(\mathcal M)$ is such that $m\in \mathcal M(X)$, we will write
$|m|=X$. 
 
\begin{definition}
Let $\mathcal{D}$ be a small preadditive category and let $\mathcal{M}$ be a right $\mathcal D$-module. For each 
$m\in el(\mathcal M)$, we consider the corresponding morphism $\eta_m:{\bf h}_{|m|}\longrightarrow \mathcal M$.  A family of elements $\{m_i \in el(\mathcal M)\}_{i\in I}$ is said to be a generating set for $\mathcal{M}$ if the induced morphism 
$$ \eta: \bigoplus_{i\in I} {\bf h}_{|m_i|} \longrightarrow \mathcal{M}~~~~~~~~~~(0,...,0,id_{|m_i|},0,...,0)\mapsto m_i$$ is an epimorphism in $Mod\text{-}\mathcal D$.  In other words, every element $m\in el(\mathcal M)$  may be expressed as a sum $m=\sum_{i\in I} \mathcal{M}(f_i)(m_i)$, where each $f_i \in Hom_\mathcal{D}(|m|,|m_i|)$ and 
all but finitely many $\{f_i\}_{i\in I}$ are zero. 
\end{definition}

\begin{lemma}\label{gen2}
Let $(\mathcal{D},C,\psi)$ be an entwining structure and let $\mathcal{M}$ be an entwined module. We consider an element
$m\in el(\mathcal M)$.  Then, there exists a finite dimensional $C$-subcomodule $V_m$ of $\mathcal{M}(|m|)$ containing $m$ and a morphism $\eta_m: V_m \otimes {\bf h}_{|m|} \longrightarrow \mathcal{M}$ in  ${\mathscr{M}(\psi)}_\mathcal{D}^C$ such that $\eta_m(|m|)(m\otimes id_{|m|})=m$.
\end{lemma}
\begin{proof}
By \cite[Theorem 2.1.7]{SCS}, we know that there exists a finite dimensional $C$-subcomodule $V_m \subseteq \mathcal{M}(|m|)$ such that $m\in V_m$. 
Now, we consider the $\mathcal{D}$-module morphism $\eta_m: V_m \otimes {\bf h}_{|m|} \longrightarrow \mathcal{M}$ defined by setting $\eta_m(Y)(v\otimes f):= \mathcal{M}(f)(v)$ for any $Y\in Ob(\mathcal{D}),~ f\in Hom_\mathcal{D}(Y,|m|)$ and $v\in V_m$. We also have 
\begin{equation*}
\begin{array}{ll}
\rho_{\mathcal{M}(Y)}\big(\eta_m(Y)(v\otimes f)\big)=\rho_{\mathcal{M}(Y)}\big(\mathcal{M}(f)(v)\big)
&= \mathcal{M}(f_\psi)(v_0)\otimes v_1^\psi=  \eta_m(Y)( v_0 \otimes f_\psi)\otimes v_1^\psi\\
&=(\eta_m(Y) \otimes id_C)\big(\rho_{V_m \otimes {\bf h}_{|m|}(Y)}(v\otimes f)\big) \qquad (\text{by}~ \eqref{rightcomod2})
\end{array}
\end{equation*}
This shows that $\eta_m(Y)$ is $C$-colinear for each $Y\in Ob(\mathcal{D})$. Hence, $\eta_m$ is a morphism in  ${\mathscr{M}(\psi)}_\mathcal{D}^C$ such that $\eta_m(|m|)(m\otimes id_{|m|})=m$.
\end{proof}

\begin{prop}\label{entgroth}
Let $(\mathcal{D},C, \psi)$ be an entwining structure. Then, the category ${\mathscr{M}(\psi)}_\mathcal{D}^C$ of entwined modules  is a Grothendieck category.
\end{prop}
\begin{proof} Given a morphism $\eta:\mathcal M\longrightarrow \mathcal N$ in ${\mathscr{M}(\psi)}_\mathcal{D}^C$, let $Ker(\eta)$ and $Coker(\eta)$ be respectively the kernel and cokernel in $Mod\text{-}\mathcal D$. Since $Comod\text{-}C$
is an abelian category, we know that $Ker(\eta)(X)$, $Coker(\eta)(X)\in Comod\text{-}C$ for each $X\in Ob(\mathcal D)$. It is easily
seen that $Ker(\eta)$ and $Coker(\eta)$ satisfy the compatibility condition in \eqref{comp 2}, i.e., $Ker(\eta)$, $Coker(\eta)
\in {\mathscr{M}(\psi)}_\mathcal{D}^C$.  Since limits and colimits in ${\mathscr{M}(\psi)}_\mathcal{D}^C$ are obtained  from
those in $Mod\text{-}\mathcal D$ and $Comod\text{-}C$, it is clear that ${\mathscr{M}(\psi)}_\mathcal{D}^C$ is a cocomplete abelian category satisfying (AB5). 
 
\smallskip  By Lemma \ref{gen2}, there is an epimorphism  
$$\bigoplus_{m\in el(\mathcal{M})} \eta_m: \bigoplus_{m\in el(\mathcal{M})}  V_m \otimes {\bf h}_{|m|} \longrightarrow \mathcal{M}$$ for any $\mathcal M\in {\mathscr{M}(\psi)}_\mathcal{D}^C$.  As such, the collection $\{V\otimes {\bf h}_X\}$, where $X$ ranges over all objects in $\mathcal{D}$ and $V$ ranges over all (isomorphism classes of) finite dimensional $C$-comodules gives a set
of generators for ${\mathscr{M}(\psi)}_\mathcal{D}^C$ in the sense of \cite[Proposition 1.9.1]{Grothen}.
\end{proof}

\begin{corollary}
The category $\mathscr{M}_\mathcal{D}^C$  of Doi-Hopf modules  is a Grothendieck category.
\end{corollary}

\section{Separability and Frobenius conditions}\label{sepfrob}
Let $\mathscr{F}: {\mathscr{M}(\psi)}_\mathcal{D}^C \longrightarrow Mod$-$\mathcal{D}$ be the forgetful functor. The next result shows that the functor $\mathscr{F}$ has a right adjoint.

\begin{lemma}\label{rightadjoint}
The forgetful functor $\mathscr{F}: {\mathscr{M}(\psi)}_\mathcal{D}^C \longrightarrow Mod$-$\mathcal{D}$ has a right adjoint $\mathscr{G}: Mod\text{-}\mathcal{D} \longrightarrow {\mathscr{M}(\psi)}_\mathcal{D}^C$ given by $\mathscr{G}(\mathcal{N}):=  \mathcal{N} \otimes C$ for each $\mathcal{N}\in  Mod\text{-}\mathcal{D}$.
\end{lemma}
\begin{proof}
From Lemma \ref{lem 6.2}, we know that $\mathscr{G}(\mathcal{N})=  \mathcal{N} \otimes C\in {\mathscr{M}(\psi)}_\mathcal{D}^C$ 
for each $\mathcal N\in Mod\text{-}\mathcal D$. We define
$\alpha:Hom_{{\mathscr{M}(\psi)}_\mathcal{D}^C}\big(\mathcal{M},\mathscr{G}(\mathcal{N})\big) \longrightarrow Hom_{Mod\text{-}\mathcal{D}}(\mathscr{F}(\mathcal{M}),\mathcal{N})$ by setting 
$$\alpha(\xi)(X)(m):=(id_{\mathcal{N}(X)} \otimes \varepsilon_C)(\xi(X)(m)) \quad $$
for each  $\xi:\mathcal{M} \longrightarrow \mathcal{N} \otimes C$ in ${\mathscr{M}(\psi)}_\mathcal{D}^C$,  $X \in Ob(\mathcal{D})$ and $m \in \mathcal{M}(X)$.

\smallskip
 We also define
$\beta:Hom_{Mod\text{-}\mathcal{D}}(\mathscr{F}(\mathcal{M}),\mathcal{N}) \longrightarrow Hom_{{\mathscr{M}(\psi)}_\mathcal{D}^C}\big(\mathcal{M},\mathscr{G}(\mathcal{N})\big)$
by setting $$\beta(\eta)(X)(m):=  \eta(X)(m_0) \otimes m_1$$
for each $\eta:\mathcal{M} \longrightarrow \mathcal{N}$ in $Mod\text{-}\mathcal{D}$, 
$X \in Ob(\mathcal{D})$ and $m \in \mathcal{M}(X)$. First we check that $\alpha(\xi)$ and $\beta(\eta)$ are morphisms in $Mod\text{-}\mathcal{D}$ and ${\mathscr{M}(\psi)}_\mathcal{D}^C$ respectively. Using the fact that  $id_\mathcal{N} \otimes \varepsilon_C: \mathcal{N} \otimes C \longrightarrow \mathcal{N}$ and $\xi$ are right $\mathcal D$-module morphisms, for any $f \in Hom_\mathcal{D}(Y,X)$, we have
\begin{equation*}
\begin{array}{ll}
\mathcal{N}(f)\big(\alpha(\xi)(X)(m)\big)= \mathcal{N}(f)\Big((id_{\mathcal{N}(X)} \otimes \varepsilon_C)\big(\xi(X)(m)\big)\Big)
&=(id_{\mathcal{N}(Y)} \otimes \varepsilon_C)(\mathcal{N}(f) \otimes id_C)\big(\xi(X)(m)\big)\\
&=(id_{\mathcal{N}(Y)} \otimes \varepsilon_C)\big(\xi(Y)\mathcal{M}(f)(m)\big)\\
&= \alpha(\xi)(Y)(\mathcal{M}(f)(m))
\end{array}
\end{equation*}
We also have
\begin{equation*}
\begin{array}{lll}
(\mathcal{N} \otimes C)(f)\left(\beta(\eta)(X)(m)\right) &=  (\mathcal{N} \otimes C)(f)(\eta(X)(m_0) \otimes m_1) &\\
&= \mathcal{N}(f_\psi)\eta(X)(m_0) \otimes {m_1}^\psi& \mbox{(by Lemma \ref{lem 6.2})}\\
&= \eta(Y)\mathcal{M}(f_\psi)(m_0) \otimes {m_1}^\psi& \\
&=  \eta(Y)\left(\big({\mathcal{M}(f)(m)}\big)_0\right) \otimes \big({\mathcal{M}(f)(m)}\big)_1& \mbox{(by \eqref{comp 2})}\\
&= \beta(\eta)(Y)(\mathcal{M}(f)(m))&\\
\end{array}
\end{equation*}
Moreover, it is easy to see that $\beta(\eta)(X)$ is  $C$-colinear for each $X\in Ob(\mathcal D)$. We now verify that $\alpha$ and $\beta$ are inverses to each other.
\begin{equation*}
\begin{array}{lll}
\beta\big(\alpha(\xi)\big)(X)(m) &=  \alpha(\xi)(X)(m_0) \otimes m_1 & \\
&=  (id_{\mathcal{N}(X)} \otimes \varepsilon_C) \big(\xi(X)(m_0)\big) \otimes m_1&\\
&= (id_{\mathcal{N}(X)} \otimes \varepsilon_C \otimes id_C)(\xi(X) \otimes id_C)\rho_{\mathcal{M}(X)}(m)&\\
&= (id_{\mathcal{N}(X)} \otimes \varepsilon_C \otimes id_C) \pi^r_{\mathcal{N}(X) \otimes C}\big(\xi(X)(m)\big)& \mbox{($\xi(X)$ is $C$-colinear)}\\
&=\xi(X)(m)& \mbox{(by \eqref{rightcomod1})}\\
\end{array}
\end{equation*}
Further, we have $\alpha\big(\beta(\eta)\big)(X)(m) =  \eta(X)(m_0)\varepsilon_C(m_1)=\eta(X)(m)$. This 
proves the result.
\end{proof}

We now describe the unit $\mu:1_{{\mathscr{M}(\psi)}_\mathcal{D}^C}\longrightarrow\mathscr{GF}$ and the counit $\nu:\mathscr{FG}\longrightarrow 1_{Mod\text{-}\mathcal{D}}$ of the adjunction in Lemma \ref{rightadjoint}:

\begin{equation}\label{unit}
\mu(\mathcal{M}):\mathcal{M}\longrightarrow \mathcal{M}\otimes C \qquad \mu(\mathcal{M})(X)(m)= m_0\otimes m_1
\end{equation}
\begin{equation}\label{counit}
\nu(\mathcal{N})= id_\mathcal{N}\otimes \varepsilon_C:\mathcal{N}\otimes C\longrightarrow \mathcal{N}\qquad \nu(\mathcal{N})(X)(n\otimes c)=\varepsilon_C(c)n 
\end{equation}
for each $\mathcal{M}\in {\mathscr{M}(\psi)}_\mathcal{D}^C$, $\mathcal{N}\in Mod\text{-}\mathcal{D}$, $X \in Ob(\mathcal{D})$.

\smallskip
We recall that a functor $F:\mathcal{A}\longrightarrow \mathcal{B}$ between arbitrary categories is said to be separable if the natural transformation $$\mathcal{\eta}:Hom_\mathcal{A}({-},{-})\longrightarrow Hom_\mathcal{B}(F({-}),F({-}))$$
induced by $F$ is a split monomorphism (see \cite{NVV}, \cite[$\S$ 1]{Rf}). The following result provides a characterization of separable functors.

\begin{theorem}\cite[Theorem 1.2]{Rf}\label{separability}
Let $F:\mathcal{A}\longrightarrow \mathcal{B}$ be a functor which has a right adjoint $G:\mathcal{B}\longrightarrow \mathcal{A}$. Let $\mu$ and $\nu$ be the unit and counit of this adjunction respectively. Then, 
\begin{itemize}
\item[(i)] $F$ is separable if and only if there exists $\upsilon \in Nat(GF,1_\mathcal{A})$ such that $\upsilon\circ \mu= 1_\mathcal{A}$, the identity natural transformation on $\mathcal{A}$.
\item[(ii)] $G$ is separable if and only if there exists $\zeta \in Nat(1_\mathcal{B},FG)$ such that $\nu \circ \zeta=1_\mathcal{B}$, the identity natural transformation on $\mathcal{B}$.
\end{itemize}
\end{theorem}

\subsection{Separability conditions}
Let $(\mathcal{D},C,\psi)$ be an entwining structure.
We now investigate the separability of the forgetful functor $\mathscr{F}: {\mathscr{M}(\psi)}_\mathcal{D}^C \longrightarrow Mod$-$\mathcal{D}$. Since $\mathscr{F}$ has a right adjoint $\mathscr{G}$, it follows from Theorem \ref{separability} that the functor $\mathscr{F}$ is separable if and only if there exists a natural transformation $\upsilon:\mathscr{G}\mathscr{F}\longrightarrow 1_{{\mathscr{M}(\psi)}_\mathcal{D}^C}$ such that $\upsilon \circ \mu=1_{{\mathscr{M}(\psi)}_\mathcal{D}^C}$, where  $\mu$ is the unit of the adjunction as explained in \eqref{unit}. Throughout Section \ref{sepfrob}, $V:=Nat(\mathscr{G}\mathscr{F},1_{{\mathscr{M}(\psi)}_\mathcal{D}^C})$ will denote the $K$-space of all natural transformations from $\mathscr{G}\mathscr{F}$ to $1_{{\mathscr{M}(\psi)}_\mathcal{D}^C}$. We will shortly give another useful interpretation of $V$. We start by proving few preparatory results required for this.

\smallskip
We recall from Lemma \ref{lem 6.2} and Lemma \ref{lem 6.3} that both  ${\bf h}_Y\otimes C$ and $C\otimes {\bf h}_Y$ are objects in ${\mathscr{M}(\psi)}_\mathcal{D}^C$ for every $Y\in Ob(\mathcal{D})$. We define a functor ${\bf h} \otimes C:\mathcal{D}\longrightarrow {\mathscr{M}(\psi)}_\mathcal{D}^C$ as
\begin{gather}
({\bf h}\otimes C)(Y):={\bf h}_Y\otimes C\label{left1}\\
({\bf h}\otimes C)(f)(Z)(g\otimes c):= fg\otimes c\label{left action}
\end{gather}
for $f\in Hom_\mathcal{D}(Y,X),~g\in {\bf h}_Y(Z)$ and $c\in C$. Similarly,  we may also obtain a  functor ${\bf h}\otimes C \otimes C:\mathcal{D}\longrightarrow {\mathscr{M}(\psi)}_\mathcal{D}^C$. 
\begin{lemma}\label{reversing}
Let $f\in Hom_{\mathcal D}(Y,X)$. For any $\upsilon \in V$ and $c,d\in C$, we have
\begin{equation}\label{Lf6.16}
\left((id_{{\bf h}_X} \otimes \varepsilon_C)\upsilon({{\bf h}_X \otimes C})\right)(Y)(f \otimes c \otimes d)=f\circ\left(( \varepsilon_C \otimes id_{{\bf h}_Y})\upsilon({ C\otimes {\bf h}_Y })\right)(Y)(c\otimes id_Y\otimes d)
\end{equation} In particular, we have
\begin{equation}\label{Lg6.16}
\left((id_{{\bf h}_X} \otimes \varepsilon_C)\upsilon({\bf h}_X \otimes C)\right)(X)(id_X \otimes c\otimes d)=\left((\varepsilon_C \otimes 
id_{{\bf h}_X})\upsilon(C \otimes {\bf h}_X)\right)(X)(c \otimes id_X \otimes d)
\end{equation}
\end{lemma}
\begin{proof}
A morphism $f:Y\longrightarrow X$ in $\mathcal{D}$ induces morphisms ${\bf h}_Y\otimes C\longrightarrow {\bf h}_X\otimes C$ and ${\bf h}_Y\otimes C\otimes C\longrightarrow {\bf h}_X\otimes C\otimes C$ in ${\mathscr{M}(\psi)}_\mathcal{D}^C$ as explained in \eqref{left action}. Since $\upsilon:\mathscr{G}\mathscr{F}\longrightarrow 1_{{\mathscr{M}(\psi)}_\mathcal{D}^C}$ is a natural transformation, it follows that the following diagram commutes:
\[
\xymatrix{
{\bf h}_Y(Y)\otimes C\otimes C \ar[d]_{\upsilon({{\bf h}_Y \otimes C})(Y)} \ar[r]^{f} &{\bf h}_X(Y)\otimes C\otimes C\ar[d]^{\upsilon({{\bf h}_X \otimes C})(Y)}\\
{\bf h}_Y(Y)\otimes C \ar[d]_{(id_{{\bf h}_Y} \otimes \varepsilon_C)(Y)}\ar[r]^{f} &{\bf h}_X(Y)\otimes C\ar[d]^{(id_{{\bf h}_X} \otimes \varepsilon_C)(Y)}\\
{\bf h}_Y(Y)\ar[r]^{f}&{\bf h}_X(Y)}
\]
Thus, we have 
\begin{equation}\label{eq 6.12}
f\circ\left((id_{{\bf h}_Y} \otimes \varepsilon_C)\upsilon({{\bf h}_Y \otimes C})\right)(Y)(id_Y \otimes c \otimes d)=\left((id_{{\bf h}_X} \otimes \varepsilon_C)\upsilon({{\bf h}_X \otimes C})\right)(Y)(f \otimes c \otimes d)
\end{equation}

We now consider the morphism $\Psi_Y: C\otimes {\bf h}_Y \longrightarrow {\bf h}_Y\otimes C$ in ${\mathscr{M}(\psi)}_\mathcal{D}^C$ given by $\Psi_Y(X):=\psi_{XY}$ as in Lemma \ref{6.6}. Then, using the naturality of $\upsilon:\mathscr{G}\mathscr{F}\longrightarrow 1_{{\mathscr{M}(\psi)}_\mathcal{D}^C}$ and \eqref{eq 6.2} we have the following commutative diagram
$$\begin{CD}
C\otimes {\bf h}_Y(Y)\otimes C  @>\psi_{YY}\otimes id_C>>  {\bf h}_Y(Y)\otimes C\otimes C\\
@V\upsilon({C \otimes {\bf h}_Y})(Y)VV        @VV \upsilon({{\bf h}_Y \otimes C})(Y)V\\
C \otimes {\bf h}_Y(Y)     @>\psi_{YY} >> {\bf h}_Y(Y) \otimes C\\
@V(\varepsilon_C \otimes id_{{\bf h}_Y})(Y)VV  @VV(id_{{\bf h}_Y} \otimes \varepsilon_C)(Y)V\\
{\bf h}_Y(Y) @> id_{{\bf h}_Y}(Y)>> {\bf h}_Y(Y)
\end{CD}$$
Using the fact that $\psi_{YY}(c \otimes id_Y)= id_Y \otimes c$, we now have 
\begin{equation}\label{eq 6.13}
\left((id_{{\bf h}_Y}\otimes \varepsilon_C)\upsilon({{\bf h}_Y \otimes C})\right)(Y)(id_Y \otimes c \otimes d)= \left(( \varepsilon_C \otimes id_{{\bf h}_Y})\upsilon({ C\otimes {\bf h}_Y })\right)(Y)(c\otimes id_Y\otimes d)
\end{equation}
Combining \eqref{eq 6.12} and \eqref{eq 6.13}, we have 
\begin{equation}\label{6.16}
\left((id_{{\bf h}_X} \otimes \varepsilon_C)\upsilon({{\bf h}_X \otimes C})\right)(Y)(f \otimes c \otimes d)=f\circ\left(( \varepsilon_C \otimes id_{{\bf h}_Y})\upsilon({ C\otimes {\bf h}_Y })\right)(Y)(c\otimes id_Y\otimes d)
\end{equation}
By putting $Y=X$ and taking $f=id_X$, the result of \eqref{Lg6.16} is clear from \eqref{6.16}.
\end{proof}

\begin{lemma}\label{2.4m}
For any $\upsilon \in V$ and $Y \in Ob(\mathcal{D})$, we have $\upsilon(C \otimes C \otimes {\bf h}_Y)=id_C \otimes \upsilon(C \otimes {\bf h}_Y)$ as a morphism 
of $\mathcal D$-modules.
\end{lemma}
\begin{proof}
For each $d \in C$, we define $\eta_d:C \otimes {\bf h}_Y \longrightarrow C \otimes C \otimes {\bf h}_Y$ by
\begin{equation*}
\eta_d(X)(c \otimes g):= d \otimes c \otimes g
\end{equation*}
for each $X \in Ob(\mathcal{D})$, $g \in {\bf h}_Y(X)$ and $c \in C$. It may be easily verified that $\eta_d$ is a morphism of right $\mathcal{D}$-modules. We now verify that $\eta_d(X):C \otimes {\bf h}_Y(X) \longrightarrow C \otimes C \otimes {\bf h}_Y(X)$ is right $C$-colinear. We have
\begin{equation*}
\begin{array}{ll}
\sigma^r_{C \otimes C \otimes {\bf h}_Y(X)}\left(\eta_d(X)(c \otimes g)\right)&=\sigma^r_{C \otimes C \otimes {\bf h}_Y(X)}(d \otimes c \otimes g)= (d \otimes c)_0 \otimes g_\psi \otimes {(d \otimes c)_1}^\psi\\
&= d \otimes c_1 \otimes g_\psi \otimes {c_2}^\psi=(\eta_d(X) \otimes id_C)(c_1 \otimes g_\psi \otimes {c_2}^\psi)\\
&=(\eta_d(X) \otimes id_C)\sigma^r_{C \otimes {\bf h}_Y(X)}(c \otimes g)
\end{array}
\end{equation*} 
Thus, $\eta_d:C \otimes {\bf h}_Y \longrightarrow C \otimes C \otimes {\bf h}_Y$ is a morphism in ${\mathscr{M}(\psi)}_\mathcal{D}^C$. Therefore, using the naturality of $\upsilon$, we have the following commutative diagram:
$$\begin{CD}
C \otimes {\bf h}_Y(X) \otimes C    @>\upsilon(C \otimes {\bf h}_Y)(X)>> C \otimes {\bf h}_Y(X)\\
@V\eta_d(X) \otimes id_CVV        @VV\eta_d(X)V\\
C \otimes C \otimes {\bf h}_Y(X) \otimes C    @> \upsilon(C \otimes C \otimes {\bf h}_Y)(X)>>  C\otimes C \otimes {\bf h}_Y(X)
\end{CD}$$
Thus, for any $g \in Hom_\mathcal{D}(X,Y)$ and $c,c' \in C$, we get
\begin{equation}
\begin{array}{ll}
\upsilon(C \otimes C \otimes {\bf h}_Y)(X)(d \otimes c \otimes g \otimes c')&=\left(\upsilon(C \otimes C \otimes {\bf h}_Y)(\eta_d \otimes id_C)\right)(X)(c \otimes g \otimes c')\\
&=\left(\eta_d \circ \upsilon(C \otimes {\bf h}_Y)\right)(X)(c \otimes g \otimes c')\\
&= d \otimes \upsilon(C \otimes {\bf h}_Y)(X)(c \otimes g \otimes c')\\
&= \left(id_C \otimes \upsilon(C \otimes {\bf h}_Y)\right)(X)(d \otimes c \otimes g \otimes c')
\end{array}
\end{equation}
The result follows.
\end{proof}

\smallskip
We now proceed to give another interpretation of $V=Nat(\mathscr{G}\mathscr{F},1_{{\mathscr{M}(\psi)}_\mathcal{D}^C})$. We consider a collection $\theta:=\{\theta_{X}:  C \otimes C \longrightarrow  End_\mathcal{D}(X)\}_{X \in Ob(\mathcal{D})}$ of $K$-linear maps  satisfying the following conditions:
\begin{gather}
\left( \theta_{X}(c \otimes d)\right)\circ f= {f_\psi}_\psi\circ \theta_{Y}\left(c^\psi \otimes d^\psi \right) \label{theta1}\\ 
\theta_{X}(c \otimes d_1) \otimes d_2 = {\left(\theta_{X}(c_2 \otimes d)\right)}_\psi \otimes {c_1}^\psi \label{theta2}
\end{gather}
for any $f \in Hom_{\mathcal{D}}(Y,X)$. Let $V_1$ be the $K$-space consisting of all such $\theta$.  

\begin{prop}\label{alpha}
Let $\upsilon \in V=Nat(\mathscr{G}\mathscr{F},1_{{\mathscr{M}(\psi)}_\mathcal{D}^C})$. For each $X\in Ob(\mathcal{D})$, we define a $K$-linear map
$$\theta_{X}:C\otimes C\longrightarrow End_{\mathcal D}(X)\qquad c\otimes d\mapsto \big((id_{{\bf h}_X} \otimes \varepsilon_C)\upsilon({{\bf h}_X \otimes C})\big)(X)(id_X\otimes c\otimes d)$$
Then, $\theta:=\{\theta_{X}\}_{X\in Ob(\mathcal{D})}$ is an element in $V_1.$ 
\end{prop}
\begin{proof}
Since $id_{{\bf h}_X} \otimes \varepsilon_C: {\bf h}_X\otimes C\longrightarrow {\bf h}_X$ is a morphism of right $\mathcal{D}$-modules, we have 
\begin{equation}\label{6.12}
\big( \theta_{X}(c \otimes d)\big)\circ f =(id_{{\bf h}_X} \otimes \varepsilon_C)(Y)\Big(\upsilon({{\bf h}_X \otimes C})(X)(id_X \otimes c\otimes d)\cdot f\Big)
\end{equation}
for $f\in Hom_\mathcal{D}(Y,X)$ and $c,d \in C$. Since $\upsilon({{\bf h}_X \otimes C}): {\bf h}_X\otimes C\otimes C\longrightarrow {\bf h}_X\otimes C$  is a morphism of right $\mathcal{D}$-modules, we also have 
\begin{equation}\label{6.13}
\begin{array}{ll}
\left(\upsilon({{\bf h}_X \otimes C})(X)(id_X\otimes c\otimes d)\right)\cdot f &=\upsilon({{\bf h}_X \otimes C})(Y)\big((id_X\otimes c\otimes d)\cdot f\big)\\
&=\upsilon({{\bf h}_X \otimes C})(Y)\big(({\bf h}_X \otimes C \otimes C)(f)(id_X\otimes c\otimes d)\big)\\
&=\upsilon({{\bf h}_X \otimes C})(Y)\big(({\bf h}_X \otimes C)(f_\psi)(id_X\otimes c) \otimes d^\psi\big)\\
&=\upsilon({{\bf h}_X \otimes C})(Y)\big({\bf h}_X({f_\psi}_\psi)(id_X) \otimes c^\psi \otimes d^\psi\big)\\
&=\upsilon({{\bf h}_X \otimes C})(Y)({f_\psi}_\psi\otimes c^\psi\otimes d^\psi)
\end{array}
\end{equation}
The morphism ${f_\psi}_\psi:Y \longrightarrow X$ in $\mathcal{D}$ induces morphisms ${\bf h}_Y\otimes C\longrightarrow {\bf h}_X\otimes C$ and ${\bf h}_Y\otimes C\otimes C\longrightarrow {\bf h}_X\otimes C\otimes C$ in ${\mathscr{M}(\psi)}_\mathcal{D}^C$. Therefore, we have
\begin{equation*}
\begin{array}{lll}
\left(\theta_{X}(c \otimes d)\right)\circ f&=(id_{{\bf h}_X} \otimes \varepsilon_C)(Y)\Big(\upsilon({{\bf h}_X \otimes C})(X)(id_X \otimes c\otimes d)\cdot f\Big)&(\text{by }\eqref{6.12})\\
&=\left((id_{{\bf h}_X} \otimes \varepsilon_C)\upsilon({{\bf h}_X \otimes C})\right)(Y)({f_\psi}_\psi \otimes c^\psi \otimes d^\psi)&(\text{by }\eqref{6.13})\\
&={f_\psi}_\psi\circ\left((id_{{\bf h}_Y} \otimes \varepsilon_C)\upsilon({{\bf h}_Y \otimes C})\right)(Y)(id_Y \otimes c^\psi \otimes d^\psi)&(\text{by }  \eqref{eq 6.12})\\
&={f_\psi}_\psi \circ \theta_{Y}(c^\psi\otimes d^\psi)&
\end{array}
\end{equation*}

\smallskip
This proves \eqref{theta1}. We now verify  that $\theta$ satisfies \eqref{theta2}. Using Lemma \ref{lem 6.3}, we know that $C \otimes {\bf h}_Y$ and $C \otimes C \otimes {\bf h}_Y$ belong to ${\mathscr{M}(\psi)}_\mathcal{D}^C$ for each $Y \in Ob(\mathcal{D})$. For each $X \in Ob(\mathcal{D})$, it may be easily seen that $C \otimes {\bf h}_Y(X)$ is also a left $C$-comodule with coaction given by $\rho_Y^l(X):=\Delta_C \otimes id_{{\bf h}_Y(X)}$. Moreover, it may be easily verified that the following diagram commutes:
$$\begin{CD}
C \otimes {\bf h}_Y(X)   @>\rho_Y^l(X)>> C\otimes C \otimes {\bf h}_Y(X) \\
@V\sigma_{C \otimes {\bf h}_Y(X)}^rVV        @VV\sigma_{C \otimes C \otimes {\bf h}_Y(X)}^rV\\
C \otimes {\bf h}_Y(X) \otimes C  @>\rho_Y^l(X) \otimes id>>  C\otimes C \otimes {\bf h}_Y(X) \otimes C
\end{CD}$$
This shows that $\rho_Y^l(X)$ is a morphism of right $C$-comodules. Further, for any $g \in Hom_\mathcal{D}(X,X')$, we have the following commutative diagram:
$$\begin{CD}
C \otimes {\bf h}_Y(X')   @>\rho_Y^l(X')>> C\otimes C \otimes {\bf h}_Y(X') \\
@V(C \otimes {\bf h}_Y)(g)VV        @VV(C\otimes C \otimes {\bf h}_Y)(g)V\\
C \otimes {\bf h}_Y(X)    @>\rho_Y^l(X)>>  C\otimes C \otimes {\bf h}_Y(X)
\end{CD}$$

Thus, $\rho_Y^l:C\otimes {\bf h}_Y\longrightarrow C\otimes C\otimes {\bf h}_Y $ is a morphism of right $\mathcal{D}$-modules. This shows that $\rho_Y^l$ is a morphism in the category  ${\mathscr{M}(\psi)}_\mathcal{D}^C$. Therefore, using the naturality of $\upsilon$ and Lemma \ref{2.4m}, we have the following commutative diagram:
$$\begin{CD}
C \otimes {\bf h}_Y(X) \otimes C    @>\upsilon(C \otimes {\bf h}_Y)(X)>> C \otimes {\bf h}_Y(X)\\
@V\rho_Y^l(X) \otimes id_CVV        @VV\rho_Y^l(X)V\\
C \otimes C \otimes {\bf h}_Y(X) \otimes C    @>\upsilon(C \otimes C \otimes {\bf h}_Y)(X)=id_C \otimes \upsilon(C \otimes {\bf h}_Y)(X)>>  C\otimes C \otimes {\bf h}_Y(X)
\end{CD}$$

For any $c \otimes id_X \otimes d \in C \otimes {\bf h}_X(X) \otimes C$, we set $ a_i \otimes f_i:=\upsilon\big(C \otimes {\bf h}_X\big)(X)(c \otimes id_X \otimes d)$. Then, we have
\begin{align*}
\rho_X^l(X)\left(a_i \otimes f_i\right)= {a_i}_1 \otimes {a_i}_2 \otimes f_i=  \left(id_C \otimes \upsilon (C \otimes {\bf h}_X)(X)\right)(c_1 \otimes c_2 \otimes id_X \otimes d)
\end{align*}
Now applying the map $id_C \otimes \varepsilon_C \otimes id_{{\bf h}_X}$ to both sides, we get
\begin{equation*}
\begin{array}{ll}
a_i \otimes f_i&= c_1 \otimes \big((\varepsilon_C \otimes id_{{\bf h}_X}) \upsilon(C \otimes {\bf h}_X)\big)(X)(c_2 \otimes id_X \otimes d)\\
&=c_1 \otimes \big((id_{{\bf h}_X} \otimes \varepsilon_C) \upsilon ({\bf h}_X \otimes C)\big)(X)(id_X \otimes c_2 \otimes d)~~~~~~(\text{by Lemma}~\ref{reversing})\\
&= c_1 \otimes \theta_X(c_2 \otimes d)
\end{array}
\end{equation*}
Therefore, we have
\begin{equation}\label{eq 6.14}
\psi\big( a_i \otimes f_i\big)= {\big(\theta_X(c_2 \otimes d)\big)}_\psi \otimes {c_1}^\psi
\end{equation}
Since $\upsilon(C \otimes {\bf h}_Y)(X)$ is a morphism of right $C$-comodules, we also have the following commutative diagram:
$$\begin{CD}
C \otimes {\bf h}_Y(X) \otimes C     @>\upsilon(C \otimes {\bf h}_Y)(X)>>  C \otimes {\bf h}_Y(X) \\
@V\pi^r_{C \otimes {\bf h}_Y(X) \otimes C}VV       @VV\sigma^r_{C \otimes {\bf h}_Y(X)}V\\
C \otimes {\bf h}_Y(X) \otimes C \otimes C     @>\upsilon(C \otimes {\bf h}_Y)(X) \otimes id_C>>  C\otimes {\bf h}_Y(X) \otimes C
\end{CD}$$
Thus, we have
\begin{align*}
\sigma^r_{C \otimes {\bf h}_X(X)}\big(a_i \otimes f_i\big)={a_i}_1 \otimes {f_i}_\psi \otimes {{a_i}_2}^\psi =\big(\upsilon(C \otimes {\bf h}_X)(X) \otimes id_C\big)\big( c \otimes id_X \otimes d_1 \otimes d_2\big)
\end{align*}
Now, applying the map $\varepsilon_C \otimes id_{{\bf h}_X} \otimes id_C$ to both sides, we get
\begin{equation*}
\begin{array}{ll}
  \varepsilon_C({a_i}_1)({f_i}_\psi \otimes {a_i}_2^\psi) &=  \big((\varepsilon_C \otimes id_{{\bf h}_X})\upsilon(C \otimes {\bf h}_X)\big)(X)\ (c \otimes id_X \otimes d_1) \otimes d_2\\
&=  \big((id_{{\bf h}_X} \otimes \varepsilon_C)\upsilon({\bf h}_X \otimes C)\big)(X)\ (id_X \otimes c \otimes d_1) \otimes d_2~~~~~~(\text{by Lemma}~\ref{reversing})
\end{array}
\end{equation*}
Therefore, 
\begin{equation}\label{eq 6.15}
\psi\big( a_i \otimes f_i\big)=\theta_X (c \otimes d_1) \otimes d_2
\end{equation}
It now follows from \eqref{eq 6.14} and \eqref{eq 6.15} that $\theta$ satisfies \eqref{theta2}.
\end{proof}

\begin{prop}\label{beta}
Let $\theta \in V_1$. Then, we have an element $\upsilon \in Nat(\mathscr{G}\mathscr{F},1_{{\mathscr{M}(\psi)}_\mathcal{D}^C})$ defined by
\begin{equation*}
\upsilon(\mathcal{M}): \mathcal{M}\otimes C\longrightarrow \mathcal{M}, \qquad m \otimes c \mapsto \mathcal{M}\big(\theta_{X}( m_1 \otimes c)\big)(m_0)
\end{equation*}
for $\mathcal{M} \in Ob\big({\mathscr{M}(\psi)}_\mathcal{D}^C \big),~ X\in Ob(\mathcal{D}),~ m\in \mathcal{M}(X)$ and $c\in C$. 
\end{prop}
\begin{proof}
We need to verify that $\upsilon(\mathcal{M}):\mathcal{M}\otimes C\longrightarrow \mathcal{M}$ is a morphism in ${\mathscr{M}(\psi)}_\mathcal{D}^C$ and that $\upsilon$ is indeed a natural transformation. We first verify that $\upsilon(\mathcal{M})$ is a morphism of right $\mathcal{D}$-modules. Let $f \in Hom_\mathcal{D}(Y,X)$. Then, we have
\begin{equation*}
\begin{array}{lll}
\mathcal{M}(f)\big(\upsilon(\mathcal{M})(X)\big)(m \otimes c)&= \mathcal{M}(f) \mathcal{M}\big(\theta_{X}(m_1 \otimes c)\big)(m_0)&\\
&=  \mathcal{M}\left(\big(\theta_{X}(m_1 \otimes c)\big) \circ f\right)(m_0)&\\
&= \mathcal{M}\big({f_\psi}_\psi\circ \theta_{Y}(m_1^\psi \otimes c^\psi) \big)(m_0)& (\text{by}~ \eqref{theta1})\\
&=  \mathcal{M}\big(\theta_{Y}(m_1^\psi \otimes c^\psi)\big) \mathcal{M}({f_\psi}_\psi)(m_0)&\\
&=  \mathcal{M}\Big(\theta_{Y}\big((\mathcal{M}(f_\psi)(m))_1 \otimes c^\psi\big)\Big)\big(\mathcal{M}(f_\psi)(m)\big)_0 &(\text{by}~ \eqref{comp 2})\\
&= \upsilon(\mathcal{M})(Y)\big( \mathcal{M}(f_\psi)(m) \otimes c^\psi \big)&\\
&= \upsilon(\mathcal{M})(Y) (\mathcal{M} \otimes C)(f)(m \otimes c)&
\end{array}
\end{equation*}
We now verify that $\upsilon(\mathcal{M})(X):\mathcal{M}(X) \otimes C\longrightarrow \mathcal{M}(X)$ is a morphism of right $C$-comodules for every $X \in Ob(\mathcal{D})$. For each $m \otimes c \in \mathcal{M}(X) \otimes C$, we have
\begin{equation*}
\begin{array}{lll}
\big(\upsilon(\mathcal{M})(X) \otimes id_C \big)\pi^r(m \otimes c)&= \upsilon(\mathcal{M})(X)(m \otimes c_1) \otimes c_2&\\
&= \mathcal{M}\big(\theta_{X}(m_1 \otimes c_1)\big)(m_0) \otimes c_2&\\
&=\mathcal{M}\left(\left(\theta_{X}((m_{1})_2 \otimes c)\right)_\psi \right)(m_0) \otimes {(m_{1})_1}^\psi&(\text{by}~ \eqref{theta2})\\
&= \mathcal{M}\left(\left(\theta_{X}(m_{1} \otimes c)\right)_\psi\right)(m_0)_0 \otimes {(m_{0})_1}^\psi&\\
&= \rho_{\mathcal M(X)}\left(\mathcal{M}\left(\theta_{X}(m_{1} \otimes c)\right)(m_0)\right)&(\text{by}~ \eqref{comp 2})\\
&= \rho_{\mathcal M(X)}\big( \upsilon(\mathcal{M})(X) (m \otimes c)\big)&
\end{array}
\end{equation*}
It remains to show that $\upsilon:\mathscr{G}\mathscr{F} \longrightarrow 1_{{\mathscr{M}(\psi)}_\mathcal{D}^C}$ is a natural transformation. Let $\eta:\mathcal{M} \longrightarrow \mathcal{N}$ be a morphism in ${\mathscr{M}(\psi)}_\mathcal{D}^C$. Then, for every $X \in Ob(\mathcal{D})$ and $m \otimes c \in \mathcal{M}(X) \otimes C$, we have
\begin{equation*}
\begin{array}{lll}
\big(\upsilon(\mathcal{N})(\eta \otimes id_C)\big)(X)(m \otimes c)&=\upsilon(\mathcal{N})(X)\big(\eta(X)(m) \otimes c\big)&\\
&=  \mathcal{N}\left(\theta_{X}\left((\eta(X)(m))_1 \otimes c \right)\right)\left(\eta(X)(m)\right)_0&\\
&=  \mathcal{N}\left(\theta_{X}\left(m_1 \otimes c \right)\right)\eta(X)(m_0)&\text{(since $\eta(X)$ is $C$-colinear)}\\
&= \eta(X)\mathcal{M}\left(\theta_{X}\left(m_1 \otimes c \right)\right)(m_0)&\\
&= \eta(X)\upsilon(\mathcal{M})(X)(m \otimes c)&
\end{array}
\end{equation*}
This proves the result.
\end{proof}

\begin{prop}\label{V=V1}
The $K$-spaces $V=Nat(\mathscr{G}\mathscr{F},1_{{\mathscr{M}(\psi)}_\mathcal{D}^C})$ and $V_1$ are isomorphic.
\end{prop}
\begin{proof}
We define $\alpha:V \longrightarrow V_1$  by setting $\alpha(\upsilon)=\theta$, where $\theta$ is the collection of  $K$-linear maps  $\{\theta_X:C \otimes C \longrightarrow End_\mathcal{D}(X)\}_{X\in Ob(\mathcal D)}$ defined by
$$\theta_{X}(c\otimes d):=\left((id_{{\bf h}_X} \otimes \varepsilon_C)\upsilon({{\bf h}_X \otimes C})\right)(X)(id_X\otimes c\otimes d)$$
for $c,d \in C$. Then, $\alpha$ is a well-defined map  by Proposition \ref{alpha}. We also define 
$\beta:V_1 \longrightarrow V$ by setting $ \beta(\theta)=\upsilon$, where $\upsilon:\mathscr{G}\mathscr{F} \longrightarrow 1_{{\mathscr{M}(\psi)}_\mathcal{D}^C}$ is defined by
\begin{equation}\label{defupsilon}
\upsilon(\mathcal{M}): \mathcal{M}\otimes C\longrightarrow \mathcal{M}, \quad m \otimes c \mapsto \mathcal{M}\big(\theta_{X}( m_1 \otimes c)\big)(m_0)
\end{equation}
for $\mathcal{M} \in Ob\big({\mathscr{M}(\psi)}_\mathcal{D}^C \big),~ X\in Ob(\mathcal{D}),~ m \otimes c \in \mathcal{M}(X) \otimes C$. By Proposition \ref{beta}, $\beta$ is well-defined. We will now verify that $\alpha$ and $\beta$ are inverses of each other. Let $\theta \in V_1$. Then, for any $X, Y \in Ob(\mathcal{D})$, we have
\begin{equation*}
\begin{array}{lll}
\left(\alpha\beta(\theta)\right)_{X}(c \otimes d)&=(id_{{\bf h}_X} \otimes \varepsilon_C)(X)\left(\beta(\theta)({{\bf h}_X \otimes C})(X)(id_X \otimes c \otimes d)\right)&\\
&=  (id_{{\bf h}_X} \otimes \varepsilon_C)(X) ({\bf h}_X \otimes C)\left(\theta_{X}\left( (id_X  \otimes c)_1 \otimes d\right)\right)(id_X \otimes c)_0&\\
&= (id_{{\bf h}_X} \otimes \varepsilon_C)(X) ({\bf h}_X \otimes C)\left(\theta_{X}(c_2 \otimes d)\right)(id_X \otimes c_1)&\\
&= (id_{{\bf h}_X} \otimes \varepsilon_C)(X)\left({\bf h}_X\left(\left(\theta_{X}(c_2 \otimes d)\right)_\psi\right)(id_X) \otimes {c_1}^\psi\right)&(\text{by}~\mbox{ Lemma \ref{rightaction1}})\\
&= (id_{{\bf h}_X} \otimes \varepsilon_C)(X)\left({\bf h}_X\left(\theta_{X}(c \otimes d_1)\right)(id_X) \otimes d_2\right)&(\text{by}~ \eqref{theta2})\\
&= \left(\theta_{X}(c \otimes d_1)\right) \varepsilon_C(d_2)= \theta_{X}(c \otimes d)&
\end{array}
\end{equation*} 
This proves that $\big(\alpha\beta(\theta)\big)_{X}=\theta_{X}$ for all $X \in  Ob(\mathcal{D})$. Therefore,
$(\alpha\beta)(\theta)=\theta$. For any $\upsilon \in V$, we now verify that $(\beta \alpha)(\upsilon)=\upsilon$. We set $\theta= \alpha(\upsilon)$. Then, by definition we have
\begin{equation}\label{6.24}
\begin{array}{ll}
(\beta\alpha)(\upsilon)(\mathcal{M})(X)(m\otimes c)&=\left(\left(\beta(\theta)\right)(\mathcal{M})\right)(X)(m \otimes c)\\
&=\mathcal{M}\big(\theta_{X}(m_1 \otimes c)\big)(m_0)\\
&=\mathcal{M}\big((id_{{\bf h}_X} \otimes \varepsilon_C)(X)\upsilon({{\bf h}_X \otimes C})(X)(id_X \otimes m_1 \otimes c)\big)(m_0)
\end{array}
\end{equation}
For any $m'\in \mathcal M(X)$, it may be easily verified that $\eta_{m'}:{\bf h}_X\longrightarrow\mathcal{M}$ defined by $\eta_{m'}(Y)(f):=\mathcal{M}(f)(m')$ for each $f\in Hom_\mathcal{D}(Y,X)$  is a morphism in $Mod\text{-}\mathcal{D}$. By Lemma \ref{rightaction1}, this induces the morphism $\eta_{m'}\otimes id_C:{\bf h}_X\otimes C\longrightarrow\mathcal{M}\otimes C$ in  $\mathscr{M}(\psi)_\mathcal{D}^C$ defined by $(\eta_{m'}\otimes id_C)(Y)(f\otimes c):=\mathcal{M}(f)(m')\otimes c$ for $f\in Hom_\mathcal{D}(Y,X)$ and $c\in C$. Since $\upsilon$ is a natural transformation, it follows easily that the following diagram commutes
$$\begin{CD}
{\bf h}_X(X)\otimes C\otimes C    @>\eta_{m'}(X)\otimes id_C\otimes id_C>>  \mathcal{M}(X)\otimes C\otimes C\\
@V\upsilon({{\bf h}_X \otimes C})(X)VV        @VV \upsilon({\mathcal{M} \otimes C})(X)V\\
{\bf h}_X(X)\otimes C     @>\eta_{m'}(X)\otimes id_C>> \mathcal{M}(X)\otimes C\\
@V (id_{{\bf h}_X} \otimes \varepsilon_C)(X)VV  @VV (id_{\mathcal{M}} \otimes \varepsilon_C)(X)V\\
{\bf h}_X(X) @> \eta_{m'}(X)>>\mathcal{M}(X)
\end{CD}$$
In particular, we have 
\begin{equation}\label{6.25}
\mbox{\small
$\mathcal{M}\left(\left((id_{{\bf h}_X} \otimes \varepsilon_C)\upsilon({{\bf h}_X \otimes C})\right)(X)( id_X\otimes m_1 \otimes c)\right)(m_0) =\left((id_{\mathcal{M}} \otimes \varepsilon_C)\upsilon(\mathcal{M}\otimes C)\right)(X)( m_0\otimes m_1 \otimes c)$}
\end{equation}

The comodule structure on entwined modules determines a morphism in ${\mathscr{M}(\psi)}_\mathcal{D}^C$ as follows. We define $\tilde{\rho}:\mathcal{M} \longrightarrow \mathcal{M} \otimes C$ given by 
\begin{equation*}
\tilde{\rho}(X):=\rho_{\mathcal{M}(X)}:\mathcal{M}(X) \longrightarrow \mathcal{M}(X) \otimes C
\end{equation*}
for any $\mathcal{M} \in Ob({\mathscr{M}(\psi)}_\mathcal{D}^C)$ and $X \in Ob(\mathcal{D})$.  We first verify that $\tilde{\rho}$ is a morphism of right $\mathcal{D}$-modules. For any $f \in Hom_\mathcal{D}(Y,X)$ and $m \in \mathcal{M}(X)$, we have
\begin{equation*}
\begin{array}{lll}
(\mathcal{M} \otimes C)(f)\left(\tilde{\rho}(X)(m)\right)&=(\mathcal{M} \otimes C)(f)(m_0 \otimes m_1)=\mathcal{M}(f_\psi)(m_0) \otimes {m_1}^\psi&\\
&=\rho_{\mathcal{M}(Y)}(\mathcal{M}(f)(m))=\tilde{\rho}(Y)(\mathcal{M}(f)(m))& (\text{by}~ \eqref{comp 2})
\end{array}
\end{equation*}
It may be verified easily that $\tilde{\rho}(X):\mathcal{M}(X) \longrightarrow \mathcal{M}(X) \otimes C$ is right $C$-colinear. Thus, $\tilde{\rho}:\mathcal{M} \longrightarrow \mathcal{M} \otimes C$ is a morphism in ${\mathscr{M}(\psi)}_\mathcal{D}^C$. Therefore, we have the following commutative diagram
$$\begin{CD}
\mathcal{M}(X) \otimes C    @>\upsilon(\mathcal{M})(X)>> \mathcal{M}(X)\\
@V\tilde{\rho}(X) \otimes id_CVV        @VV\tilde{\rho}(X)V\\
\mathcal{M}(X) \otimes C  \otimes C    @> \upsilon(\mathcal{M} \otimes C)(X)>>  \mathcal{M}(X) \otimes C 
\end{CD}$$
Thus, we get
\begin{equation*}
\begin{array}{ll}
\upsilon(\mathcal{M} \otimes C)(X)\left((\tilde{\rho}(X) \otimes id_C)(m \otimes c)\right)&=\upsilon(\mathcal{M} \otimes C)(X)(m_0 \otimes m_1 \otimes c)\\
&=\tilde{\rho}(X)\left(\upsilon(\mathcal{M})(X)(m \otimes c)\right)=\rho_{\mathcal{M}(X)}\left(\upsilon(\mathcal{M})(X)(m \otimes c)\right)
\end{array}
\end{equation*}
Now applying $id_{\mathcal{M}(X)} \otimes \varepsilon_C$ on both sides, we obtain
\begin{equation*}
(\beta\alpha)(\upsilon)(\mathcal{M})(X)(m\otimes c)=  (id_{\mathcal{M}} \otimes \varepsilon_C)(X)\upsilon(\mathcal{M}\otimes C)(X)( m_0\otimes m_1 \otimes c)=\upsilon(\mathcal{M})(X)(m \otimes c)
\end{equation*}
\end{proof}

\begin{theorem}\label{ThmIx}
Let $\mathscr{F}: {\mathscr{M}(\psi)}_\mathcal{D}^C \longrightarrow Mod$-$\mathcal{D}$ be the forgetful functor and $\mathscr{G}:Mod$-$\mathcal{D} \longrightarrow {\mathscr{M}(\psi)}_\mathcal{D}^C$, $\mathcal{N} \mapsto \mathcal{N} \otimes C$ be its right adjoint. Then, $\mathscr{F}$ is separable if and only if  there exists $\theta \in V_1$ such that
\begin{equation*}
\theta_{X} \circ \Delta_C =  \varepsilon_C \cdot id_X \qquad \forall X \in Ob(\mathcal{D})
\end{equation*}
\end{theorem}
\begin{proof}
We first recall from \eqref{unit} that the unit of the adjunction is given by
$$\mu(\mathcal{M}):\mathcal{M}\longrightarrow \mathcal{M}\otimes C \qquad \mu(\mathcal{M})(X)(m)= m_0\otimes m_1$$ for $\mathcal{M}\in Ob({\mathscr{M}(\psi)}_\mathcal{D}^C)$ and $m \in \mathcal{M}(X)$. Suppose that $\mathscr{F}$ is separable. Then, by Theorem \ref{separability}, there exists $\upsilon \in V$ such that $\upsilon \circ \mu = 1_{{\mathscr{M}(\psi)}_\mathcal{D}^C}$. Therefore, using Proposition \ref{V=V1}, corresponding to $\upsilon \in V$ we can obtain an element  $\theta \in V_1$  given by
$\theta_{X}(c \otimes d)=\big((id_{{\bf h}_X} \otimes \varepsilon_C)\upsilon({\bf h}_X \otimes C)\big)(X)(id_X \otimes c \otimes d)$ for each $c,d \in C$.
Moreover, we have
\begin{equation*}
\begin{array}{ll}
(\theta_{X} \circ \Delta_C)(c)&=\big((id_{{\bf h}_X} \otimes \varepsilon_C)\upsilon({\bf h}_X \otimes C)\big)(X)(id_X \otimes c_1 \otimes c_2)\\
&= \big((id_{{\bf h}_X} \otimes \varepsilon_C)\upsilon({\bf h}_X \otimes C)\big)(X)\big(\mu({\bf h}_X \otimes C)(X)\big)(id_X \otimes c)\\
&= \big((id_{{\bf h}_X} \otimes \varepsilon_C)(id_{{{\bf h}_X} \otimes C})\big)(X)(id_X \otimes c)= \big(id_{{\bf h}_X}(X) \otimes \varepsilon_C\big)(id_X \otimes c)=\varepsilon_C(c) id_X
\end{array}
\end{equation*}
for any $c \in C$.
Conversely, suppose that $\theta\in V_1$ is such that $\theta_{X} \circ \Delta_C =  \varepsilon_C \cdot id_X$ for every
$X\in Ob(\mathcal D)$. Corresponding to $\theta \in V_1$ there exists $\upsilon \in V$ defined by
\begin{equation*}
\upsilon(\mathcal{M}): \mathcal{M}\otimes C\longrightarrow \mathcal{M}, \quad m \otimes c \mapsto \mathcal{M}\big(\theta_{X}( m_1 \otimes c)\big)(m_0)
\end{equation*}
for $\mathcal{M} \in Ob\big({\mathscr{M}(\psi)}_\mathcal{D}^C \big),~ X\in Ob(\mathcal{D}),~ m\in \mathcal{M}(X)$ and $c\in C$. Further, we have
\begin{equation*}
\begin{array}{ll}
(\upsilon \circ \mu)(\mathcal{M})(X)(m)=\upsilon(\mathcal{M})(X)(\mu(\mathcal{M})(X)(m))&=  \upsilon(\mathcal{M})(X)(m_0 \otimes m_1)\\
&=  \mathcal{M}\left(\theta_{X}\left((m_0)_1 \otimes m_1\right)\right)(m_0)_0\\
&=  \mathcal{M}\left(\theta_{X}\left( (m_1)_1 \otimes (m_1)_2\right)\right)(m_0)\\
&=  \mathcal{M}\left((\theta_{X}\circ \Delta_C)(m_1)\right)(m_0)\\
&=  \mathcal{M}\left((id_X) \varepsilon_C(m_1)\right)(m_0)= m
\end{array}
\end{equation*}
This shows that $\upsilon \circ \mu=1_{{\mathscr{M}(\psi)}_\mathcal{D}^C}$. Hence, $\mathscr{F}$ is separable by Theorem \ref{separability}.
\end{proof}

Next we investigate the separability of the functor $\mathscr{G}: Mod\text{-}\mathcal{D} \longrightarrow {\mathscr{M}(\psi)}_\mathcal{D}^C$ given by $\mathscr{G}(\mathcal{N})=  \mathcal{N} \otimes C$ for any $\mathcal{N}\in  Mod\text{-}\mathcal{D}$. Since $\mathscr{G}$ is a right adjoint of $\mathscr{F}$, it follows from Theorem \ref{separability} that the functor $\mathscr{G}$ is separable if and only if there exists a natural transformation $\omega:1_{Mod\text{-}\mathcal{D}} \longrightarrow \mathscr{F}\mathscr{G}$ such that $\nu\circ \omega =1_{Mod\text{-}\mathcal{D}}$, where  $\nu$ is the counit of the adjunction as explained in \eqref{counit}.
We set $W:=Nat(1_{Mod\text{-}\mathcal{D}},\mathscr{F}\mathscr{G})$ and proceed to give another interpretation of $W$.

\smallskip We define $h:\mathcal{D}^{op} \otimes \mathcal{D} \longrightarrow Vect_K$ as
\begin{align}\label{n1}
h(X,Y):&=Hom_\mathcal{D}(X,Y) \qquad \big(h(\phi)\big)(f):=\phi''f\phi'
\end{align}
for any $(X,Y) \in Ob(\mathcal{D}^{op} \otimes \mathcal{D})$, $\phi:=(\phi',\phi'') \in Hom_{\mathcal{D}^{op} \otimes \mathcal{D}}\big((X,Y),(X',Y')\big)$ and $f \in Hom_\mathcal{D}(X,Y)$. Similarly, we define the functor $h \otimes C:\mathcal{D}^{op} \otimes \mathcal{D} \longrightarrow Vect_K$ as
\begin{align}\label{n2}
(h \otimes C)(X,Y):&=Hom_\mathcal{D}(X,Y) \otimes C \qquad \big((h \otimes C)(\phi)\big)(f \otimes c):=\phi''f\phi'_\psi \otimes c^\psi
\end{align}
for any $(X,Y) \in Ob(\mathcal{D}^{op} \otimes \mathcal{D})$, $\phi:=(\phi',\phi'') \in Hom_{\mathcal{D}^{op} \otimes \mathcal{D}}\big((X,Y),(X',Y')\big)$, $f \in Hom_\mathcal{D}(X,Y)$ and $c \in C$. By slight abuse of notation, we will make no distinction between functors $\mathcal{D}^{op} \otimes \mathcal{D} \longrightarrow Vect_K$ and functors $\mathcal D\longrightarrow Mod\text{-}\mathcal D$. 
We observe that  $h \otimes C:\mathcal{D}^{op} \otimes \mathcal{D} \longrightarrow Vect_K$ corresponds to $\mathscr F\circ ({\bf h}\otimes C)
$ when viewed as a functor from $\mathcal D\longrightarrow Mod\text{-}\mathcal D$. 

\smallskip Given a natural transformation $\eta: h\longrightarrow h\otimes C$, it is easy to see that 
\begin{equation*}
  \eta({-},Y): {\bf h}_Y=Hom_\mathcal{D}({-},Y)\longrightarrow {\bf h}_Y\otimes C= Hom_\mathcal{D}({-},Y)\otimes C   
\end{equation*}
is a morphism of right $\mathcal{D}$-modules for each $Y \in Ob(\mathcal{D})$. Similarly, for each $X \in Ob(\mathcal{D})$,  
 \begin{equation*}
  \eta(X,{-}): {_X}{\bf h}=Hom_\mathcal{D}(X,{-})\longrightarrow {_X}{\bf h} \otimes C= Hom_\mathcal{D}(X,{-})\otimes C   
\end{equation*}
is a morphism of left $\mathcal{D}$-modules.

\smallskip
Throughout the rest of this section, we set $W_1:=Nat(h,h \otimes C)$, the $K$-space consisting of all natural transformations between the functors $h$ and $h \otimes C$.

\begin{lemma}\label{integral} 
Let $\eta \in W_1$. We set $\eta(X,X)(id_X)=\sum a_X \otimes c_{X}$ for each $X \in Ob(\mathcal{D})$ and $\eta(Y,Z)(g):=\sum \hat{g} \otimes c_g$ for any $g \in Hom_\mathcal{D}(Y,Z)$. Then, 
\begin{equation*}
\eta(Y,Z)(g)=\sum \hat{g} \otimes c_g=\sum a_Zg_\psi \otimes {c_Z}^\psi=\sum ga_Y \otimes c_Y
\end{equation*}
\end{lemma}
\begin{proof} Since $\eta(-,Z):{\bf h}_Z \longrightarrow {\bf h}_Z \otimes C$ is a morphism of right $\mathcal{D}$-modules for each $Z \in Ob(\mathcal{D})$, we have the following commutative diagram:
$$\begin{CD}
{\bf h}_Z(Z)@>\eta(Z,Z)>> {\bf h}_Z(Z) \otimes C\\
@V{\bf h}_Z(g)VV        @VV ({\bf h}_Z \otimes C)(g)V\\
{\bf h}_Z(Y)    @>\eta(Y,Z)>> {\bf h}_Z(Y) \otimes C
\end{CD}$$
This diagram alongwith Lemma \ref{rightaction1} gives
\begin{equation}\label{2.9.1}
\eta(Y,Z)(g)=\sum \big(({\bf h}_Z \otimes C)(g)\big)(a_Z \otimes c_{Z})=\sum {\bf h}_Z(g_\psi)(a_Z) \otimes {c_Z}^\psi=\sum a_Zg_\psi \otimes {c_Z}^\psi
\end{equation}
Since $\eta(Y,{-}):{_Y}{\bf h} \longrightarrow {_Y}{\bf h} \otimes C$ is a morphism of left $\mathcal{D}$-modules, we also have the following commutative diagram:
$$\begin{CD}
{_Y}{\bf h}(Y)@>\eta(Y,Y)>> {_Y}{\bf h}(Y) \otimes C\\
@V{_Y}{\bf h}(g)VV        @VV ({_Y}{\bf h} \otimes C)(g)V\\
{_Y}{\bf h}(Z)    @>\eta(Y,Z)>> {_Y}{\bf h}(Z) \otimes C
\end{CD}$$ 
This gives
\begin{equation}\label{2.9.2}
\eta(Y,Z)(g)=({_Y}{\bf h} \otimes C)(g)\left(\sum a_Y \otimes c_Y\right)=\sum ga_Y \otimes c_Y
\end{equation}
The result now follows from \eqref{2.9.1} and \eqref{2.9.2}.
\end{proof}

\begin{prop}\label{W=W1}
The $K$-spaces $W=Nat(1_{Mod\text{-}\mathcal{D}},\mathscr{F}\mathscr{G})$ and $W_1=Nat(h,h \otimes C)$ are isomorphic.
\end{prop}
\begin{proof}
We define a $K$-linear map $\gamma:W \longrightarrow W_1$ by setting
$$\eta=\gamma(\omega):h \longrightarrow h \otimes C \qquad \eta(X,Y):=\omega({\bf h}_Y)(X)$$
for any $(X,Y) \in Ob(\mathcal{D}^{op} \otimes \mathcal{D})$. We now verify that the map is well-defined. Let $\phi:=(\phi',\phi'') \in Hom_{\mathcal{D}^{op} \otimes \mathcal{D}}\big((X,Y),(X',Y')\big)$. Since $\omega({\bf h}_Y):{\bf h}_Y \longrightarrow {\bf h}_Y \otimes C$ is a morphism of right $\mathcal{D}$-modules, we have the following commutative diagram:
\begin{equation}\label{d1}
\begin{CD}
{\bf h}_Y(X)   @>\omega({\bf h}_Y)(X)>> {\bf h}_Y(X) \otimes C \\
@V{\bf h}_Y(\phi')VV        @VV({\bf h}_Y \otimes C)(\phi')V\\
{\bf h}_Y(X')    @>\omega({\bf h}_Y)(X')>>{\bf h}_Y(X') \otimes C
\end{CD}
\end{equation}
The morphism $\phi'':Y \longrightarrow Y'$ in $\mathcal{D}$ induces a morphism $\phi'':{\bf h}_Y \longrightarrow {\bf h}_{Y'}$ of right $\mathcal{D}$-modules. Therefore, using the naturality of $\omega$, we get the following commutative diagram:
\begin{equation}\label{d2}
\begin{CD}
{\bf h}_Y(X')   @>\omega({\bf h}_Y)(X')>> {\bf h}_Y(X') \otimes C \\
@V_{X'}{\bf h}(\phi'')={\bf h}_{\phi''}(X')VV        @VV(_{X'}{\bf h}(\phi'')\otimes id_C)=({\bf h}_{\phi''}(X') \otimes id_C)V\\
{\bf h}_{Y'}(X')    @>\omega({\bf h}_{Y'})(X')>>{\bf h}_{Y'}(X') \otimes C
\end{CD}
\end{equation}
We now observe that for $f\in Hom_{\mathcal D}(X,Y)$, we have
\begin{align*}
\big(h(\phi)\big)(f)&=\phi''f\phi'=_{X'}{\bf h}(\phi'')\big({\bf h}_Y(\phi')(f)\big)\\
\big((h\otimes C)(\phi)\big)(g \otimes c)&=\phi''g\phi'_\psi \otimes c^\psi=(_{X'}{\bf h}(\phi'')\otimes id_C)\big(({\bf h}_Y \otimes C)(\phi')(g \otimes c)\big)
\end{align*}
Thus, by combining the diagrams \eqref{d1} and \eqref{d2}, we obtain the following commutative diagram:
$$\begin{CD}
h(X,Y)   @>\eta(X,Y)>> h(X,Y) \otimes C \\
@Vh(\phi)VV        @VV(h \otimes C)(\phi)V\\
h(X',Y')    @>\eta(X',Y')>>h(X',Y') \otimes C
\end{CD}$$
This shows that $\eta \in W_1$.

\smallskip
Conversely, let $\eta \in W_1 = Nat(h,h \otimes C)$. For any $Y \in Ob(\mathcal{D})$,
\begin{equation}\label{etaX}
  \eta({-},Y): {\bf h}_Y \longrightarrow {\bf h}_Y\otimes C  
\end{equation}
is a morphism of right $\mathcal{D}$-modules. For any $f \in Hom_\mathcal{D}(X,Y)$, the naturality of $\eta$ gives us the following commutative diagram:
\begin{equation}\label{eex2.26}\begin{CD}
{\bf h}_X =Hom_\mathcal{D}({-},X)  @>\eta(-,X)>>Hom_\mathcal{D}({-},X)\otimes C= {\bf h}_X \otimes C \\
@Vh(-,f)VV        @VVh(-,f) \otimes id_CV\\
{\bf h}_Y  =Hom_\mathcal{D}({-},Y) @>\eta(-,Y)>>Hom_\mathcal{D}({-},Y)\otimes C={\bf h}_Y \otimes C
\end{CD}
\end{equation}
Now, for any $\mathcal{M}$ in $Mod$-$\mathcal{D}$, we know that $\mathcal{M}=\underset{y\in el(\mathcal M)}{colim}~ {\bf h}_{|y|}$. Similarly, $\mathcal{M}\otimes C=\underset{y\in el(\mathcal M)}{colim}({\bf h}_{|y|}\otimes C)$ where the colimit is taken in $Mod\text{-}\mathcal D$. Thus, the morphisms as in \eqref{etaX} induce a morphism $\omega(\mathcal{M}):\mathcal{M}\longrightarrow \mathcal{M}\otimes C$ of right $\mathcal{D}$-modules. Moreover, for any morphism $\mathcal{M} \overset{\zeta}{\longrightarrow} \mathcal{N}$ in $Mod$-$\mathcal{D}$, the commutative diagrams as in  \eqref{eex2.26}  induce the following equality: $$(\zeta\otimes id_C)\circ\omega(\mathcal{M})=\omega(\mathcal{N})\circ \zeta$$
Therefore, for $\eta \in W_1$ we have obtained a natural transformation $\omega:1_{Mod\text{-}\mathcal{D}} \longrightarrow \mathscr{FG}$ in $W$. We will denote this $K$-linear map by $\delta: W_1\longrightarrow W$, i.e., $\delta(\eta)= \omega$ determined by $\omega({\bf h}_Y):= \eta({-},Y)$ for each $Y\in Ob(\mathcal{D})$. It may be easily verified that the morphisms
$\gamma$ and $\delta$ are inverses of each other. 
\end{proof}

\begin{theorem}\label{ThmIIx}
Let $\mathscr{F}: {\mathscr{M}(\psi)}_\mathcal{D}^C \longrightarrow Mod$-$\mathcal{D}$ be the forgetful functor and $\mathscr{G}:Mod$-$\mathcal{D} \longrightarrow {\mathscr{M}(\psi)}_\mathcal{D}^C$, $\mathcal{N} \mapsto \mathcal{N} \otimes C$ be its right adjoint.  Then $\mathscr{G}$ is separable if and only if there exists $\eta \in W_1=Nat(h, h \otimes C)$ such that
\begin{equation}\label{Gsep}
(id_h \otimes \varepsilon_C)\eta=id_h
\end{equation}
\end{theorem}
\begin{proof}
Suppose that $\mathscr{G}$ is separable. Then, by Theorem \ref{separability}, there exists $\omega \in W=Nat(1_{Mod\text{-}\mathcal{D}},\mathscr{F}\mathscr{G})$ such that $\nu \circ \omega = 1_{Mod\text{-}\mathcal{D}}$, where $\nu$ is the counit of the adjunction. Using Proposition \ref{W=W1}, corresponding to $\omega \in W$, there exists an element $\eta \in W_1$ given by $\eta(X,Y)=\omega({\bf h}_Y)(X)$ for every $(X,Y) \in Ob(\mathcal{D}^{op} \otimes \mathcal{D})$. The condition \eqref{Gsep} now follows from the definition of the counit in \eqref{counit}.

\smallskip
Conversely, let $\eta \in W_1$ be such that $(id_h \otimes \varepsilon_C)\eta=id_h$. We consider $\omega:1_{Mod\text{-}\mathcal{D}}\longrightarrow \mathscr{FG}$ given by $\omega({\bf h}_Y):=\eta({-},Y)$ for each $Y\in Ob(\mathcal{D})$. Then, $(id_{{\bf{h}}_Y}\otimes \varepsilon_C)\omega({{\bf{h}}_Y})=(id_{{\bf{h}}_Y}\otimes \varepsilon_C)\eta({-},Y)=id_{{{\bf{h}}_Y}}$. Since $\mathscr F$ is a left adjoint and it is clear from the definition that $\mathscr G$ preserves colimits, we obtain that  $(id_{\mathcal N}\otimes \varepsilon_C)\omega(\mathcal{N})=id_{\mathcal N}$ for any $\mathcal{N}\in Mod$-$\mathcal{D}$, i.e, $(id\otimes \varepsilon_C)\omega=1_{Mod{\text{-}}\mathcal{D}}$. Therefore, $\mathscr{G}$ is separable by Theorem \ref{separability}.
\end{proof}

\subsection{Frobenius conditions}
Let $F:\mathcal{A}\longrightarrow \mathcal{B}$ be a functor which has a right adjoint $G:\mathcal{B}\longrightarrow \mathcal{A}$. Then, the pair $(F, G)$ is called a Frobenius pair if $G$ is both a right and a  left adjoint of $F$. We recall the following characterization for Frobenius pairs (see \cite[$\S$ 1]{uni})
\begin{theorem}
Let $F:\mathcal{A}\longrightarrow \mathcal{B}$ be a functor which has a right adjoint $G$. Then, $(F, G)$ is a Frobenius pair if and only if there exist $\upsilon \in Nat(GF,1_\mathcal{A})$ and $\omega \in Nat(1_\mathcal{B},FG)$ such that
\begin{align}
F(\upsilon(M)) \circ \omega({F(M)})&=id_{F(M)}\label{fc1}\\
\upsilon({G(N)}) \circ G(\omega(N))&=id_{G(N)}\label{fc2}
\end{align}
for all $M \in \mathcal{A}$ and $N \in \mathcal{B}$.
\end{theorem}

\begin{lemma}\label{2.13.}
For any $\omega \in W=Nat(1_{Mod\text{-}\mathcal{D}},\mathscr{F}\mathscr{G})$, $N \in Comod\text{-}C$ and $Y \in Ob(\mathcal{D})$, we have $\omega(N \otimes {\bf h}_Y)=id_N \otimes \omega({\bf h}_Y)$.
\end{lemma}
\begin{proof} For each $n \in N$, we define $\zeta_n:{\bf h}_Y \longrightarrow N \otimes {\bf h}_Y$ by
\begin{equation*}
\zeta_n(X)(f):=n \otimes f
\end{equation*}
for any $X \in Ob(\mathcal{D})$ and $f \in Hom_\mathcal{D}(X,Y)$. It may be easily verified that $\zeta_n$ is a morphism of right $\mathcal{D}$-modules. Therefore, using the naturality of $\omega$, we have the following commutative diagram:
$$\begin{CD}
{\bf h}_Y(X)@>\omega({\bf h}_Y)(X)>> {\bf h}_Y(X) \otimes C\\
@V\zeta_n(X)VV        @VV \zeta_n(X) \otimes id_CV\\
N \otimes {\bf h}_Y(X)    @>\omega(N \otimes {\bf h}_Y)(X)>> N \otimes {\bf h}_Y(X) \otimes C
\end{CD}$$
Let $f \in Hom_\mathcal{D}(X,Y)$. We set $\omega({\bf h}_Y)(X)(f)=\sum \hat{f} \otimes c_f$. Then, we have
\begin{equation*}
\begin{array}{ll}
\omega(N \otimes {\bf h}_Y)(X)(n \otimes f) &= (\zeta_n(X) \otimes id_C)\omega({\bf h}_Y)(X)(f) =\sum (\zeta_n(X) \otimes id_C)(\hat{f} \otimes c_f)\\
&= \sum n \otimes \hat{f} \otimes c_f= \left(id_N \otimes \omega({\bf h}_Y)\right)(X)(n \otimes f)
\end{array}
\end{equation*}
The result follows.
\end{proof}

\begin{theorem}\label{Frobcondition}
Let $\mathscr{F}: {\mathscr{M}(\psi)}_\mathcal{D}^C \longrightarrow Mod$-$\mathcal{D}$ be the forgetful functor and $\mathscr{G}:Mod$-$\mathcal{D} \longrightarrow {\mathscr{M}(\psi)}_\mathcal{D}^C$, $\mathcal{N} \mapsto \mathcal{N} \otimes C$ be its right adjoint. Then, $(\mathscr{F},\mathscr{G})$ is a Frobenius pair if and only if there exist $\theta \in V_1$ and $\eta \in W_1$ such that the following conditions hold:
\begin{align}
\varepsilon_C(d)f&=\sum \hat{f} \circ \theta_{X}(c_f \otimes d)\label{fro1}\\
\varepsilon_C(d)f &= \sum \hat{f}_\psi \circ \theta_{X}(d^\psi \otimes c_f)\label{fro2}
\end{align}
for any $f \in Hom_\mathcal{D}(X,Y)$, $d \in C$ and $\eta(X,Y)(f)=\sum \hat{f} \otimes c_f$.
\end{theorem}
\begin{proof}
Suppose there exist $\theta \in V_1$ and $\eta \in W_1$ such that \eqref{fro1} and \eqref{fro2} hold. Then, using the isomorphisms $V\cong V_1$ and $W\cong W_1$ as in Propositions \ref{V=V1} and \ref{W=W1}, there exist $\upsilon \in V$ and $\omega \in W$ corresponding to $\theta \in V_1$ and $\eta \in W_1$ respectively.
We also know by Proposition \ref{entgroth} that the collection $\{N \otimes {\bf h}_Y\}$, where $N$ ranges over all (isomorphisms classes of) finite dimensional $C$-comodules and $Y$ ranges over all objects in $\mathcal{D}$, forms a generating set for ${\mathscr{M}(\psi)}_\mathcal{D}^C$. Therefore, we first verify the condition \eqref{fc1} for $\mathcal{M}=N \otimes {\bf h}_Y\in {\mathscr{M}(\psi)}_\mathcal{D}^C$, where $N \in Comod\text{-}C$ and $Y \in Ob(\mathcal{D})$. For any $n \otimes f \in N \otimes Hom_\mathcal{D}(X,Y)$, we have

\begin{equation}\label{eq2.32qa}
\begin{array}{lll}
&\big(\mathscr{F}(\upsilon({N \otimes {\bf h}_Y})) \circ \omega({\mathscr{F}(N \otimes {\bf h}_Y)})\big)(X)(n \otimes f)&\\
& \quad =\upsilon({N \otimes  {\bf h}_Y})(X)\big(id_N \otimes \omega({\bf h}_Y)\big)(X)(n \otimes f)& (\text{by Lemma}~ \ref{2.13.})\\
&\quad =\upsilon({N \otimes  {\bf h}_Y})(X)\left(id_N \otimes \eta(X,Y)\right)(n \otimes f)&\\
& \quad = \sum \upsilon({N \otimes  {\bf h}_Y})(X)(n \otimes \hat{f} \otimes c_f)&\\
& \quad = \sum (N \otimes h_Y)\left(\theta_{X}((n \otimes \hat{f})_1 \otimes c_f)\right)(n \otimes \hat{f})_0&(\text{by}~ \eqref{defupsilon})\\
& \quad = \sum (N \otimes h_Y)(\theta_{X}({n_1}^\psi \otimes c_f))(n_0 \otimes \hat{f}_\psi)& (\text{by}~ \eqref{rightcomod2})\\
& \quad = \sum n_0 \otimes \hat{f}_\psi \circ \theta_{X}\big({n_1}^\psi \otimes c_f\big)& (\text{by}~ \eqref{rightaction2})\\
& \quad =  n_0 \otimes \varepsilon_C(n_1)f&(\text{by}~ \eqref{fro2})\\
&  \quad = n \otimes f
\end{array}
\end{equation}
This proves \eqref{fc1} for the generators of ${\mathscr{M}(\psi)}_\mathcal{D}^C$. As explained in the proof of Proposition \ref{entgroth}, for any $\mathcal{M}$ in ${\mathscr{M}(\psi)}_\mathcal{D}^C$, there is an epimorphism  
$$\bigoplus_{m\in el(\mathcal{M})} \eta_m: \bigoplus_{m\in el(\mathcal{M})}  V_m \otimes {\bf h}_{|m|} \longrightarrow \mathcal{M}$$ in ${\mathscr{M}(\psi)}_\mathcal{D}^C$.
The morphism $\eta:=\underset{m\in el(\mathcal{M})}{\bigoplus} \eta_m$ induces the following commutative diagram:
\begin{equation}\label{2.14d1}
\begin{CD}
\bigoplus\mathscr{F}(V_m\otimes {\bf{ h}}_{|m|})   @>\bigoplus\mathscr{F}\left(\upsilon\left({V_m \otimes {\bf{ h}}_{|m|}}\right)\right)  \omega\left({\mathscr{F}\left(V_m \otimes {\bf{ h}}_{|m|}\right)}\right)>> \bigoplus\mathscr{F}(V_m\otimes {\bf{ h}}_{|m|}) \\
@V\mathscr{F}(\eta)VV        @VV\mathscr{F}(\eta) V\\
\mathscr{F}(\mathcal{M})    @>\mathscr{F}\left(\upsilon(\mathcal{M})\right) \omega\left({\mathscr{F}(\mathcal{M})}\right) >>\mathscr{F}(\mathcal{M})
\end{CD}
\end{equation}
From \eqref{eq2.32qa}, it follows that $\mathscr{F}\left(\upsilon\left({V_m \otimes {\bf{ h}}_{|m|}}\right)\right)  \omega\left({\mathscr{F}\left(V_m \otimes {\bf{ h}}_{|m|}\right)}\right)=id_{\mathscr{F}(V_m \otimes {\bf{ h}}_{|m|})}$ for each $m\in el(\mathcal M)$. Thus, by the commutative diagram \eqref{2.14d1}, we have 
\begin{equation}\label{epi}
\big(\mathscr{F}(\upsilon(\mathcal M)) \circ \omega\left({\mathscr{F}(\mathcal{M})}\right)\big)\circ \mathscr{F}(\eta)= \mathscr{F}(\eta)
\end{equation}
Since $\mathscr{F}$ is a left adjoint, it preserves epimorphisms. Since $\eta$ is an epimorphism, so is $\mathscr{F}(\eta)$. Therefore, \eqref{epi} implies that $\mathscr{F}(\upsilon(\mathcal M)) \circ \omega\left({\mathscr{F}(\mathcal{M})}\right)=id_{\mathscr{F}(\mathcal{M})}.$ This proves \eqref{fc1} for any $\mathcal{M} \in Ob({\mathscr{M}(\psi)}_\mathcal{D}^C)$. 

\smallskip
Next, we verify the condition \eqref{fc2}.  From the definition, it is clear that $\mathscr{G}$ preserves colimits. Since any $\mathcal D$-module may be expressed as the colimit of representable functors, it is enough to verify the condition \eqref{fc2} for representable functors. For any $f \otimes d \in {\bf h}_Y(X)\otimes C$, we have 
\begin{equation*}
\begin{array}{lll}
\big(\upsilon\left({\mathscr{G}({\bf{ h}}_Y)}\right)(X) \circ \mathscr{G}(\omega\left({{\bf{ h}}_Y})\right)(X)\big)(f \otimes d)&=\big(\upsilon\left({\mathscr{G}({\bf{ h}}_Y)}\right)(X) \circ (\omega\left({{\bf{ h}}_Y}\right) \otimes id_C)(X)\big)(f \otimes d)&\\
&=\big(\upsilon\left({\bf{ h}}_Y \otimes C \right)(X)\circ \left(\eta(X,Y) \otimes id_C\right)\big)(f \otimes d)&\\
&= \sum \upsilon\left({{\bf{ h}}_Y \otimes C}\right)(X)(\hat{f} \otimes c_f \otimes d)&\\
&= \sum ({\bf{ h}}_Y \otimes C)\left(\theta_{X}\left((\hat{f} \otimes c_f)_1 \otimes d\right)\right) (\hat{f} \otimes c_f)_0&(\text{by}~ \eqref{defupsilon})\\
&= \sum ({\bf{ h}}_Y \otimes C)\left(\theta_{X}\left({c_f}_2 \otimes d\right)\right) (\hat{f} \otimes {c_f}_1)& (\text{by}~ \eqref{rightcomod1})\\
&= \sum {\bf{ h}}_Y \Big(\big(\theta_{X}({c_f}_2 \otimes d)\big)_\psi \Big)(\hat{f})  \otimes {{c_f}_1}^\psi& (\text{by}~ \eqref{rightaction1})\\
&= \sum \hat{f} \circ \big(\theta_{X}({c_f}_2 \otimes d)\big)_\psi \otimes {c_f}^\psi_1\\
&= \sum \hat{f} \circ \big(\theta_{X}(c_f \otimes d_1)\big) \otimes d_2& (\text{by}~ \eqref{theta2})\\
&=  \varepsilon_C(d_1)f \otimes d_2& (\text{by}~ \eqref{fro1})\\
&= f \otimes d
\end{array}
\end{equation*}
This proves \eqref{fc2}. Therefore, $(\mathscr F,\mathscr G)$ is a Frobenius pair.

\smallskip
Conversely, suppose $(\mathscr F,\mathscr G)$ is a Frobenius pair. Then, there exist $\upsilon \in V$ and $\omega \in W$ satisfying \eqref{fc1} and \eqref{fc2}. Then, using the isomorphisms $V\cong V_1$ and $W\cong W_1$ as in Propositions \ref{V=V1} and \ref{W=W1}, there exist $\theta \in V_1$ and $\eta \in W_1$ corrresponding to $\upsilon \in V$ and $\omega \in W$ respectively. We will now verify the conditions \eqref{fro1} and \eqref{fro2}. Taking $\mathcal{M}=C\otimes {\bf h}_Y$ in \eqref{fc1}, for any $d\in  C$ and $f \in Hom_\mathcal{D}(X,Y)$ we have
\begin{equation*}
\begin{array}{lll}
d \otimes f &= \big(\upsilon\left({C \otimes {\bf h}_Y}\right)(X) \circ \omega\left({C \otimes {\bf{ h}}_Y}\right)(X)\big)(d \otimes f)&\\
&= \big(\upsilon\left({C \otimes {\bf{ h}}_Y}\right)(X) \circ (id_C \otimes \omega_{{\bf{ h}}_Y}(X))\big)(d \otimes f)& (\text{by Lemma}~\ref{2.13.})\\
&=\upsilon\left({C \otimes {\bf{ h}}_Y}\right)(X)\left(d \otimes \eta(X,Y)(f)\right)&\\
&=\sum \upsilon\left({C \otimes {\bf{ h}}_Y}\right)(X)\left(d \otimes \hat{f} \otimes c_f\right)&\\
&= \sum (C \otimes {\bf{ h}}_Y)\left(\theta_{X}\left((d \otimes \hat{f})_1 \otimes c_f\right)\right) (d \otimes \hat{f})_0&(\text{by}~ \eqref{defupsilon})\\
&= \sum (C \otimes {\bf{ h}}_Y)\left(\theta_{X}\left(d_2^\psi \otimes c_f\right)\right) (d_1 \otimes \hat{f}_\psi)&(\text{by}~ \eqref{rightcomod2})\\
&= \sum d_1 \otimes \hat{f}_\psi \circ \theta_{X}(d_2^\psi \otimes c_f)&(\text{by}~ \eqref{rightaction2})
\end{array}
\end{equation*} 
Applying $\varepsilon_C \otimes id_{{\bf h}_Y(X)}$ on both sides, we get 
\begin{equation*}
\begin{array}{lll}
\varepsilon_C(d)f & =\sum \varepsilon_C(d_1)\hat{f}_\psi \circ \theta_X({d_2}^\psi\otimes c_f) & \\
&=\sum \varepsilon_C({d_1}^\psi)\hat{f}_{\psi_\psi} \circ \theta_X({d_2}^\psi\otimes c_f)  & \mbox{(by \eqref{eq 6.2})}\\
&= \sum \varepsilon_C((d^\psi)_1)\hat{f}_\psi \circ \theta_X((d^\psi)_2\otimes c_f)& \mbox{(by \eqref{eq 6.3})}\\
&=\sum \hat{f}_\psi\circ \theta_X(d^\psi\otimes c_f))&\\
\end{array}
\end{equation*}
This proves \eqref{fro2}. 
Now, taking $\mathcal{N}={\bf{ h}}_Y$ in \eqref{fc2}, we have
\begin{equation*}
\begin{array}{lll}
f \otimes d&=\big(\upsilon\left({\mathscr G({\bf{ h}}_Y)}\right)(X) \circ \mathscr G\left(\omega\left({{\bf{ h}}_Y}\right)\right)(X)\big)(f \otimes d)&\\
&=\left( \upsilon\left({{\bf{ h}}_Y \otimes C}\right)(X) \circ (\omega\left({{\bf{ h}}_Y}) \otimes id_C\right)(X)\right)(f \otimes d)&\\
&= \left(\upsilon({{\bf{ h}}_Y \otimes C})(X) \circ \left(\eta(X,Y) \otimes id_C\right)\right)(f \otimes d)&\\
&= \sum \upsilon\left({{\bf{ h}}_Y \otimes C}\right)(X) (\hat{f} \otimes c_f \otimes d)&\\
&= \sum ({\bf{ h}}_Y \otimes C)\left(\theta_{X}\left((\hat{f} \otimes c_f)_1 \otimes d\right)\right)(\hat{f} \otimes c_f)_0&(\text{by}~ \eqref{defupsilon})\\
&= \sum ({\bf{ h}}_Y \otimes C)\big(\theta_{X}\left({c_f}_2 \otimes d\right)\big)(\hat{f} \otimes {c_f}_1)& (\text{by}~ \eqref{rightcomod1})\\
&=\sum {\bf{ h}}_Y\big(\left(\theta_{X}({c_f}_2 \otimes d)\right)_\psi\big)(\hat{f}) \otimes {{c_f}_1}^\psi& (\text{by}~ \eqref{rightaction1})\\
&=\sum \hat{f} \circ \left(\theta_{X}({c_f}_2 \otimes d)\right)_{\psi} \otimes {{c_f}_1}^\psi&\\
&=\sum \hat{f}\circ\theta_{X}(c_f \otimes d_1) \otimes d_2&(\text{by}~ \eqref{theta2})
\end{array}
\end{equation*}
Applying $id_{{\bf h}_Y(X)} \otimes \varepsilon_C$ on both sides, we get \eqref{fro1}. This proves the result.
\end{proof}

\subsection{Frobenius conditions in the case of a finite dimensional coalgebra }

We continue with $(\mathcal D,C,\psi)$ being an entwining structure. 
For each $Y \in Ob(\mathcal{D})$, we obtain an object $Hom(C,{\bf h}_Y)$ in $Mod\text{-}\mathcal D$ by setting
\begin{equation}\label{rightactionhom}
\begin{array}{c}
Hom(C,{\bf h}_Y)(X):=Hom_K(C,{\bf h}_Y(X))\\
Hom(C,{\bf h}_Y)(g): ~Hom_K\big(C,{\bf h}_Y(X)\big) \longrightarrow Hom_K\big(C,{\bf h}_Y(X')\big)~ \text{given by}\\
Hom(C,{\bf h}_Y)(g)(\phi)(x)=(\phi \cdot g)(x):=\phi(x)g 
\end{array}
\end{equation}
for any $X \in Ob(\mathcal{D})$, $g \in Hom_\mathcal{D}(X',X)$, $\phi \in Hom_K\big(C,{\bf h}_Y(X)\big)$ and $x \in C$. 
Using \eqref{rightactionhom}, we now define a functor $Hom(C,h):\mathcal{D} \longrightarrow Mod\text{-}\mathcal D$ as follows:
\begin{equation}\label{leftactionhom}
\begin{array}{c}
Hom(C,h)(Y):=Hom(C,{\bf h}_Y)\\
\big(Hom(C,h)(f)\big)(Z): ~\big(Hom(C,{\bf h}_Y)\big)(Z) \longrightarrow \big(Hom(C,{\bf h}_{X})\big)(Z)~ \text{given by}\\
\big(Hom(C,h)(f)\big)(Z)(\phi)(x)=(f \cdot \phi)(x):= f_\psi\circ \phi(x^\psi)
\end{array}
\end{equation}
for any $f \in Hom_\mathcal{D}(Y,X)$,  $\phi \in Hom_K\big(C,{\bf h}_Y(Z)\big)$ and $x \in C$. 

\smallskip For the rest of this section, we assume that $C$ is finite dimensional. Then, for each $Z\in Ob(\mathcal D)$, we have an isomorphism
\begin{equation}\label{rxeq2.37}
Hom_K\big(C,{\bf h}_Y(Z)\big)\cong C^* \otimes {\bf h}_Y(Z)
\end{equation}   Let $\{d_i\}_{1\leq i\leq k}$ be a  basis for $C$ and 
$\{d_i^*\}_{1\leq i\leq k}$ be its dual basis.

\begin{lemma}\label{Lrep2.15}  Let $C$ be a finite dimensional coalgebra. Then, we have a functor
\begin{equation}\label{rxeq2.38}
C^*\otimes h:\mathcal D\longrightarrow Mod\text{-}\mathcal D \qquad Y\mapsto C^*\otimes {\bf h}_Y
\end{equation}
\end{lemma}
\begin{proof}
For each $Y\in Ob(\mathcal D)$, it is clear that $C^*\otimes {\bf h}_Y\in Mod\text{-}\mathcal D$. We consider 
$f\in Hom_{\mathcal D}(Y,X)$ and an element $c^*\otimes g\in C^*\otimes {\bf h}_Y(Z)$. By the isomorphism in \eqref{rxeq2.37},
$c^*\otimes g$ corresponds to the element $\phi_{c^*\otimes g}\in Hom_K\big(C,{\bf h}_Y(Z)\big)$ given by $\phi_{c^*\otimes g}
(x)=c^*(x)g$ for each $x\in C$. From the action in \eqref{leftactionhom}, the element $f\cdot \phi_{c^*\otimes g}\in 
\big(Hom(C,{\bf h}_{X})\big)(Z)$ is given by
\begin{equation*}
(f\cdot \phi_{c^*\otimes g})(x)=f_\psi\circ \phi_{c^*\otimes g}(x^\psi)=c^*(x^\psi)(f_\psi\circ g)
\end{equation*} Again, using the isomorphism in \eqref{rxeq2.37}, the element in $C^*\otimes {\bf h}_X(Z)$ corresponding 
to  $f\cdot \phi_{c^*\otimes g}$ is given by $\sum_{i=1}^kc^*(d_i^\psi)d_i^*\otimes f_\psi g$. It may be easily verified that $(C^* \otimes h)(f): C^* \otimes {\bf h}_Y \longrightarrow C^* \otimes {\bf h}_X$ is a morphism of right $\mathcal{D}$-modules. The result now follows.
\end{proof}

Since $C$ is a coalgebra, its vector space dual $C^*$ is an algebra with the convolution product $(c^*\mbox{\tiny $\bullet$  }d^*)(x):=\sum c^*(x_1)d^*(x_2)$ for $c^*, d^* \in C^*$ and $x \in C$. Let $N$ be any left $C^*$-module. Then, we have a $K$-linear map $\rho:N \longrightarrow Hom(C^*,N)$ defined by $\rho(n)(c^*):=c^*n$ for $n \in N$ and $c^* \in C^*$. 

\smallskip In general, there is an embedding $N \otimes C \hookrightarrow Hom(C^*,N)$ given by $(n \otimes x)(c^*):=c^*(x)n$ for $x\in C$. Since $C$ is finite dimensional, this embedding is also a surjection. This gives us a $K$-linear map $\rho:N \longrightarrow N \otimes C$ which makes $N$ a right $C$-comodule (see, for instance, \cite[$\S$ 2.2]{SCS}).  Then, $\rho(n)=\sum_{i=1}^k d_i^*n \otimes d_i$. In particular, $C^*$ becomes a right $C$-comodule with 
\begin{equation}\label{C*comod}
\rho_{C^*}(c^*)=\sum_{i=1}^k d_i^*\mbox{\tiny $\bullet$  }c^*  \otimes d_i
\end{equation} Considering the element $\varepsilon_C\in C^*$,  the coassociativity of the coaction 
$\rho_{C^*}$ may be used to verify that
\begin{equation}\label{mamta}
\sum_{j=1}^k\sum_{i=1}^k(d_i^*\mbox{\tiny $\bullet$  }d_j^*)\otimes d_i\otimes d_j=\sum_{j=1}^kd_j^*\otimes \Delta(d_j)
\end{equation}

\begin{prop}\label{Prep2.15} Let $C$ be a finite dimensional coalgebra. Then, we have a functor:
\begin{equation*}
C^*\otimes h:\mathcal D\longrightarrow {\mathscr{M}(\psi)}_\mathcal{D}^C \qquad Y\mapsto C^*\otimes {\bf h}_Y
\end{equation*}
\end{prop}
\begin{proof} From \eqref{C*comod}, we know that $C^*$ is a right $C$-comodule. 
 Applying Lemma \ref{lem 6.3}, it follows that each $C^* \otimes {\bf h}_Y$ is an object in ${\mathscr{M}(\psi)}_\mathcal{D}^C$. Accordingly,  the right $C$-comodule structure on $C^* \otimes {\bf h}_Y(Z)$ for any $Z \in Ob(\mathcal{D})$  is given by the following composition:
$$\begin{CD}
\sigma^r_{ C^*\otimes {\bf h}_Y(Z)}:C^* \otimes {\bf h}_Y(Z)   @>\rho_{C^*} \otimes id>> C^* \otimes C \otimes {\bf h}_Y(Z) @>id \otimes \psi_{ZY}>>C^* \otimes {\bf h}_Y(Z) \otimes C \\
\end{CD}$$ Explicitly, we have
$\sigma^r_{ C^*\otimes {\bf h}_Y(Z)}(c^* \otimes g)=\sum\limits_{i=1}^k d_i^*\mbox{\tiny $\bullet$  }  c^* \otimes g_\psi \otimes  d_i^\psi$
for each $c^* \otimes g\in C^*\otimes {\bf h}_Y(Z)$. We consider $f\in Hom_{\mathcal D}(Y,X)$. By Lemma \ref{Lrep2.15}, this induces a morphism $C^*\otimes {\bf h}_Y\longrightarrow C^*\otimes {\bf h}_X$ in $Mod\text{-}\mathcal D$. In order to show that
$C^*\otimes h:\mathcal D\longrightarrow {\mathscr{M}(\psi)}_\mathcal{D}^C$ is a functor, it therefore suffices to show
that each morphism 
\begin{equation} C^*\otimes {\bf h}_Y(Z)\longrightarrow C^*\otimes {\bf h}_X(Z) \qquad (c^*\otimes g)\mapsto \sum_{j=1}^kc^*(d_j^\psi)d_j^*\otimes f_\psi g
\end{equation}  
is right $C$-colinear. For any $c^* \otimes g \in  C^* \otimes {\bf h}_Y(Z)$, we have
\begin{equation*}
\begin{array}{lll}
\sigma^r_{C^* \otimes {\bf h}_{X}(Z)}\left(f \cdot (c^* \otimes g)\right)&= \sum_{j=1}^k\sigma^r_{C^* \otimes {\bf h}_{X}(Z)}\left(c^*(d_j^\psi)d_j^*\otimes f_\psi g\right)&\\
&= \sum_{i=1}^k \sum_{j=1}^kc^*(d_j^\psi)d_i^*\mbox{\tiny $\bullet$  } d_j^*\otimes (f_\psi g)_\psi \otimes d_i^\psi&\\
&= \sum_{i=1}^k \sum_{j=1}^kc^*(d_j^\psi)d_i^* \mbox{\tiny $\bullet$  } d_j^*\otimes {f_\psi}_\psi g_\psi \otimes {d_i^\psi}^\psi& (\text{by}~ \eqref{eq 6.1})\\
&= \sum_{j=1}^k c^*({{d_j}_2}^\psi)  d_j^*\otimes {f_\psi}_\psi g_\psi \otimes {{{d_j}_1}^\psi}^\psi& (\text{by}~ \eqref{mamta})\\
&= \sum_{j=1}^k c^*({{d_j}_2}^\psi)  d_j^*\otimes {f_\psi}_\psi g_\psi \otimes  \left(\sum_{i=1}^k d_i^*({{d_j}_1}^\psi)d_i^\psi\right)&\\
&=  \sum_{j=1}^k \sum_{i=1}^k  d_i^*({{d_j}_1}^\psi)  c^*({{d_j}_2}^\psi)  d_j^*\otimes {f_\psi}_\psi g_\psi \otimes d_i^\psi&\\
&=  \sum_{j=1}^k \sum_{i=1}^k  d_i^*({({d_j}^\psi)}_1)  c^*({({d_j}^\psi)}_2)  d_j^*\otimes {f_\psi} g_\psi \otimes d_i^\psi& (\text{by}~ \eqref{eq 6.3})\\
&=  \sum_{j=1}^k \sum_{i=1}^k ( d_i^* \mbox{\tiny $\bullet$  } c^*)(d_j^\psi)    d_j^*\otimes {f_\psi} g_\psi \otimes d_i^\psi&\\
&= (f \otimes id_C) \cdot \big( \sum_{i=1}^k ( d_i^* \mbox{\tiny $\bullet$  } c^*) \otimes  g_\psi \otimes d_i^\psi\big)&\\
&= (f \otimes id_C) \cdot \left(\sigma^r_{C^* \otimes {\bf h}_{Y}(Z)}(c^* \otimes g)\right)&
\end{array}
\end{equation*}
\end{proof}

Since $C$ is finite dimensional, the right $C$-comodule structure on $C^* \otimes {\bf h}_Y(X)$ induces a  right $C$-comodule structure on $Hom(C,{\bf h}_Y(X))$ for each $X,Y \in Ob(\mathcal{D})$ which we now explain. Let $\phi \in Hom(C,{\bf h}_Y(X))$. Then, $\phi$ corresponds to the element $\sum_{1 \leq i \leq k} d_i^* \otimes \phi(d_i) \in C^* \otimes {\bf h}_Y(X)$. We know by Proposition \ref{Prep2.15} that $$\sigma^r_{C^* \otimes {\bf h}_Y(X)}\left(\sum_{i=1}^k d_i^* \otimes \phi(d_i)\right)=\sum_{j=1}^k \sum_{i=1}^k d_j^* \mbox{\tiny $\bullet$  } d_i^* \otimes (\phi(d_i))_\psi \otimes d_j^\psi$$
The element $\sum_{i=1}^k d_j^* \mbox{\tiny $\bullet$  } d_i^* \otimes (\phi(d_i))_\psi \otimes d_j^\psi \in C^* \otimes {\bf h}_Y(X) \otimes C$
corresponds to the element $\phi_0 \otimes \phi_1 \in Hom(C,{\bf h}_Y(X)) \otimes C$ given by
\begin{equation}\label{coactionhom}
\begin{array}{ll}
\phi_0(x) \otimes \phi_1 &= \sum_{j=1}^k \sum_{i=1}^k (d_j^* \mbox{\tiny $\bullet$  } d_i^*)(x)(\phi(d_i))_\psi \otimes d_j^\psi\\
&= \sum_{j=1}^k \sum_{i=1}^k d_j^*(x_1) d_i^*(x_2)(\phi(d_i))_\psi \otimes d_j^\psi\\
&=\psi(x_1 \otimes \phi(x_2))
\end{array}
\end{equation} for $x\in C$.
It now follows from \eqref{rightactionhom}, \eqref{leftactionhom}, \eqref{coactionhom} and Proposition \ref{Prep2.15} that we have a functor
\begin{equation}
Hom(C,h):\mathcal D\longrightarrow {\mathscr{M}(\psi)}_\mathcal{D}^C \qquad Y\mapsto Hom_K(C,{\bf h}_Y)
\end{equation}

We also recall from \eqref{left1} and \eqref{left action}, the functor $h \otimes C:\mathcal{D}\longrightarrow {\mathscr{M}(\psi)}_\mathcal{D}^C$, defined as follows:
\begin{equation*}
\begin{array}{l}
(h \otimes C)(Y):={\bf h}_Y\otimes C\\
(h\otimes C)(f)(Z)(g\otimes c):= fg\otimes c
\end{array}
\end{equation*}
for $f\in Hom_\mathcal{D}(Y,X)$ and $g \otimes c \in {\bf h}_Y(Z) \otimes C$. We now set $V_2:=Nat(h \otimes C, C^* \otimes h)$.

\begin{prop}\label{V1isoV2}
Let $C$ be a finite dimensional coalgebra. Then,
$$V=Nat(\mathscr{G}\mathscr{F},1_{{\mathscr{M}(\psi)}_\mathcal{D}^C}) \cong V_1 \cong V_2=Nat(h \otimes C, C^* \otimes h)$$
\end{prop}
\begin{proof}
Since $C$ is finite dimensional, we know that $C^* \otimes {\bf h}_Y(X) \cong Hom_K\big(C,{\bf h}_Y(X)\big)$ for each $X,Y \in Ob(\mathcal{D})$. We first define a $K$-linear map $\Upsilon_{XY}:{\bf h}_Y(X) \otimes C \longrightarrow C^* \otimes {\bf h}_Y(X)$ given by
\begin{equation}\label{R2.43Rr}
\big(\Upsilon_{XY}(f \otimes c)\big)(d):=f_\psi \circ \theta_{X}( d^\psi \otimes c)
\end{equation}
for any $f \in Hom_\mathcal{D}(X,Y)$ and $c,d \in C$. In other words, we have
\begin{equation}\label{phiform}
\Upsilon_{XY}(f \otimes c)=\sum_{i=1}^k d_i^* \otimes \big(f_\psi \circ \theta_{X}( d_i^\psi \otimes c)\big)
\end{equation}
where $\{d_i\}_{1\leq i\leq k}$ is a basis for $C$ and $\{d_i^*\}_{1\leq i\leq k}$ is its dual basis.

\smallskip
We now define $\alpha':V_1 \longrightarrow V_2$ by setting $\alpha'(\theta)=\Upsilon$ with $\Upsilon:h \otimes C \longrightarrow C^* \otimes h$ defined as follows:
\begin{align*}
\Upsilon_Y: {\bf h}_Y \otimes C \longrightarrow C^* \otimes {\bf h}_Y \qquad \Upsilon_Y(X):=\Upsilon_{XY}
\end{align*}
for any $X,Y \in Ob(\mathcal{D})$.  
We now verify that $\alpha'$ is a well-defined map. For this, we first check that $\Upsilon_Y: {\bf h}_Y \otimes C \longrightarrow C^* \otimes {\bf h}_Y$ is a morphism in ${\mathscr{M}(\psi)}_\mathcal{D}^C$ for every $Y \in Ob(\mathcal{D})$. For any $g \in Hom_\mathcal{D}(X',X)$, we need to show that the following diagram commutes:
$$\begin{CD}
{\bf h}_Y(X)\otimes C  @>\Upsilon_Y(X)>> C^* \otimes {\bf h}_Y(X)\\
@V({\bf h}_Y \otimes C)(g)VV        @VV (C^* \otimes {\bf h}_Y)(g)V\\
{\bf h}_Y(X') \otimes C     @>\Upsilon_Y(X') >> C^* \otimes {\bf h}_Y(X')
\end{CD}$$
For any $f\otimes c \in {\bf h}_Y(X)\otimes C$, we have
\begin{equation*}
\begin{array}{lll}
(C^* \otimes {\bf h}_Y)(g)\Upsilon_Y(X)(f \otimes c)&=\sum\limits_{i=1}^k (C^* \otimes {\bf h}_Y)(g)\big(d_i^* \otimes f_\psi \circ \theta_{X}(d_i^\psi \otimes c)\big)&\\
&=\sum\limits_{i=1}^k d_i^* \otimes \big(f_\psi \circ \theta_{X}(d_i^\psi \otimes c)\big)\circ g &(\text{by}~ \eqref{rightaction2})\\
&=\sum\limits_{i=1}^k d_i^* \otimes f_\psi {g_\psi}_\psi \circ \theta_{X'}({d_i^\psi}^\psi \otimes c^\psi)&(\text{by}~ \eqref{theta1})\\
&=\sum\limits_{i=1}^k d_i^* \otimes (f{g_\psi})_\psi \circ \theta_{X'}(d_i^\psi \otimes c^\psi) &(\text{by}~ \eqref{eq 6.1})\\
&=\Upsilon_Y(X')(fg_\psi \otimes c^\psi)=\Upsilon_Y(X')({\bf h}_Y \otimes C)(g)(f \otimes c)&
\end{array}
\end{equation*}
This shows that $\Upsilon_Y$ is a morphism of right $\mathcal{D}$-modules for every $Y\in Ob(\mathcal{D})$. Next we verify that $\Upsilon_Y(X): {\bf h}_Y(X) \otimes C \longrightarrow C^* \otimes {\bf h}_Y(X)$ is right $C$-colinear for every $X,Y \in Ob(\mathcal{D})$. We have
\begin{equation*}
\begin{array}{lll}
\sigma^r_{C^* \otimes {\bf h}_Y(X)}\big(\Upsilon_Y(X)(f \otimes c)\big)&= \sum\limits_{i=1}^k \sigma^r_{C^* \otimes {\bf h}_Y(X)}\big(d_i^* \otimes f_\psi \circ \theta_{X}(d_i^\psi \otimes c)\big)&\\
&= \sum\limits_{i,j=1}^k d_j^* \mbox{\tiny $\bullet$  } d_i^* \otimes \big(f_\psi \circ \theta_{X}(d_i^\psi \otimes c)\big)_\psi \otimes d_j^\psi& (\text{by}~\eqref{C*comod})\\
&=\sum\limits_{i=1}^k d_i^*\otimes \left(f_\psi \circ \theta_X\left({d_{i_2}}^\psi \otimes c\right)\right)_\psi \otimes {d_{i_1}}^\psi &(\text{by~}\eqref{mamta})\\
&=\sum\limits_{i=1}^k d_i^*\otimes {f_\psi}_\psi \circ \left(\theta_X\left({d_{i_2}}^\psi \otimes c\right)\right)_\psi \otimes {{d_{i_1}}^\psi}^\psi  &(\text{by}~ \eqref{eq 6.1})\\
&=\sum\limits_{i=1}^k d_i^*\otimes f_\psi \circ \left(\theta_X\left({(d_i^\psi)}_2 \otimes c\right)\right)_\psi \otimes {{(d_i^\psi)}_1}^\psi  &(\text{by}~ \eqref{eq 6.3})\\
&= \sum\limits_{i=1}^k d_i^*  \otimes f_\psi\circ \theta_{X}(d_i^\psi \otimes c_1) \otimes c_2& (\text{by}~ \eqref{theta2})\\
&=\Upsilon_Y(X)(f \otimes c_1) \otimes c_2\\
&=(\Upsilon_Y(X)\otimes id_C)\left(\pi^r_{{\bf h}_Y(X) \otimes C}(f \otimes c)\right)
\end{array}
\end{equation*} 
Finally, we verify that $\Upsilon$ is a natural transformation from $h \otimes C$ to $C^* \otimes h$, i.e., 
the following diagram commutes for any $g \in Hom_\mathcal{D}(Y,Y')$:
$$\begin{CD}
{\bf h}_Y \otimes C  @>\Upsilon_Y>> C^* \otimes {\bf h}_Y\\
@V(h \otimes C)(g)VV        @VV (C^* \otimes h)(g)V\\
{\bf h}_{Y'} \otimes C     @>\Upsilon_{Y'} >> C^* \otimes {\bf h}_{Y'}
\end{CD}$$
For any $f \otimes c \in {\bf h}_Y(X) \otimes C$, we have 
\begin{equation*}
\begin{array}{lll}
(C^* \otimes h)(g)(X)\Upsilon_Y(X)(f \otimes c)&=\sum\limits_{i=1}^k (C^* \otimes h)(g)(X)\big(d_i^* \otimes f_\psi \circ \theta_{X}(d_i^\psi \otimes c)\big)& (\text{by}~ \eqref{phiform})\\
&= \sum\limits_{i,j=1}^k d_i^*(d_j^\psi) d_j^* \otimes g_\psi f_\psi  \theta_{X}(d_i^\psi \otimes c)& (\text{by Lemma}~ \ref{Lrep2.15})\\
&= \sum\limits_{j=1}^k d_j^* \otimes g_\psi f_\psi \circ \theta_{X}\left(\sum\limits_{i=1}^n d_i^*({d_j}^\psi){d_i}^\psi \otimes c\right)&\\  
&=\sum\limits_{j=1}^k d_j^* \otimes g_\psi f_\psi \circ \theta_{X}({d_j^\psi}^{\psi} \otimes c)\\
&=\sum\limits_{j=1}^k d_j^* \otimes (gf)_\psi \circ \theta_{X}( d_j^\psi \otimes c)& (\text{by}~ \eqref{eq 6.1})\\
&= \Upsilon_{Y'}(X)(gf \otimes c)=\Upsilon_{Y'}(X)(h \otimes C)(g)(X)(f \otimes c)
\end{array}
\end{equation*}
This proves that $\Upsilon \in V_2$.

\smallskip
For the converse, we first observe that the functors $C^* \otimes h$ and $Hom(C,h)$ are isomorphic which follows from \eqref{rxeq2.37}. We define $\beta':V_2 \longrightarrow V_1$ by setting $\beta'(\Upsilon)=\theta$ with $\theta_{X}: C \otimes C \longrightarrow End_\mathcal{D}(X)$ defined as follows:
$$\theta_{X}(c \otimes d):=\big(\Upsilon_{XX}(id_X \otimes d)\big)(c)$$
for any $X \in Ob(\mathcal{D})$ and $c, d \in C$. We will now verify that $\theta$ satisfies \eqref{theta1} and \eqref{theta2}. For each $X \in Ob(\mathcal{D})$, we know that $\Upsilon_X:{\bf h}_X \otimes C \longrightarrow Hom(C,{\bf h}_X)$ is a morphism of right $\mathcal{D}$-modules. Therefore, for any $f \in Hom_\mathcal{D}(Y,X)$, we have the following commutative diagram:
\begin{equation}\label{drightactionhom}
\begin{CD}
{\bf h}_X(X) \otimes C  @>\Upsilon_{X}(X)>> Hom_K(C,{\bf h}_X(X))\\
@V({\bf h}_X \otimes C)(f)VV        @VV Hom(C,{\bf h}_X)(f)V\\
{\bf h}_{X}(Y) \otimes C     @>\Upsilon_{X}(Y) >> Hom_K(C,{\bf h}_X(Y))
\end{CD}
\end{equation}
Since $\Upsilon:h \otimes C \longrightarrow Hom(C,h)$ is a natural transformation, the following diagram also commutes for any $f \in Hom_D(Y,X)$:
\begin{equation}\label{dleftactionhom}
\begin{CD}
{\bf h}_Y \otimes C  @>\Upsilon_Y>> Hom(C,{\bf h}_Y)\\
@V(h \otimes C)(f)VV        @VV Hom(C,h)(f)V\\
{\bf h}_{X} \otimes C     @>\Upsilon_{X} >> Hom(C,{\bf h}_X)
\end{CD}
\end{equation}
Therefore, we have
\begin{equation*}
\begin{array}{lll}
\theta_{X}(c \otimes d) \circ f&=\big(\left(\Upsilon_{XX}(id_X \otimes d)\right)(c)\big) \circ f&\\
&= \big(\left(\Upsilon_{XX}(id_X \otimes d)\right) \cdot f\big)(c) &(\text{by}~ \eqref{rightactionhom})\\
&= \left(\Upsilon_{X}(Y) (h_X \otimes C)(f)(id_X \otimes d)\right)(c)&(\text{by}~ \eqref{drightactionhom})\\
&=  \left(\Upsilon_{X}(Y)h_X(f_\psi)(id_X) \otimes d^\psi\right) (c)&\\
&=  \Upsilon_{YX}\big(f_\psi \otimes d^\psi \big)(c)&\\
&= \Upsilon_{YX} \circ \big((h \otimes C)(f_\psi)(Y)(id_Y \otimes d^\psi)\big)(c)&\\
&= \big(f_\psi \cdot \Upsilon_{YY}(id_Y \otimes d^\psi)\big)(c)& (\text{by}~ \eqref{dleftactionhom})\\
&= {f_\psi}_\psi \circ \big(\Upsilon_{YY}(id_Y \otimes d^\psi)\big)(c^\psi)& (\text{by}~ \eqref{leftactionhom})\\
&= {f_\psi}_\psi \circ \big(\theta_{Y}(c^\psi \otimes d^\psi)\big)
\end{array}
\end{equation*}
This proves \eqref{theta1}. Further, we have
\begin{equation*}
\begin{array}{lll}
 \big(\theta_X(c_2 \otimes d)\big)_\psi \otimes {c_1}^\psi &=  \psi\big(c_1 \otimes \theta_X(c_2 \otimes d)\big)&\\
&=  \psi\big(c_1 \otimes \left(\Upsilon_{XX}(id_X \otimes d)\right)(c_2)\big)&\\
&=  \big(\Upsilon_{XX}(id_X \otimes d)\big)_0(c) \otimes (\Upsilon_{XX}(id_X \otimes d)\big)_1 & (\text{by}~ \eqref{coactionhom})\\
&=   \big(\Upsilon_{XX}(id_X \otimes d)_0\big)(c) \otimes (id_X \otimes d)_1& (\Upsilon_{XX} \text{~is~} C\text{-colinear})\\
&=  \big(\Upsilon_{XX}(id_X \otimes d_1)\big)(c) \otimes d_2 &(\text{by}~ \eqref{rightcomod1})\\
&=   \theta_{X}(c \otimes d_1) \otimes d_2&
\end{array}
\end{equation*}
This proves \eqref{theta2}. 
It remains to show that $\alpha'$ and $\beta'$ are inverses of each other. For every $\theta \in V_1$ and $c,d \in C$, it follows from 
\eqref{R2.43Rr} that
\begin{equation*}
\begin{array}{ll}
\big((\beta' \circ \alpha')(\theta)\big)_X(c\otimes d)&=(\alpha'(\theta))_{XX}(id_X\otimes d)(c)
=\theta_X(c\otimes d)
\end{array}
\end{equation*}
Finally, for any $\Upsilon \in V_2$, $f\in Hom_\mathcal{D}(X,Y)$ and $c,d \in C$, we have 
\begin{equation*}
\begin{array}{lll}
\big((\alpha' \circ \beta')(\Upsilon)\big)_{XY}(f\otimes c)(d)&= \sum\limits_{i=1}^k d_i^*(d) f_\psi \circ\big((\beta'(\Upsilon))_X(d_i^\psi \otimes c)\big)&\\
&= \sum\limits_{i=1}^k d_i^*(d) f_\psi\circ \big(\Upsilon_{XX}(id_X \otimes c)(d_i^\psi)\big)&\\
&=f_\psi\circ \big(\Upsilon_{XX}(id_X \otimes c)(d^\psi)\big)&\\
&= \big(f \cdot \left(\Upsilon_{XX}(id_X \otimes c)\right)\big)(d)&(\text{by}~\eqref{leftactionhom})\\
&= \left(\Upsilon_{XY}(f \otimes c)\right)(d)& (\text{by}~ \eqref{dleftactionhom})
\end{array}
\end{equation*}
This proves the result.
\end{proof}

\begin{prop}\label{W1isoW2}
Let $C$ be a finite dimensional coalgebra. Then, we have isomorphisms
$$W=Nat(1_{Mod\text{-}\mathcal{D}},\mathscr{F}\mathscr{G}) \cong W_1=Nat(h,h \otimes C) \cong W_2:=Nat(C^* \otimes h, h \otimes C)$$
\end{prop}
\begin{proof}
Given an $\eta:h \longrightarrow h \otimes C$, we want to define $\Phi:C^* \otimes h \longrightarrow h \otimes C$. For each $Y \in Ob(\mathcal{D})$, we first define a $K$-linear map $\Phi_{YY}:C^* \otimes {\bf{ h}}_Y(Y) \longrightarrow {\bf{ h}}_Y(Y) \otimes C$ by the following composition:
$$\mbox{\small $\begin{CD}
C^* \otimes {\bf{ h}}_Y(Y) @>id_{C^*} \otimes \eta(Y,Y)>> C^* \otimes  {\bf{ h}}_Y(Y) \otimes C @>id_{C^* \otimes {\bf{ h}}_Y(Y)} \otimes \Delta_C>> C^* \otimes {\bf{ h}}_Y(Y)   \otimes C \otimes C \\
@. @. @V\tau \otimes id_C VV \\
@. @. {\bf{ h}}_Y(Y) \otimes C \otimes (C^* \otimes C) @>ev>> {\bf{ h}}_Y(Y) \otimes C
\end{CD}$}$$
i.e.,  $\Phi_{YY}(c^* \otimes id_Y)=\sum a_Y \otimes c^*(c_{Y_2})c_{Y_1}$, where $\sum a_Y \otimes c_{Y}=\eta(Y,Y)(id_Y)$ as in the notation of Lemma \ref{integral}.
We observe that an element $c^* \otimes f \in C^* \otimes {\bf{ h}}_Y(X)$ may be written as $c^* \otimes f= (C^* \otimes {\bf{ h}}_Y)(f)(c^* \otimes id_Y)$.  For each $X \in Ob(\mathcal{D})$, we now define $\Phi_{XY}:C^* \otimes {\bf{ h}}_Y(X) \longrightarrow {\bf{ h}}_Y(X) \otimes C$ as follows:
\begin{equation}\label{defPhiW2}
\Phi_{XY}(c^* \otimes f):=({\bf{ h}}_Y \otimes C)(f)\left(\Phi_{YY}(c^* \otimes id_Y)\right)=\sum a_Yf_\psi \otimes c^*(c_{Y_2}){c_{Y_1}}^\psi
\end{equation}
for any  $c^* \otimes f  \in C^* \otimes {{\bf h}}_Y(X)$.

We define $\gamma':W_1 \longrightarrow W_2$ by setting $\gamma'(\eta)=\Phi$ with $\Phi:C^* \otimes h \longrightarrow h \otimes C$ given by
\begin{equation*}
\Phi_Y:C^* \otimes {\bf{ h}}_Y \longrightarrow {\bf{ h}}_Y \otimes C \qquad \Phi_Y(X):= \Phi_{XY}
\end{equation*}
for every $X,Y \in Ob(\mathcal{D})$. We now verify that $\gamma'$ is a well-defined map. For this, we first check that $\Phi_Y: C^* \otimes {\bf{ h}}_Y \longrightarrow  {\bf{ h}}_Y \otimes C $ is a morphism of right $\mathcal{D}$-modules for every $Y \in Ob(\mathcal{D})$, i.e., the following diagram commutes for any $g \in Hom_\mathcal{D}(X',X)$: 
$$\begin{CD}
C^* \otimes {\bf{ h}}_Y(X)  @>\Phi_Y(X)>> {\bf{ h}}_Y(X)\otimes C\\
@V(C^* \otimes {\bf{ h}}_Y)(g)VV        @VV ({\bf{ h}}_Y \otimes C)(g)V\\
C^* \otimes {\bf{ h}}_Y(X')     @>\Phi_Y(X') >> {\bf{ h}}_Y(X') \otimes C
\end{CD}$$
We have
\begin{equation*}
\begin{array}{ll}
\Phi_Y(X')(C^* \otimes {\bf{ h}}_Y)(g)(c^* \otimes f)&=\Phi_Y(X')(c^* \otimes fg)=\sum a_Y(fg)_\psi \otimes c^*(c_{Y_2}){c_{Y_1}}^\psi\\
&=\sum a_Yf_\psi g_\psi \otimes c^*(c_{Y_2}){{c_{Y_1}}^\psi}^\psi= ({\bf{ h}}_Y \otimes C)(g)\left(\Phi_Y(X)(c^* \otimes f)\right)
\end{array}
\end{equation*}
Next we verify that $\Phi_Y(X): C^* \otimes {\bf{ h}}_Y(X) \longrightarrow {\bf{ h}}_Y(X) \otimes C$ is right $C$-colinear for any $X,Y \in Ob(\mathcal{D})$:
\begin{equation*}
\begin{array}{ll}
(\Phi_Y(X) \otimes id_C)\left(\sigma^r_{C^* \otimes {\bf{ h}}_Y(X)}(c^* \otimes f)\right) &=\sum_{i=1}^k \Phi_Y(X)\big(d_i^* \mbox{\tiny $\bullet$  }  c^* \otimes f_\psi\big) \otimes d_i^\psi\\
&=\sum_{i=1}^k \sum a_Y{f_\psi}_\psi \otimes (d_i^* \mbox{\tiny $\bullet$  }  c^*)(c_{Y_2}){c_{Y_1}}^\psi \otimes d_i^\psi\\
&=\sum_{i=1}^k \sum a_Y{f_\psi}_\psi \otimes d_i^*(c_{Y_2}) c^*(c_{Y_3}){c_{Y_1}}^\psi \otimes d_i^\psi\\
&=\sum a_Y{f_\psi}_\psi \otimes c^*(c_{Y_3}){c_{Y_1}}^\psi \otimes {c_{Y_2}}^\psi\\
&=\sum a_Y{f_\psi}_\psi \otimes c^*(c_{Y_2}){{(c_{Y_1})}_1}^\psi \otimes {{(c_{Y_1})}_2}^\psi \\
&= \sum a_Y{f_\psi} \otimes c^*(c_{Y_2}) {({c_{Y_1}}^\psi)}_1 \otimes {({c_{Y_1}}^\psi)}_2 \\
&= \pi^r_{{\bf{ h}}_Y(X) \otimes C}\left(\Phi_Y(X)(c^* \otimes f)\right)
\end{array}
\end{equation*} It follows that $\Phi_Y: C^* \otimes {\bf{ h}}_Y \longrightarrow  {\bf{ h}}_Y \otimes C $ is a morphism in ${\mathscr{M}(\psi)}_\mathcal{D}^C$.
To show that $\Phi\in Nat(C^*\otimes h,h\otimes C)$, it remains to verify that the following diagram commutes:
$$\begin{CD}
C^* \otimes {\bf{ h}}_Y  @>\Phi_Y>> {\bf{ h}}_Y \otimes C\\
@V(C^* \otimes h)(g)VV        @VV (h \otimes C)(g)V\\
 C^* \otimes {\bf{ h}}_{Z}     @>\Phi_{Z} >>{\bf{ h}}_{Z} \otimes C
\end{CD}$$
for any $g \in Hom_\mathcal{D}(Y,Z)$. For any $X \in Ob(\mathcal{D})$ and $c^* \otimes f \in C^* \otimes {\bf{ h}}_Y(X)$, we have 
\begin{equation*}
\begin{array}{ll}
\Phi_{Z}(X)(C^* \otimes h)(g)(X)(c^* \otimes f)&= \sum_{i=1}^k \Phi_{Z}(X)\Big(c^*(d_i^\psi)d_i^* \otimes g_\psi f\Big)\\
&= \sum_{i=1}^k \sum c^*(d_i^\psi)a_{Z}(g_\psi f)_\psi \otimes d_i^*(c_{{Z}_2}){c_{{Z}_1}}^\psi\\
&= \sum_{i=1}^k \sum c^*(d_i^\psi)a_{Z}{g_\psi}_\psi f_\psi \otimes d_i^*(c_{{Z}_2}){c_{{Z}_1}}^{\psi^\psi}\\
&=\sum c^*({c_{{Z}_2}}^\psi)a_{Z}{g_\psi}_\psi f_\psi \otimes {c_{{Z}_1}}^{\psi^\psi}\\
&= \sum c^*\big({({c_Z}^\psi)}_2\big)a_{Z}g_\psi f_\psi \otimes {{({c_Z}^\psi)}_1}^\psi \\
&= \sum c^*\big(c_{Y_2}\big)ga_Y f_\psi \otimes ( c_{Y_1})^\psi \quad~~~~~ (\text{by}~ \text{Lemma}~ \ref{integral})\\
&= (h \otimes C)(g)(X)\Phi_Y(X)(c^* \otimes f)
\end{array}
\end{equation*}

Conversely, we define $\delta':W_2 \longrightarrow W_1$ by setting $\delta'(\Phi)=\eta$ with $\eta:h \longrightarrow h \otimes C$ given by
\begin{equation}\label{W2toW1}
\eta(X,Y)(f):=\Phi_Y(X)\big(\varepsilon_C \otimes f\big)
\end{equation}
for any $(X,Y) \in Ob(\mathcal{D}^{op} \otimes \mathcal{D})$ and $f \in Hom_\mathcal{D}(X,Y)$.
We now verify that $\eta \in W_1$. Let $\phi:(X,Y) \longrightarrow (X',Y')$ be a morphism in $\mathcal{D}^{op} \otimes \mathcal{D}$ given by $\phi':X' \longrightarrow X$ and $\phi'':Y \longrightarrow Y'$ in $\mathcal{D}$. Then, using the fact that $\Phi_Y:C^* \otimes {\bf{ h}}_Y \longrightarrow {\bf{ h}}_Y \otimes C$ is a morphism of right $\mathcal{D}$-modules, we have
\begin{equation*}
\begin{array}{ll}
({\bf{ h}}_Y \otimes C)(\phi')\eta(X,Y)(f)&=({\bf{ h}}_Y \otimes C)(\phi') \Phi_Y(X)(\varepsilon_C \otimes f)=\Phi_Y(X')(C^* \otimes {\bf{ h}}_Y)(\phi')(\varepsilon_C \otimes f)\\
&=\Phi_Y(X')(\varepsilon_C \otimes f\phi')=\eta(X',Y)\left({\bf{ h}}_Y(\phi')(f)\right)
\end{array}
\end{equation*}
for any $f \in Hom_\mathcal{D}(X,Y)$. This shows that the following diagram commutes:
\begin{equation}\label{diagram3}
\begin{CD}
{\bf{ h}}_Y(X)   @>\eta(X,Y)>> {\bf{ h}}_Y(X) \otimes C \\
@V{\bf{ h}}_Y(\phi')VV        @VV({\bf{ h}}_Y \otimes C)(\phi')V\\
{\bf{ h}}_Y(X')    @>\eta(X',Y)>>{\bf{ h}}_Y(X') \otimes C
\end{CD}
\end{equation}
Now using the naturality of $\Phi:C^* \otimes h \longrightarrow h \otimes C$, we also have
\begin{equation*}
\begin{array}{ll}
({_{X'}}{\bf h} \otimes C)(\phi'')\eta(X',Y)(g)&=(h \otimes C)(\phi'')\Phi_Y(X')(\varepsilon_C \otimes g)\\
&= \Phi_{Y'}(X')(C^* \otimes h)(\phi'')(\varepsilon_C \otimes g)\\
&= \Phi_{Y'}(X')\big(\sum_{i=1}^k \varepsilon_C(d_i^\psi)d_i^* \otimes \phi''_\psi g \big)\\
&= \Phi_{Y'}(X')\big(\sum_{i=1}^k \varepsilon_C(d_i)d_i^* \otimes \phi'' g \big)\\
&= \Phi_{Y'}(X')\big( \varepsilon_C \otimes \phi'' g \big)=\eta(X',Y')( \phi'' g)=\eta(X',Y')({_{X'}}{\bf{ h}}(\phi'')(g))
\end{array}
\end{equation*}
for any $g \in {\bf{ h}}_Y(X')$. Thus, we get the following commutative diagram:
\begin{equation}\label{diagram4}
\begin{CD}
{\bf{ h}}_Y(X')   @>\eta(X',Y)>> {\bf{ h}}_Y(X') \otimes C \\
@V{_{X'}}{\bf{ h}}(\phi'')VV        @VV({_{X'}}{\bf{ h}} \otimes C)(\phi'')V\\
{\bf{ h}}_{Y'}(X')    @>\eta(X',Y')>>{\bf{ h}}_{Y'}(X') \otimes C
\end{CD}
\end{equation}
It now follows from \eqref{diagram3} and \eqref{diagram4} that the following diagram commutes:
$$\begin{CD}
h(X,Y)   @>\eta(X,Y)>> h(X,Y) \otimes C \\
@Vh(\phi)VV        @VV(h \otimes C)(\phi)V\\
h(X',Y')    @>\eta(X',Y')>>h(X',Y') \otimes C
\end{CD}$$
This shows that $\eta \in W_2$.
It remains to check that $\gamma'$ and $\delta'$ are inverses of each other. 
First we verify that $\big((\delta' \circ \gamma')(\eta)\big)(X,Y)=\eta(X,Y)$ for all $X,Y \in Ob(\mathcal{D})$. For this, we set $\Phi=\gamma'(\eta)$. Then, for any $f \in Hom_\mathcal{D}(X,Y)$, we have
\begin{equation*}
\begin{array}{lll}
\left((\delta'\circ \gamma')(\eta)\right)(X,Y)(f)&= \Phi_Y(X)(\varepsilon_C \otimes f)&\\
&= ({\bf{ h}}_Y \otimes C)(f){\Phi}_Y(Y)(\varepsilon_C \otimes id_Y)& (\text{by}~ \eqref{defPhiW2})\\
&= \sum ({\bf{ h}}_Y \otimes C)(f)\big(a_Y \otimes \varepsilon(c_{Y_2})c_{Y_1}\big)&\\
&= \sum ({\bf{ h}}_Y \otimes C)(f)(a_Y \otimes c_Y)&\\
&= \sum a_Yf_\psi \otimes {c_Y}^\psi = \eta(X,Y)(f)&(\text{by Lemma}~ \ref{integral})
\end{array}
\end{equation*}
Next, we will show that $\big((\gamma' \circ \delta')(\Phi)\big)_Y(X)=\Phi_Y(X)$ for any $X,Y \in Ob(\mathcal{D})$. Since $C^* \otimes {\bf{ h}}_Y(X)$ and ${\bf{ h}}_Y(X) \otimes C$ are right $C$-comodules for any $X,Y \in Ob(\mathcal{D})$, they are also left $C^*$-modules. The left actions are respectively given by
\begin{gather}
d^*(c^* \otimes f):=\sum_{i=1}^k d^*(d_i^\psi) (d_i^* \mbox{\tiny $\bullet$  } c^*) \otimes f_\psi \label{leftactionC*1}\\
d^*(f \otimes x):= d^*(x_2) (f \otimes x_1)\label{leftactionC*2}
\end{gather}
for any $d^*, c^* \in C^*$, $f \in {\bf{ h}}_Y(X)$ and $x \in C$. Moreover, since $\Phi_Y(X):C^* \otimes {\bf{ h}}_Y(X) \longrightarrow {\bf{ h}}_Y(X) \otimes C$ is right $C$-colinear, it is also left $C^*$-linear.
We  now set $\eta=\delta'(\Phi)$. Then, for any $c^* \otimes f \in C^* \otimes {\bf{ h}}_Y(X)$, we have
\begin{equation*}
\begin{array}{lll}
&\left((\gamma' \circ \delta')(\Phi)\right)_Y(X)(c^* \otimes f)\\
& \quad ={\left(\gamma'(\eta)\right)}_Y(X)(c^* \otimes f)&\\
& \quad =\sum a_Yf_\psi \otimes c^*(c_{Y_2}){c_{Y_1}}^\psi &\\
& \quad =\sum ({\bf{ h}}_Y \otimes C)(f)\left(c^*(c_{Y_2})(a_Y \otimes c_{Y_1})\right) & \\
& \quad =({\bf{ h}}_Y \otimes C)(f)\left(c^*\left(\sum a_Y \otimes c_Y\right)\right)& (\text{by}~ \eqref{leftactionC*2})\\
& \quad =({\bf{ h}}_Y \otimes C)(f)\big(c^*\left(\Phi_Y(Y)(\varepsilon_C \otimes id_Y)\right)\big)&\\
& \quad =({\bf{ h}}_Y \otimes C)(f)\left(\Phi_Y(Y)(c^* (\varepsilon_C \otimes id_Y))\right)&(\text{since}~\Phi_Y(X) \text{~is}~C^*\text{-linear})~ \\
& \quad =({\bf{ h}}_Y \otimes C)(f)\Phi_Y(Y)(c^*    \otimes id_Y)&(\text{by}~ \eqref{leftactionC*1}) \\
& \quad = \Phi_Y(X)(C^* \otimes {\bf{ h}}_Y)(f)(c^* \otimes id_Y)& (\Phi_Y~ \text{is a morphism of right}~ \mathcal{D}\text{-modules})\\
& \quad= \Phi_Y(X)(c^* \otimes f)&\\
\end{array}
\end{equation*}
This proves the result.
\end{proof}

\begin{theorem}\label{ThmIIIx}
Let $(\mathcal{D},C,\psi)$ be an entwining structure and assume that $C$ is a finite dimensional coalgebra. Let $\mathscr{F}: {\mathscr{M}(\psi)}_\mathcal{D}^C \longrightarrow Mod$-$\mathcal{D}$ be the functor forgetting the $C$-coaction and $\mathscr{G}:Mod$-$\mathcal{D} \longrightarrow {\mathscr{M}(\psi)}_\mathcal{D}^C$ given by $\mathcal{N} \mapsto \mathcal{N} \otimes C$ be its right adjoint. Then, the following statements are equivalent:

\smallskip
(i) $(\mathscr{F},\mathscr{G})$ is a Frobenius pair.

(ii) There exist $\eta \in W_1$ and $\theta \in V_1$ such that the corresponding morphisms
$\gamma'(\eta)=\Phi : C^*\otimes h \longrightarrow h\otimes C \text{~and~}\alpha'(\theta)=\Upsilon: h\otimes C\longrightarrow C^*\otimes h \text{~given by ~}$
\begin{equation*}
\begin{array}{ll}
\Phi_{XY}(c^*\otimes f) &= \sum a_Yf_\psi \otimes c^*(c_{Y_2}){c_{Y_1}}^\psi\\
\Upsilon_{XY}(f \otimes d) &= \sum_{i=1}^k d_i^* \otimes f_\psi \circ \theta_{X}( d_i^\psi \otimes d)
\end{array}
\end{equation*}
where  $f\in {\bf{ h}}_Y(X)$, $c^*\in C^*$ and $d\in C$, are inverses of each other.

\smallskip
(iii) $C^*\otimes h$ and $h\otimes C$ are isomorphic as objects of the category $_\mathcal{D}{\mathscr{M}(\psi)}_\mathcal{D}^C$ of  functors from $\mathcal{D}$ to ${\mathscr{M}(\psi)}_\mathcal{D}^C$.
\end{theorem}
\begin{proof}
(i) $\Rightarrow$ (ii) By assumption, there exist $\eta \in W_1$ and $\theta \in V_1$ satisfying \eqref{fro1} and \eqref{fro2}. Then, $\alpha'(\theta) = \Upsilon$ and $\gamma'(\eta)=\Phi$ are morphisms in $_\mathcal{D}{\mathscr{M}(\psi)}_\mathcal{D}^C$ in the notation of Proposition \ref{V1isoV2} and Proposition \ref{W1isoW2}. Since $\Upsilon_{XY}:{\bf{ h}}_Y(X) \otimes C \longrightarrow C^* \otimes {\bf{ h}}_Y(X)$ and $\Phi_{XY}:C^* \otimes {\bf{ h}}_Y(X) \longrightarrow {\bf{ h}}_Y(X) \otimes C$ are right $C$-colinear, they are also left $C^*$-linear. Using this fact and \eqref{leftactionC*1}, we have
\begin{equation*}
\begin{array}{lll}
\Upsilon_{XY}(\Phi_{XY}(c^* \otimes f))&= \Upsilon_{XY}\left(\Phi_{XY}\left((C^*\otimes {\bf{h}}_Y)(f)(c^*\otimes id_Y)\right)\right)&\\
&=\Upsilon_{XY}\left(({\bf{h}}_Y \otimes C)(f)\left(\Phi_{YY}(c^* \otimes id_Y)\right)\right)&\\
&= (C^*\otimes {\bf{h}}_Y)(f)\left( \Upsilon_{YY}\left(\Phi_{YY}(c^*\otimes id_Y)\right)\right)&\\
&= (C^*\otimes {\bf{h}}_Y)(f)\left( \Upsilon_{YY}\left(\Phi_{YY}(c^* \mbox{\tiny $\bullet$  }  \varepsilon_C \otimes id_Y)\right)\right)&\\
&= (C^*\otimes {\bf{h}}_Y)(f)\left(c^* \cdot \left( \Upsilon_{YY}\left(\Phi_{YY}(\varepsilon_C\otimes id_Y)\right)\right)\right)&\\
&=  (C^*\otimes {\bf{h}}_Y)(f)\left(c^* \cdot \left( \Upsilon_{YY}\left(\eta(Y,Y)(id_Y)\right)\right)\right)&\\
&=  (C^*\otimes {\bf{h}}_Y)(f)\left(c^* \cdot \left( \sum_{i=1}^k\sum d_i^* \otimes {(a_Y)}_\psi \circ \theta_{X}( d_i^\psi \otimes c_Y)              \right)\right)&(\text{by~} \eqref{phiform})\\
&=  (C^*\otimes {\bf{h}}_Y)(f)\left(c^* \cdot \left( \sum_{i=1}^k \varepsilon_C(d_i)d_i^* \otimes id_Y\right)\right)&(\text{by~} \eqref{fro2})\\
&=  (C^*\otimes {\bf{h}}_Y)(f)\left(c^* \mbox{\tiny $\bullet$  }  \varepsilon_C \otimes id_Y\right) = c^* \otimes f &
\end{array}
\end{equation*}
for any $c^* \otimes f\in C^*\otimes {\bf{ h}}_Y(X)$. Thus, $\Upsilon \circ \Phi = id_{C^*\otimes h}$. 

\smallskip
Using the naturality of $\Upsilon$ and $\Phi$, we have
\begin{equation*}
\begin{array}{lll}
\Phi_{XY}(\Upsilon_{XY}(f\otimes c) )&=\Phi_{XY}\left(\Upsilon_{XY}\left((h\otimes C)(f)(X)(id_X\otimes c)\right)\right)&\\
&=(h\otimes C)(f)(X)\left(\Phi_{XX}\left(\Upsilon_{XX}(id_X\otimes c)\right)\right)&\\
&=(h\otimes C)(f)(X)\left(\Phi_{XX}\left(\sum_{i=1}^k d_i^* \otimes \theta_{X}( d_i \otimes c)\right)\right)&\\
&=(h\otimes C)(f)(X)\left(\sum_{i=1}^k  \sum a_X{\left(\theta_{X}( d_i\otimes c)\right)}_\psi \otimes d_i^*(c_{X_2}){c_{X_1}}^\psi\right)&\\
&= (h\otimes C)(f)(X) \sum a_X{(\theta_{X}(c_{X_2}\otimes c))}_\psi \otimes {c_{X_1}}^\psi&\\
&= (h\otimes C)(f)(X) \sum a_X \circ {\theta_{X}( c_{X}\otimes c_1)}\otimes c_2&(\text{by~}\eqref{theta2})\\
&=(h\otimes C)(f)(X)\left( \varepsilon_C(c_1)id_X\otimes c_2 \right)&(\text{by~} \eqref{fro1})\\
&= f\otimes c&
\end{array}
\end{equation*}
for any $f \otimes c \in {{\bf h}}_Y(X) \otimes C$. Thus, $\Phi\circ \Upsilon= id_{h\otimes C}$. This proves $(ii)$.

\smallskip
(ii) $\Rightarrow$ (iii) is obvious since both $\Phi$ and $\Upsilon$ are morphisms in $_\mathcal{D}{\mathscr{M}(\psi)}_\mathcal{D}^C$.

\smallskip
(iii) $\Rightarrow$ (i) Let $\Phi: C^*\otimes h\longrightarrow h\otimes C$ denote the isomorphism in $_\mathcal{D}{\mathscr{M}(\psi)}_\mathcal{D}^C$. We consider the following morphism of $(\mathcal{D}^{op} \otimes \mathcal{D})$-modules
\begin{equation*}
\Lambda: h \longrightarrow C^*\otimes h \qquad \Lambda_Y(X)(f):= \varepsilon_C\otimes f 
\end{equation*}
for any $f\in Hom_\mathcal{D}(X,Y)$. We now set $\eta=\Phi\circ\Lambda\in W_1$ and $\theta= \beta'(\Phi^{-1})\in V_1$ where $\beta'$ is as in Proposition \ref{V1isoV2}. If $\eta(X,Y)(f)=\sum \hat{f} \otimes c_f$, then
\begin{equation}\label{2.53}
\begin{array}{lll}
\varepsilon_C\otimes f=\Phi^{-1}_{XY}(\Phi_{XY}(\varepsilon_C\otimes f))&= \Phi^{-1}_{XY}(\eta(X,Y)(f))&\\
&=\sum\Phi^{-1}_{XY}(\hat{f}\otimes c_f)&\\
&=\sum(\alpha'(\theta))_{XY}(\hat{f}\otimes c_f)&\\
&=\sum\sum_{i=1}^k d_i^* \otimes \hat{f}_\psi \circ \theta_{X}( d_i^\psi \otimes c_f)&(\text{by~}\eqref{phiform})
\end{array}
\end{equation}
Using the isomorphism as in \eqref{rxeq2.37} and evaluating the equality in \eqref{2.53} at $d\in C$, we get \eqref{fro2}. We also have
\begin{equation*}
\begin{array}{lll}
&id_X\otimes d=\Phi_{XX}(\Phi_{XX}^{-1}(id_X\otimes d))=\Phi_{XX}\big((\alpha'(\theta))_{XX}(id_X\otimes d)\big)&\\
&=\sum_{i=1}^k\Phi_{XX}(d_i^* \otimes \theta_{X}( d_i\otimes d))&(\text{by~}\eqref{phiform})\\
&=\sum_{i=1}^k\Phi_{XX}\left((C^* \otimes {{\bf h}}_X)(\theta_{X}( d_i\otimes d))(d_i^* \otimes id_X)\right)&\\
&=\sum_{i=1}^k ({{\bf h}}_X \otimes C)(\theta_{X}( d_i\otimes d))\left(\Phi_{XX}(d_i^* \otimes id_X)\right)&\\
&=\sum_{i=1}^k ({{\bf h}}_X \otimes C)(\theta_{X}( d_i\otimes d))\left(\Phi_{XX}(d_i^*\cdot (\varepsilon_C \otimes id_X))\right)&(\text{by}~ \eqref{leftactionC*1}) \\
&=\sum_{i=1}^k ({{\bf h}}_X \otimes C)(\theta_{X}( d_i\otimes d))\left(d_i^*\cdot \Phi_{XX}(\varepsilon_C \otimes id_X)\right)&(\text{since~}\Phi_{XX} \text{~is~} C^*\text{-linear})\\
&=\sum_{i=1}^k ({{\bf h}}_X \otimes C)(\theta_{X}( d_i\otimes d))\left(d_i^*\cdot (\eta(X,X)(id_X))\right)&\\
&=\sum_{i=1}^k \sum ({{\bf h}}_X \otimes C)(\theta_{X}( d_i\otimes d))\left(d_i^*\cdot (a_X \otimes c_X)\right)&\\
&=\sum_{i=1}^k \sum ({{\bf h}}_X \otimes C)(\theta_{X}( d_i\otimes d))\left(d_i^*(c_{X_2})(a_X \otimes c_{X_1})\right)&\\
&=\sum ({{\bf h}}_X \otimes C)(\theta_{X}(c_{X_2} \otimes d))(a_X \otimes c_{X_1})&\\
&=\sum a_X \circ \left(\theta_{X}( c_{X_2}\otimes d)\right)_\psi\otimes {c_{X_1}}^\psi&\\
&=\sum a_X \circ \theta_{X}( c_{X}\otimes d_1) \otimes d_2&(\text{by~}\eqref{theta2})
\end{array}
\end{equation*}
By applying the map $id_{{\bf h}_X(X)}\otimes \varepsilon_C$, we obtain 
\begin{equation}\label{6.45}
\varepsilon_C(d)\cdot id_X=\sum a_X \big(\theta_{X}( c_{X}\otimes d)\big)
\end{equation}
Now using Lemma \ref{integral} and \eqref{6.45}, we obtain
\begin{equation*}
\begin{array}{lll}
\sum \hat{f}\otimes  \left(\theta_{X}( c_f\otimes d)\right)=\sum (id_{{\bf h}_Y(X)}\otimes \theta_X)(\hat{f} \otimes c_f \otimes d)&= (id_{{\bf h}_Y(X)}\otimes \theta_X)\left(\eta(X,Y)(f) \otimes d \right)&\\
&=\sum (id_{{\bf h}_Y(X)}\otimes \theta_X)\left(fa_X \otimes c_X \otimes d \right)&\\
\end{array}
\end{equation*}
for any $f \in Hom_\mathcal{D}(X,Y)$ and $d \in C$. Applying to both sides the composition $Hom_{\mathcal D}(X,Y)\otimes 
Hom_{\mathcal D}(X,X)\longrightarrow Hom_{\mathcal D}(X,Y)$, we obtain $\sum \hat{f}\circ  \left(\theta_{X}( c_f\otimes d)\right)=\sum f a_X \circ \theta_{X}( c_{X}\otimes d)=\varepsilon_C(d)f$. This proves \eqref{fro1}. Therefore, $(\mathscr{F},\mathscr{G})$ is a Frobenius pair by Theorem \ref{Frobcondition}. This completes the proof.
\end{proof}

\section{Categorical Galois extensions and entwining structures }
Let $\mathcal{D}$ be a small $K$-linear category. Let $(\mathcal D,C,\psi)$ be a right-right entwining structure. We denote by ${_\mathcal{D}}\mathscr{M}_{\mathcal{D}}$ the category of $\mathcal{D}$-$\mathcal{D}$ bimodules, i.e., the category whose objects are functors from $\mathcal{D}^{op} \otimes \mathcal{D}$ to $Vect_K$ and whose morphisms are natural transformations between these functors. We recall the functors $h$ and $h \otimes C$ in ${_\mathcal{D}}\mathcal{M}_{\mathcal{D}}$ from \eqref{n1} and \eqref{n2} respectively:
\begin{align}\label{h_D}
h(X,Y)&=Hom_\mathcal{D}(X,Y) \qquad \big(h(\phi)\big)(f)=\phi''f\phi'\\
(h \otimes C)(X,Y)&=Hom_\mathcal{D}(X,Y) \otimes C \qquad \big((h \otimes C)(\phi)\big)(f \otimes c)=\phi''f\phi'_\psi \otimes c^\psi
\end{align}
for any $(X,Y) \in Ob(\mathcal{D}^{op} \otimes \mathcal{D})$, $\phi:=(\phi',\phi'') \in Hom_{\mathcal{D}^{op} \otimes \mathcal{D}}\big((X,Y),(X',Y')\big)$ and $f \in Hom_\mathcal{D}(X,Y)$, $c\in C$. We refer, for instance, to \cite[$\S$ 2.2]{EV} for the tensor product  which makes ${_\mathcal{D}}\mathscr{M}_{\mathcal{D}}$ a monoidal category with $h \in {_\mathcal{D}}\mathscr{M}_{\mathcal{D}}$ as the unit object.

\begin{definition}
A $\mathcal D$-coring $\mathscr C$  is a coalgebra object  in the monoidal category ${_\mathcal{D}}\mathscr{M}_{\mathcal{D}}$. Explicitly, a $\mathcal{D}$-coring is a functor $\mathscr{C}:\mathcal{D}^{op} \otimes \mathcal{D} \longrightarrow Vect_K$ with two morphisms
\begin{equation*}
\Delta_{\mathscr C}:\mathscr{C} \longrightarrow \mathscr{C} \otimes_\mathcal{D} \mathscr{C}, \qquad \varepsilon_{\mathscr C}:\mathscr{C} \longrightarrow h
\end{equation*}
satisfying the coassociativity and counit axioms in ${_\mathcal{D}}\mathscr{M}_{\mathcal{D}}$. A right $\mathscr{C}$-comodule consists of a right $\mathcal{D}$-module $\mathcal{M}$ equipped with a morphism $\rho_\mathcal{M}:\mathcal{M} \longrightarrow \mathcal{M} \otimes_\mathcal{D} \mathscr{C}$ of right $\mathcal{D}$-modules satisfying
\begin{equation}\label{con1}
(id_\mathcal{M} \otimes_\mathcal{D} \Delta_{\mathscr C})\circ \rho_\mathcal{M}=(\rho_\mathcal{M} \otimes_\mathcal{D} id_\mathscr{C}) \circ \rho_\mathcal{M}\qquad (id_\mathcal{M} \otimes_\mathcal{D} \varepsilon_{\mathscr C}) \circ \rho_\mathcal{M} = id_\mathcal{M}
\end{equation}
A morphism $\eta:(\mathcal{M},\rho_{\mathcal M}) \longrightarrow (\mathcal{N},\rho_{\mathcal N})$ of right $\mathscr{C}$-comodules is a morphism $\eta:\mathcal{M} \longrightarrow \mathcal{N}$ of right $\mathcal{D}$-modules satisfying
\begin{equation*}
\rho_\mathcal{N} \circ \eta = (\eta \otimes_\mathcal{D} id_\mathscr{C}) \circ \rho_\mathcal{M}
\end{equation*}
The category of right $\mathscr{C}$-comodules will be denoted by $Comod$-$\mathscr{C}$. 
\end{definition}

\begin{lemma}\label{entcoring}
Let $(\mathcal{D},C,\psi)$ be a right-right entwining structure. Then, the functor $h \otimes C$ is a $\mathcal{D}$-coring.
\end{lemma}
\begin{proof}
It may be verified that $(h \otimes C) \otimes_\mathcal{D} (h \otimes C) \cong h \otimes C \otimes C$. This gives us morphisms
\begin{equation}\label{4.4xc}
\begin{array}{c}
id_h \otimes \Delta_C: h \otimes C \longrightarrow h \otimes C \otimes C \cong (h \otimes C) \otimes_\mathcal{D} (h \otimes C)\\
id_h \otimes \varepsilon_C: h \otimes C \longrightarrow h
\end{array}
\end{equation}
in ${_\mathcal{D}}\mathscr{M}_{\mathcal{D}}$. Using the coassociativity and counitality of the $K$-coalgebra $C$, it may be verified   that  $id_h \otimes \Delta_C$ and $id_h \otimes \varepsilon_C$ satisfy the coassociativity and counit axioms in the category ${_\mathcal{D}}\mathscr{M}_{\mathcal{D}}$. Thus, $h \otimes C$ is a coalgebra object in ${_\mathcal{D}}\mathscr{M}_{\mathcal{D}}$.
\end{proof}

\begin{prop}\label{comodovercoring}
Let $(\mathcal{D},C,\psi)$ be a right-right entwining structure. Then, the category $\mathscr{M}(\psi)_\mathcal{D}^C$ of entwined modules is identical to the category $Comod$-$(h \otimes C)$. 
\end{prop}

\begin{proof}
Let $\mathcal{M}\in \mathscr{M}(\psi)_\mathcal{D}^C$. It may be verified that $\mathcal{M}\otimes C \cong \mathcal{M}\otimes_ \mathcal{D}(h\otimes C)$ as right $\mathcal{D}$-modules.  Then, 
by Lemma \ref{lem 6.2}, $\mathcal{M} \otimes C \in  \mathscr{M}(\psi)_\mathcal{D}^C$ and we have 
\begin{equation*}
\rho_{\mathcal{M}(X)}(\mathcal{M}(f)(m))= \mathcal{M}(f_\psi)(m_{0}) \otimes  {m_1}^\psi =(\mathcal{M}\otimes C)(f)(m_0\otimes m_1)
\end{equation*}
for any $f\in Hom_\mathcal{D}(X,Y)$ and $m\in \mathcal{M}(Y)$. We thus obtain a morphism $\rho_\mathcal{M}: \mathcal{M}\longrightarrow \mathcal{M}\otimes C\cong \mathcal{M}\otimes_ \mathcal{D}(h\otimes C)$ of right $\mathcal{D}$-modules given by $\rho_\mathcal{M}(X):= \rho_{\mathcal{M}(X)}$ for each $X\in Ob(\mathcal{D})$. 

Applying \eqref{4.4xc}, we have
\begin{equation} \label{4/5xc}
\begin{CD}
\mathcal M\otimes_{\mathcal D}(h\otimes C) @>id_{\mathcal M}\otimes_{\mathcal D}\Delta_{(h\otimes C)}= id_\mathcal{M}\otimes id_{h}\otimes \Delta_C>> \mathcal M\otimes_{\mathcal D}(h\otimes C) \otimes_{\mathcal D}(h\otimes C)\\
@V\cong VV  @V\cong VV \\
 \mathcal M \otimes C@> id_\mathcal{M}\otimes \Delta_C>> \mathcal M\otimes C\otimes C
\end{CD}
\end{equation} and 
\begin{equation}\label{4/6xc}
\begin{CD}
\mathcal M\otimes_{\mathcal D}(h\otimes C) @>id_\mathcal{M}\otimes \varepsilon_{(h\otimes C)} = id_\mathcal{M}\otimes id_{h}\otimes \varepsilon_C>>\mathcal M\otimes_{\mathcal D}h \\
@V\cong VV @V\cong VV\\
\mathcal M\otimes C@> id_\mathcal{M}\otimes \varepsilon_C>> \mathcal M\\
\end{CD}
\end{equation} The conditions in \eqref{con1} now  follow from the fact that $\rho_{\mathcal{M}(X)}$ is a $C$-coaction for each $X \in Ob(\mathcal{D})$. 
Therefore, $\mathcal{M}$ is a right $(h \otimes C)$-comodule.

\smallskip
Conversely, let $\mathcal{N} \in Comod$-$(h \otimes C)$. Then, $\mathcal{N}$ is a right $\mathcal{D}$-module with a given morphism $\rho_\mathcal{N}:\mathcal{N} \longrightarrow \mathcal{N} \otimes_\mathcal{D} (h\otimes C)\cong \mathcal{N}\otimes C$ of right $\mathcal{D}$-modules satisfying the conditions in \eqref{con1}. Thus, for each $Y\in Ob(\mathcal{D})$, we have a morphism $\rho_\mathcal{N}(Y): \mathcal{N}(Y) \longrightarrow \mathcal{N}(Y)\otimes C$ which satisfies 
\begin{equation}\label{4/7xc}  (id_{\mathcal{N}(Y)} \otimes \Delta_C)\circ \rho_{\mathcal{N}}(Y)=(\rho_{\mathcal{N}}(Y) \otimes id_{C}) \circ \rho_{\mathcal{N}}(Y)\qquad (id_{\mathcal{N}}(Y) \otimes \varepsilon_C)\circ \rho_{\mathcal{N}}(Y)= id_{\mathcal{N}(Y)}
\end{equation}
In \eqref{4/7xc}, we have identified $id_\mathcal{N}\otimes \Delta_C = id_\mathcal{N}\otimes \Delta_{h \otimes C}$ and $id_\mathcal{N}\otimes \varepsilon_C= id_\mathcal{N}\otimes \varepsilon_{(h\otimes C)}$ as in \eqref{4/5xc} and \eqref{4/6xc} respectively. Therefore, $\rho_\mathcal{N}(Y)$ defines a right $C$-comodule structure on $\mathcal{N}(Y)$ for every $Y\in Ob(\mathcal{D})$. Since $\rho_{\mathcal{N}}$ is a morphism of right $\mathcal{D}$-modules, we also have 
\begin{equation} \rho_{\mathcal{N}}(X)(\mathcal{N}(f)(n))=(\mathcal{N}\otimes C)(f)(n_0\otimes n_1)=\mathcal{N}(f_\psi)(n_{0}) \otimes  {n_1}^\psi
\end{equation} for any $f\in Hom_\mathcal{D}(X,Y)$ and $n\in \mathcal{N}(Y)$. Therefore, $\mathcal{N} \in \mathscr{M}(\psi)_\mathcal{D}^C$.
\end{proof}

\begin{lemma}
Let $i: \mathcal{E}\longrightarrow\mathcal{D}$ be an inclusion of small $K$-linear categories. Then, the functor $h \otimes_\mathcal{E} h: \mathcal{E}^{op} \otimes \mathcal{E} \longrightarrow Vect_K$ is a $\mathcal{D}$-coring, where $h$ is the $\mathcal{D}$-$\mathcal{D}$-bimodule as in \eqref{h_D}.
\end{lemma}
\begin{proof}
It is immediate that the functor $h \otimes_\mathcal{E} h$ is a $\mathcal{D}$-$\mathcal{D}$-bimodule. We need to show that $h \otimes_\mathcal{E} h$ is a coalgebra
object in ${_\mathcal{D}}\mathscr{M}_{\mathcal{D}}$. We now define $\Delta: h \otimes_\mathcal{E} h \longrightarrow (h \otimes_\mathcal{E} h) \otimes_{\mathcal{D}} (h \otimes_\mathcal{E} h) \cong  (h \otimes_\mathcal{E} h )\otimes_\mathcal{E} h$ as follows: for $(X,Y)\in Ob(\mathcal D^{op}\otimes \mathcal D)$, we set 
\begin{equation}\label{deltaXY}
\begin{array}{c}
\Delta(X,Y): {\bf h}_Y \otimes_\mathcal{E} {_X}{\bf h} \longrightarrow (h\otimes_\mathcal{E} h)({-},Y)\otimes_\mathcal{E} h(X,{-}) \cong {\bf h}_Y \otimes_\mathcal{E} h \otimes_\mathcal{E} {_X}{\bf h}\\
f \otimes f' \mapsto f \otimes id_Z \otimes f'
\end{array}
\end{equation}
for any $f \otimes f' \in {\bf h}_Y(Z) \otimes {_X}{\bf h}(Z)$ and $Z \in Ob(\mathcal{E})$. It is easy to check that $\Delta(X,Y)$ is well-defined. Also, it can be verified that for any morphism $(\phi',\phi''): (X,Y) \longrightarrow (X',Y')$ in $\mathcal{D}^{op} \otimes \mathcal{D}$, the following diagram commutes:
$$\begin{CD}
{\bf h}_Y \otimes_\mathcal{E} {_X}{\bf h}  @>\Delta(X,Y)>> {\bf h}_Y \otimes_\mathcal{E} h \otimes_\mathcal{E} {_X}{\bf h} \\
@Vh_{\phi''} \otimes_\mathcal{E} {_{\phi'}}h VV        @VVh_{\phi''} \otimes_\mathcal{E} id_h \otimes_\mathcal{E} {_{\phi'}}h V\\
{\bf h}_{Y'} \otimes_\mathcal{E} {_{X'}}{\bf h}  @>\Delta(X',Y')>> {\bf h}_{Y'} \otimes_\mathcal{E} h \otimes_\mathcal{E} {_{X'}}{\bf h}
\end{CD}$$
Thus, $\Delta$ is a morphism of $\mathcal{D}$-$\mathcal{D}$-bimodules.
The map $\varepsilon:h \otimes_\mathcal{E} h \longrightarrow h$ is defined by composition. It may be verified that $\Delta$ and $\varepsilon$ satisfy the coassociativity and counit axioms respectively.
\end{proof}
Let $\mathcal{D}$ be a small $K$-linear category and let $C$ be a $K$-coalgebra. We consider the category $_\mathcal{D}\mathscr{M}^C$ of left-right Doi-Hopf modules (compare Example \ref{2.3fed}). Explicitly, an object in $_\mathcal{D}\mathscr{M}^C$ consists of a left $\mathcal{D}$-module $\mathcal{M}$ with a given right $C$-comodule structure on $\mathcal{M}(X)$ for each $X \in Ob(\mathcal{D})$ such that the following compatibility condition holds:
\begin{equation*}
 \big({\mathcal{M}(f)(m)}\big)_0 \otimes \big({\mathcal{M}(f)(m)}\big)_{1}= \mathcal{M}(f)(m_{0}) \otimes  m_1
\end{equation*} for each $f\in Hom_\mathcal{D}(X,Y)$ and $m\in \mathcal{M}(X)$. 
A morphism $\eta:\mathcal M \longrightarrow \mathcal{N}$ in  $_\mathcal{D}\mathscr{M}^C$ is a  left $\mathcal{D}$-module morphism such that each $\eta(X):\mathcal{M}(X) \longrightarrow \mathcal{N}(X)$ is right $C$-colinear.
By definition, $(h \otimes C)(X,-)= {_X}{\bf h} \otimes C$ is a left $\mathcal{D}$-module for each $X \in Ob(\mathcal{D})$. The map  $id \otimes \Delta_C:Hom_\mathcal{D}(X,Y)\otimes C \longrightarrow Hom_\mathcal{D}(X,Y)\otimes C \otimes C$ gives a right $C$-comodule structure on $({_X{\bf h}}\otimes C)(Y)$ for each $Y \in Ob(\mathcal{D})$. Clearly, $_X{\bf h} \otimes C \in {_\mathcal{D}}\mathscr{M}^C$.

\smallskip
From this point onwards, we suppose additionally that each $Hom_{\mathcal D}(X,Y)$ has a given right $C$-comodule structure denoted by
\begin{equation*}\rho_{XY}: Hom_\mathcal{D}(X,Y)\longrightarrow Hom_\mathcal{D}(X,Y)\otimes C 
\end{equation*}

\begin{definition}\label{coinv*}
 Let $\mathcal{E}\subseteq \mathcal D$ be the subcategory  with $Ob(\mathcal{E})=Ob(\mathcal{D})$ and 
 \begin{equation*}
 \begin{array}{ll}
 Hom_\mathcal{E}(X,Y)&=Hom^C_{Mod\text{-}\mathcal{D}}({{\bf h}_X},{{\bf h}_Y})=\{\mbox{ $\eta\in Hom_{Mod\text{-}\mathcal{D}}({\bf h}_X,{\bf h}_Y)$ $\vert$ $\eta$ is objectwise  $C$-colinear}\}\\
&=\{\mbox{$g\in Hom_{\mathcal D}(X,Y)$ $\vert$ $\rho_{ZY}(gf)=({_Z{\bf h}}\otimes C)(g)(\rho_{ZX}(f))$ $\textrm{ }\forall\textrm{ }f\in Hom_{\mathcal D}(Z,X)$ }\} \\ 
\end{array}
 \end{equation*} We will say that $\mathcal E$ is the subcategory of $C$-coinvariants of $\mathcal D$.
 \end{definition}
 
 \begin{example}
Let $H$ be a Hopf algebra over $K$ and let $\mathcal{D}$ be a right co-$H$-category. In this case, the subcategory $\mathcal E$ of $H$-coinvariants of $\mathcal D$ is given by setting $Ob(\mathcal{E})=Ob(\mathcal{D})$ and $Hom_\mathcal{E}(X,Y)=Hom_\mathcal{D}(X,Y)^{coH}$. 
\end{example}
 
 \smallskip It follows that the right $C$-comodule structures $\rho_{XY}: Hom_{\mathcal{D}}(X,Y) \longrightarrow Hom_{\mathcal{D}}(X,Y) \otimes C$   induce a morphism ${_X{\bf h}}\longrightarrow {_X{\bf h}}\otimes C $ of left $\mathcal{E}$-modules for each $X\in Ob(\mathcal{D})$.  Further, for every $Y \in Ob(\mathcal{D})$, this induces a morphism
\begin{equation}\label{er4.9}
\begin{array}{c}
(h\otimes_{\mathcal{E}} {_X{\bf h}})(Y)= {\bf h}_Y \otimes_{\mathcal{E}} {_X{\bf h}} \longrightarrow {\bf h}_Y \otimes_\mathcal{E} {_X{\bf h}} \otimes C = (h\otimes_{\mathcal{E}} {_X{\bf h}})(Y) \otimes C\\
f\otimes f' \mapsto f\otimes \rho_{XZ}(f')\\
\end{array}
\end{equation}
where $f\in Hom_\mathcal{D}(Z,Y),~f'\in Hom_\mathcal{D}(X,Z)$ and $Z\in Ob(\mathcal{E})$.  It may be easily verified that the coaction in \eqref{er4.9} makes $h\otimes_{\mathcal{E}} {_X{\bf h}}$ an object of  ${_\mathcal{D}}\mathscr{M}^C$.

\medskip

We obtain therefore canonical morphisms of $K$-vector spaces given by the following composition
\begin{equation*} \left\{can_{XY}: {{\bf h}_Y}\otimes_\mathcal{E} {_X{\bf h}}\longrightarrow {{\bf h}_Y}\otimes_\mathcal{E}({_X{\bf h}}\otimes C)\longrightarrow {{\bf h}_Y}\otimes_\mathcal{D}({_X{\bf h}}\otimes C)\cong Hom_\mathcal{D}(X,Y) \otimes C\right\}_{(X,Y)\in Ob(\mathcal D)^2}
\end{equation*} For each $X\in Ob(\mathcal D)$, this induces a morphism in ${_\mathcal{D}}\mathscr{M}^C$ as follows
\begin{equation*}
can_X:h\otimes_{\mathcal E}{_X{\bf h}}\longrightarrow {_X{\bf h}}\otimes C \qquad can_X(Y):=can_{XY}
\end{equation*}

\begin{definition}\label{4.5tr}
Let $C$ be a $K$-coalgebra and $\mathcal{D}$ be a small $K$-linear category such that $Hom_\mathcal{D}(X,Y)$ has a right $C$-comodule structure for every $X,Y \in Ob(\mathcal{D})$. Let $\mathcal{E}$ be a $K$-linear subcategory of $\mathcal{D}$. Then, $\mathcal{D}$ is called a $C$-Galois extension of $\mathcal{E}$ if 
\begin{itemize}
\item[(i)] $Ob(\mathcal{E})=Ob(\mathcal{D})$ and $Hom_\mathcal{E}(X,Y)=Hom^C_{Mod\text{-}\mathcal{D}}({{\bf h}_X},{{\bf h}_Y})$. 
\item[(ii)] The induced canonical morphism $can_X:h\otimes_{\mathcal{E}} {_X{\bf h}}\longrightarrow {_X{\bf h}}\otimes C$ is an isomorphism in $_\mathcal{D}\mathscr{M}^C$ for each $X\in Ob(\mathcal{D})$.
\end{itemize}
\end{definition}

Let $\mathcal{D}$ be a $C$-Galois extension of $\mathcal{E}$. For each $X\in Ob(\mathcal{D})$, we define 
\begin{equation}\label{transl}
\tau_X: C\longrightarrow {{\bf h}_X}\otimes_\mathcal{E} {_X{\bf h}}\qquad \tau_X(c):=can_{XX}^{-1}(id_X \otimes c)
\end{equation} We refer to these as the translation maps of the Galois extension. 

\begin{lemma}\label{trans}
Let $\mathcal{D}$ be a $C$-Galois extension of $\mathcal{E}$. Let $\{\tau_X: C\longrightarrow {{\bf h}_X}\otimes_\mathcal{E} {_X{\bf h}}\}_{X \in Ob(\mathcal{D})}$ be the associated translation maps. We use the notation $\tau_X(c)= c^{(1)}\otimes c^{(2)}$ (summation omitted). Then,\\
(i) $\tau_X$ is right $C$-colinear i.e., $c^{(1)}\otimes {c^{(2)}}_0\otimes {c^{(2)}}_1= (c_1)^{(1)}\otimes (c_1)^{(2)}\otimes c_2$.\smallskip
\newline (ii) For any $f\in Hom_\mathcal{D}(X,Y)$, we have $f_0(f_1)^{(1)} \otimes (f_1)^{(2)} = id_Y \otimes f\in {{\bf h}_Y}\otimes_\mathcal{E} {_X{\bf h}}$.\smallskip
\newline (iii) $c^{(1)}c^{(2)}=\varepsilon_C(c) \cdot id_X$.
\end{lemma}
\begin{proof}
The $C$-colinearity of $\tau_X$ follows from the $C$-colinearity of $can_{XX}^{-1}$. Explicitly, for any $c\in C$, we have
\begin{align*}
c^{(1)}\otimes {c^{(2)}}_0\otimes {c^{(2)}}_1 &= (id \otimes \rho)\tau_X(c)= (id \otimes \rho)can_{XX}^{-1}(id_X\otimes c)\\
&= (can_{XX}^{-1} \otimes id_C)(id\otimes \Delta_C)(id_X\otimes c)\\
&= (can_{XX}^{-1} \otimes id_C)(id_X\otimes c_1\otimes c_2)\\
&= can_{XX}^{-1}(id_X\otimes c_1)\otimes c_2 = \tau_X(c_1)\otimes c_2= (c_1)^{(1)}\otimes (c_1)^{(2)}\otimes c_2
\end{align*}
This proves (i). Since $can^{-1}_{X}:{_X{\bf h}}\otimes C \longrightarrow h\otimes_{\mathcal{E}} {_X{\bf h}}$ is a morphism of left $\mathcal{D}$-modules for each $X \in Ob(\mathcal{D})$, we also have
\begin{align*}
f_0(f_1)^{(1)} \otimes (f_1)^{(2)}&=(h \otimes_\mathcal{E} {_X{\bf h}})(f_0)\left(\tau_X(f_1)\right)\\
&=(h \otimes_\mathcal{E} {_X{\bf h}})(f_0)\left(can_{XX}^{-1}(id_X \otimes f_1)\right)\\
&= can^{-1}_{XY}\left(({_X{\bf h}} \otimes C)(f_0)(id_X \otimes f_1)\right)\\
&= can^{-1}_{XY}(f_0 \otimes f_1)=id_Y \otimes f
\end{align*}
This proves (ii). Again using the definition of $can_{XX}$ and $\tau_X$, we have $(can_{XX} \circ \tau_X )(c)=id_X \otimes c$. Thus,
\begin{equation*}
can_{XX}(c^{(1)} \otimes c^{(2)})= c^{(1)}{c^{(2)}}_0 \otimes {c^{(2)}}_1=id_X \otimes c
\end{equation*}
Now,  by applying the map $id \otimes \varepsilon_C$ to both sides, we get (iii).
\end{proof}

\begin{theorem}\label{Gal-ent}
Let $\mathcal{D}$ be a $C$-Galois extension of $\mathcal{E}$. We denote by 
$\rho_{XY}:Hom_\mathcal{D}(X,Y)\longrightarrow Hom_\mathcal{D}(X,Y)\otimes C$ the right $C$-comodule structure maps.
Then, 
there exists a unique right-right entwining structure $(\mathcal{D},C,\psi)$ which makes ${\bf h}_Y$ an object in ${\mathscr{M}(\psi)}_\mathcal{D}^C$ for every $Y\in Ob(\mathcal{D})$ with its canonical $\mathcal{D}$-module structure and right $C$-coactions $\{\rho_{XY}\}_{X\in Ob(\mathcal{D})}$.

\smallskip This entwining structure $(\mathcal D,C,\psi)$ is given by
$$\psi_{XY}:C\otimes Hom_\mathcal{D}(X,Y)\xrightarrow {\tau_Y\otimes id} {{\bf h}_Y}\otimes_\mathcal{E} {_Y{\bf h}}\otimes Hom_\mathcal{D}(X,Y)\longrightarrow {{\bf h}_Y}\otimes_\mathcal{E} {_X{\bf h}}\xrightarrow{can_{XY}}Hom_\mathcal{D}(X,Y)\otimes C$$
\end{theorem}

\begin{proof}
Using Lemma \ref{trans}, the proof will follow essentially in the same way as that of \cite[Theorem 2.7]{BH}.
\end{proof}

\begin{lemma}\label{L4.19cop}
Let $\mathcal{D}$ be a $C$-Galois extension of $\mathcal{E}$. Then, $h \otimes_\mathcal{E} h \cong h \otimes C$ as $\mathcal{D}$-corings.
\end{lemma}
\begin{proof}
We define $can:h \otimes_\mathcal{E} h \longrightarrow h \otimes C$ by setting $can(X,Y):=can_{XY}$ for each $(X,Y) \in Ob(\mathcal{D}^{op} \otimes \mathcal{D})$. We first verify that $can$ is a morphism of $\mathcal{D}$-$\mathcal{D}$-bimodules. Clearly, $can(X,-)=can_X$ which, by definition, is a morphism of left $\mathcal{D}$-modules. Therefore, it suffices to show that $can(-,Y)$ is a morphism of right $\mathcal{D}$-modules, i.e., the following diagram commutes for any $g \in Hom_\mathcal{D}(Z,Z')$:
$$\begin{CD}
({\bf h}_Y \otimes_\mathcal{E} h)(Z')  @>can(Z',Y)>> ( {\bf h}_Y \otimes C)(Z') \\
@V({\bf h}_Y \otimes_\mathcal{E} h)(g)  VV        @VV( {\bf h}_Y \otimes C)(g) V\\
({\bf h}_Y \otimes_\mathcal{E} h)(Z)   @>can(Z,Y)>> ( {\bf h}_Y \otimes C)(Z)
\end{CD}$$
By Theorem \ref{Gal-ent}, we know that ${\bf h}_W$ is an object in  ${\mathscr{M}(\psi)}_\mathcal{D}^C$ for each $W\in Ob(\mathcal D)$. Thus, for any $f \in {\bf h}_W(Z')$, we have
\begin{equation*}
(fg)_0\otimes (fg)_1=\rho_{ZW}(fg)=\rho_{ZW}({\bf h}_W(g)(f))=({\bf h}_W \otimes C)(g)(f_0 \otimes f_1)=f_0g_\psi \otimes {f_1}^\psi
\end{equation*}
Therefore, for any $f' \otimes f \in {\bf h}_Y(W) \otimes_{\mathcal E} {_{Z'}}{\bf h}(W)$, we obtain
\begin{equation*}
\begin{array}{ll}
can(Z,Y)\left(({\bf h}_Y \otimes_\mathcal{E} h)(g)(f' \otimes f)\right)&=can(Z,Y)(f' \otimes fg)=f'\circ (fg)_0 \otimes (fg)_1=f'f_0g_\psi \otimes {f_1}^\psi\\
&=( {\bf h}_Y \otimes C)(g)\left(can(Z',Y)(f' \otimes f)\right)
\end{array}
\end{equation*}
It remains to verify that $can$ is also a coalgebra morphism. First, we show that  the following diagram commutes:
$$\begin{CD}
h \otimes_\mathcal{E} h @>can>> h \otimes C\\
@V\Delta_{h \otimes_\mathcal{E} h}  VV        @VV\Delta_{h \otimes C} V\\
(h \otimes_\mathcal{E} h) \otimes_\mathcal{D} (h \otimes_\mathcal{E} h)  @>can \otimes_\mathcal{D} can>> (h \otimes C) \otimes_\mathcal{D} (h \otimes C)
\end{CD}$$
For any $(X,Y) \in Ob(\mathcal{D}^{op} \otimes \mathcal{D})$ and $w \otimes w' \in  {\bf h}_Y(W) \otimes {_{X}}{\bf h}(W)$, we have
\begin{equation*}
\begin{array}{ll}
\Delta_{h \otimes C}(X,Y)\left(can_{XY}(w \otimes w')\right)&=\Delta_{h \otimes C}(X,Y)(ww'_0 \otimes w'_1)=(ww'_0 \otimes w'_{11}) \otimes_\mathcal{D} (id_X \otimes w'_{12})\\
&=(ww'_{00} \otimes w'_{01}) \otimes_\mathcal{D} (id_X \otimes w'_{1})\\
&=({_X}{\bf h}\otimes C)(w)(\rho_{XW}(w'_0))\otimes_\mathcal{D} (id_X \otimes w'_{1})\\
&=({_X}{\bf h}\otimes C)(w)\left({\bf h}_W \otimes C)(w'_0)(\rho_{WW}(id_W))\right)\otimes_\mathcal{D} (id_X \otimes w'_{1})\\
&=({\bf h}_Y \otimes C)(w'_0)\left(({_W}{\bf h} \otimes C)(w)(\rho_{WW}(id_W))\right) \otimes_\mathcal{D} (id_X \otimes w'_{1})\\
&=\left(({_W}{\bf h} \otimes C)(w)(\rho_{WW}(id_W))\right)\cdot w'_0 \otimes_\mathcal{D} (id_X \otimes w'_{1})\\
&=\left(({_W}{\bf h} \otimes C)(w)(\rho_{WW}(id_W))\right)\otimes_\mathcal{D} (w'_0 \otimes w'_{1})\\
&=(w\circ {id_W}_0 \otimes {id_W}_1)\otimes_\mathcal{D} (w'_0 \otimes w'_{1})\\
&=can_{WY}(w \otimes id_W) \otimes_\mathcal{D} can_{XW}(id_W \otimes w')
\end{array}
\end{equation*} It may be verified easily that $can$ is compatible with counits. 
Since $can$ is a morphism in the category of $\mathcal{D}$-$\mathcal{D}$-bimodules and $can(X,Y)=can_{XY}$ is an isomorphism for each $(X,Y) \in Ob(\mathcal{D}^{op} \otimes \mathcal{D})$, it follows that $can$ is an isomorphism with inverse given by $can^{-1}(X,Y):=can^{-1}_{XY}$. This proves the result.
\end{proof}

\begin{definition}
Let $\mathcal{D}$ be a small $K$-linear category such that $Hom_\mathcal{D}(X,Y)$ is a right $C$-comodule for every $X,Y \in Ob(\mathcal{D})$. Let $\Phi_{XY}:C\longrightarrow Hom_\mathcal{D}(X,Y)$ and $\Phi_{YZ}:C\longrightarrow Hom_\mathcal{D}(Y,Z)$ be two $C$-comodule maps. Then, their convolution product is given by 
$$\Phi_{YZ}*\Phi_{XY}: C\longrightarrow Hom_\mathcal{D}(X,Z),~~~~~~ c\mapsto \Phi_{YZ}(c_1)\circ\Phi_{XY}(c_2)$$  A collection of right $C$-comodule maps $\Phi=\{\Phi_{XY}:C\longrightarrow Hom_\mathcal{D}(X,Y)\}_{X,Y\in Ob(\mathcal{D})}$ is said to be convolution invertible if there exists a collection  $\Phi'=\{\Phi'_{XY}:C\longrightarrow Hom_\mathcal{D}(X,Y)\}_{X,Y\in Ob(\mathcal{D})}$  of $C$-comodule maps such that
$$(\Phi_{XY}*\Phi'_{YX})(c)= \varepsilon_C(c)\cdot id_Y =(\Phi'_{XY}*\Phi_{YX})(c)$$ for every $c\in C$.
\end{definition}

\begin{theorem}\label{th4.11}
Let $C$ be a $K$-coalgebra and $\mathcal{D}$ be a small $K$-linear category such that $Hom_\mathcal{D}(X,Y)$ has a right $C$-comodule structure $\rho_{XY}$ for every $X,Y \in Ob(\mathcal{D})$.  Let $\mathcal{E}$ be the subcategory of $C$-coinvariants of $\mathcal{D}$. If there exists a convolution invertible collection $\Phi=\{\Phi_{XY}:C\longrightarrow Hom_\mathcal{D}(X,Y)\}_{X,Y\in Ob(\mathcal{D})}$ of right
$C$-comodule maps, then the following are equivalent:

(i) $\mathcal{D}$ is a $C$-Galois extension of $\mathcal{E}$.

(ii) There exists a right-right entwining structure $(\mathcal{D},C,\psi)$ such that ${\bf h}_Y$ is an object in ${\mathscr{M}(\psi)}_\mathcal{D}^C$ for every $Y\in Ob(\mathcal{D})$ with its canonical $\mathcal{D}$-module structure and right $C$-coactions $\{\rho_{XY}\}_{X \in Ob(\mathcal{D})}$.

(iii) For any $f\in Hom_\mathcal{D}(X,Y)$, the morphism $f_0\circ\Phi'_{ZX}(f_1) \in Hom_\mathcal{E}(Z, Y)$ for every $Z \in Ob(\mathcal{D})$, where $\Phi'$ is the convolution  inverse of $\Phi$. 
\end{theorem}

\begin{proof}
By Theorem \ref{Gal-ent}, we have $(i)\Rightarrow (ii)$. To prove $(ii)\Rightarrow (iii)$, we will use the equality 
\begin{equation}\label{4.9jg}
(_X{\bf h}\otimes C)\left(\Phi'_{XY}(c)\right)(\rho_{XX}(id_X))=\psi_{XY}(c_1\otimes \Phi'_{XY}(c_2))
\end{equation}
 for any $c\in C$. We first give a proof of this. Since ${\bf h}_Y\in {\mathscr{M}(\psi)}_\mathcal{D}^C$, we have  
\begin{equation}\label{eq4.6}
\begin{array}{ll}
\rho_{XY}(f)=\rho_{XY}({\bf h}_Y(f)(id_Y))&={\bf h}_Y(f_\psi)({id_Y}_0)\otimes {{id_Y}_1}^\psi\\ & ={id_Y}_0f_\psi \otimes  {{id_Y}_1}^\psi=(_X{\bf{h}}\otimes C)({id_Y}_0)(\psi_{XY}({id_Y}_1\otimes f))\\
\end{array}
\end{equation}
for any $f\in Hom_\mathcal{D}(X,Y)$. Also, for any $c\in C$, we have 
\begin{equation}\label{eq.4.7}
\begin{array}{ll}
(_X{\bf{h}}\otimes C)({id_X}_0)(\psi_{XX}({id_X}_1\otimes \Phi_{YX}(c_1)\Phi'_{XY}(c_2)))&=(_X{\bf{h}}\otimes C)({id_X}_0)(\psi_{XX}({id_X}_1\otimes \varepsilon_C(c)id_X))\\ 
&=\varepsilon_C(c){id_X}_0\otimes {id_X}_1\\
\end{array}
\end{equation}
Now, using \eqref{eq.4.7}, we have
\begin{equation*}
\begin{array}{ll}
&(_X{\bf h}\otimes C)\left(\Phi'_{XY}(c)\right)(\rho_{XX}(id_X))\\
&= (_X{\bf h}\otimes C)\left(\Phi'_{XY}(c_1)\right)(\varepsilon_C(c_2){id_X}_0\otimes {id_X}_1)\\
&=(_X{\bf h}\otimes C)\left(\Phi'_{XY}(c_1)\right)\left((_X{\bf{h}}\otimes C)({id_X}_0)(\psi_{XX}({id_X}_1\otimes \Phi_{YX}(c_2)\Phi'_{XY}(c_3)))\right)~~~~~~~~~~~~~~~~~ (\textnormal{using~} \eqref{eq.4.7})\\
&=(_X{\bf h}\otimes C)\left(\Phi'_{XY}(c_1)\right)\left((_X{\bf{h}}\otimes C)({id_X}_0)\left((\Phi_{YX}(c_2))_\psi(\Phi'_{XY}(c_3))_\psi \otimes {{{id_X}_1}^\psi}^\psi\right)\right)~~~~~~~~~~~~~(\textnormal{using~} \eqref{eq 6.1})\\
&=(_X{\bf h}\otimes C)\left(\Phi'_{XY}(c_1)\right)\left({id_X}_0\circ(\Phi_{YX}(c_2))_\psi(\Phi'_{XY}(c_3))_\psi \otimes {{{id_X}_1}^\psi}^\psi\right)\\
&=(_X{\bf h}\otimes C)\left(\Phi'_{XY}(c_1)\right)\left(({\bf h}_X\otimes C)(\Phi'_{XY}(c_3))({id_X}_0 \circ (\Phi_{YX}(c_2))_\psi\otimes {{id_X}_1}^\psi)\right)\\
&=(_X{\bf h}\otimes C)\left(\Phi'_{XY}(c_1)\right)\left(({\bf h}_X\otimes C)(\Phi'_{XY}(c_3))\Big((_Y{\bf{h}}\otimes C)({id_X}_0)\left(\psi_{YX}\left({id_X}_1\otimes \Phi_{YX}(c_2)\right)\right)\Big)\right)\\
&=(_X{\bf h}\otimes C)\left(\Phi'_{XY}(c_1)\right)\left(({\bf h}_X\otimes C)(\Phi'_{XY}(c_3))\left(\rho_{YX}(\Phi_{YX}(c_2))\right)\right)~~~~~~~~~~~~~~~~~~~~~~~~~~~~~~~~(\textnormal{using~} \eqref{eq4.6})\\
&=(_X{\bf h}\otimes C)\left(\Phi'_{XY}(c_1)\right)\left(({\bf h}_X\otimes C)(\Phi'_{XY}(c_4))(\Phi_{YX}(c_2)\otimes c_3)\right)~~~~~~~~~~~~~~~~~~~~(\textnormal{since~} \Phi_{YX} \textnormal{~is~} C\textnormal{-colinear})\\
&=\Phi'_{XY}(c_1)\Phi_{YX}(c_2){\left(\Phi'_{XY}(c_4)\right)}_\psi\otimes {c_3}^\psi\\
&=\varepsilon_C(c_1)id_Y{\left(\Phi'_{XY}(c_3)\right)}_\psi\otimes {c_2}^\psi\\
&=({\bf h}_Y\otimes C)(\Phi'_{XY}(c_3))(\varepsilon_C(c_1)id_Y\otimes c_2)\\
&=({\bf h}_Y\otimes C)(\Phi'_{XY}(c_2))(id_Y\otimes c_1)\\
&=(\Phi'_{XY}(c_2))_\psi \otimes {c_1}^\psi=\psi_{XY}(c_1\otimes \Phi'_{XY}(c_2))
\end{array}
\end{equation*}
This proves the equality \eqref{4.9jg}. 

\smallskip
For any $f\in Hom_\mathcal{D}(X,Y)$, consider the morphism $f_0\circ\Phi'_{ZX}(f_1):Z \longrightarrow Y$ in $\mathcal{D}$. Then $f_0\circ\Phi'_{ZX}(f_1)$ induces a morphism of right $\mathcal{D}$-modules ${\bf{h}}_Z \longrightarrow {\bf{h}}_Y$ which we denote by $\tilde{f}$.
We now verify that the map $\tilde{f}(X'):{\bf{h}}_Z(X') \longrightarrow {\bf{h}}_Y(X')$ is right $C$-colinear for each $X' \in Ob(\mathcal{D})$. Since 
${\bf{h}}_Y$ is an object in ${\mathscr{M}(\psi)}_\mathcal{D}^C$ for every $Y\in Ob(\mathcal{D})$, the following diagram commutes for any $g \in Hom_\mathcal{D}(X',Z)$:
\begin{equation}\label{diag@}
\begin{CD}
{\bf h}_Y(Z) @>\rho_{XY}>> {\bf h}_Y(Z) \otimes C \\
@ V{\bf h}_Y(g)VV        @VV ({\bf h}_Y \otimes C)(g) V\\
{\bf h}_Y(X') @>\rho_{X'Y}>> {\bf h}_Y(X') \otimes C
\end{CD}
\end{equation}
Thus, we have
\begin{equation*}
\begin{array}{lll}
\rho_{X'Y}(\tilde{f}(X')(g))&=\rho_{X'Y}\left(f_0\circ\Phi'_{ZX}(f_1)\circ g\right)=\rho_{X'Y}\left({\bf{h}}_Y(\Phi'_{ZX}(f_1)\circ g)(f_0)\right)\\
&=({\bf{h}}_Y \otimes C)(\Phi'_{ZX}(f_1) \circ g)(\rho_{XY}(f_0))& (\text{using}~ \eqref{diag@})\\
&=({\bf{h}}_Y \otimes C)(g)(({\bf{h}}_Y \otimes C)(\Phi'_{ZX}(f_1))(f_{00} \otimes f_{01}))\\
&=({\bf{h}}_Y \otimes C)(g)\left(f_{0}{\left(\Phi'_{ZX}(f_{12})\right)}_\psi \otimes {f_{11}}^\psi\right)\\
&=({\bf{h}}_Y \otimes C)(g)\big((_Z{\bf h}\otimes C)(f_0)\left(\psi_{ZX}(f_{11} \otimes \Phi'_{ZX}(f_{12}))\right)\big)\\
&=({\bf{h}}_Y \otimes C)(g)\big((_Z{\bf h}\otimes C)(f_0)(_Z{\bf h}\otimes C)\left((\Phi'_{ZX})(f_1)\right)(\rho_{ZZ}(id_Z))\big)& (\text{using } \eqref{4.9jg})\\
&=({\bf{h}}_Y \otimes C)(g)\left((_Z{\bf h}\otimes C)(f_0 \circ \Phi'_{ZX}(f_1))(\rho_{ZZ}(id_Z))\right)\\
&=(h \otimes C)\left(g,f_0 \circ \Phi'_{ZX}(f_1)\right)(\rho_{ZZ}(id_Z))\\
&=(_{X'}{\bf h}\otimes C)\left(f_0 \circ \Phi'_{ZX}(f_1)\right)\left(({\bf{h}}_Z \otimes C)(g)(\rho_{ZZ}(id_Z))\right)\\
&=(\tilde{f}(X') \otimes id_C)\left(({\bf{h}}_Z \otimes C)(g)(\rho_{ZZ}(id_Z))\right)\\
&=(\tilde{f}(X') \otimes id_C)(\rho_{X'Z}(g))
\end{array}
\end{equation*}
Therefore, $\tilde{f} \in Hom^C_{Mod\text{-}\mathcal{D}}({{\bf h}_Z},{{\bf h}_Y})=Hom_\mathcal{E}(Z,Y)$.

\smallskip
For $(iii) \Rightarrow (i)$, we start by showing that $can_{XY}:{\bf h}_Y \otimes_{\mathcal{E}} {_X{\bf h}}\longrightarrow Hom_{\mathcal{D}}(X,Y) \otimes C$ is an isomorphism for each $X,Y \in Ob(\mathcal{D})$. We define $can^{-1}_{XY}:Hom_{\mathcal{D}}(X,Y) \otimes C \longrightarrow {\bf h}_Y \otimes_{\mathcal{E}} {_X{\bf h}}$ by 
\begin{equation}\label{lsteq}
can^{-1}_{XY}(f \otimes c):= f \circ \Phi'_{YX}(c_1) \otimes_\mathcal{E} \Phi_{XY}(c_2)\in {\bf h}_Y \otimes_{\mathcal{E}} {_X{\bf h}}
\end{equation}
for any $f \in Hom_{\mathcal{D}}(X,Y)$ and $c \in C$. Then, using the $C$-colinearity of $\Phi_{XY}$, we have
\begin{equation*}
(can_{XY} \circ can^{-1}_{XY})(f \otimes c)=f \circ \Phi'_{YX}(c_1) \circ \left( \Phi_{XY}(c_2)\right)_0 \otimes  \left(\Phi_{XY}(c_2)\right)_1=f \circ \Phi'_{YX}(c_1) \circ \Phi_{XY}(c_2) \otimes c_3=f \otimes c
\end{equation*}
On the other hand,  by assumption, we obtain
\begin{equation*}
(can^{-1}_{XY} \circ can_{XY})(g \otimes_\mathcal{E} g')=gg'_0\Phi'_{YX}(g'_{11}) \otimes_\mathcal{E} \Phi_{XY}(g'_{12})=g \otimes_\mathcal{E} g'_0\Phi'_{YX}(g'_{11}) \Phi_{XY}(g'_{12})=g \otimes_\mathcal{E} g'
\end{equation*}
for any $g \otimes_\mathcal{E} g' \in {\bf h}_Y \otimes_{\mathcal{E}} {_X{\bf h}}$. From the definition in \eqref{lsteq}, it is clear that setting $can_X^{-1}(Y):=can_{XY}^{-1}$
for each $Y\in Ob(\mathcal D)$ determines a morphism in $_\mathcal{D}\mathscr{M}^C$ which is inverse to $can_X$. This completes the proof.
\end{proof}

\begin{example}
Let $H$ be a Hopf algebra over $K$. If $\mathcal{C}$ is a left $H$-module category, then  the smash product category $\mathcal{C} \# H$ (see \cite{CiSo}) is a right co-$H$-category with the right $H$-coaction determined by $f \# h \mapsto f \# h_1 \otimes h_2$ on each $Hom_{\mathcal{C} \# H}(X,Y)=Hom_{\mathcal C}(X,Y)\otimes H$. By definition, we know that  $Ob(\mathcal{C})=Ob(\mathcal{C} \# H)$. It is easy to see that $Hom_\mathcal{C}(X,Y)=Hom_{\mathcal{C} \#H}(X,Y)^{coH}$. 

\smallskip We claim that 
$\mathcal C\# H$ is an $H$-Galois extension of $\mathcal C$.  We first observe that for any $f\#h \in Hom_{\mathcal{C}\#H}(Z,Y)$ and $f'\#h' \in Hom_{\mathcal{C}\#H}(X,Z)$, we have 
\begin{equation*}
(f\#h) \otimes_\mathcal{C} (f'\#h')=(f\#h)(f'\#1_H) \otimes_\mathcal{C} (id_X \#h')
\end{equation*}
Thus, $can_{XY}:{\bf h}_Y \otimes_{\mathcal{C}} {_X{\bf h}} \longrightarrow Hom_{\mathcal{C}\#H}(X,Y) \otimes H$ has the following form
\begin{equation*}
can_{XY}\left((f\#h) \otimes_\mathcal{C} (f'\#h')\right)=(f\#h)(f'\#1_H)(id_X \#h_1') \otimes h_2'
\end{equation*}
 for each $X,Y \in Ob(\mathcal{C}\#H)$,
Then, it may be verified that for each $X,Y \in Ob(\mathcal{C}\#H)$, $can_{XY}$ is an isomorphism with inverse  $can_{XY}^{-1}:Hom_{\mathcal{C}\#H}(X,Y) \otimes H\longrightarrow {\bf h}_Y \otimes_{\mathcal{C}} {_X{\bf h}} $ determined by
\begin{equation*}
can^{-1}_{XY}\left((g\#k) \otimes k'\right):=(g\#k)(id_X \#S(k_1')) \otimes_\mathcal{C} (id_X \#k_2')
\end{equation*}
\end{example}

\begin{prop}\label{DEC}
Let $\mathcal{D}$ be a $C$-Galois extension of $\mathcal{E}$. If there exists a convolution invertible collection $\Phi=\{\Phi_{XY}:C\longrightarrow Hom_\mathcal{D}(X,Y)\}_{X,Y\in Ob(\mathcal{D})}$ of right
$C$-comodule maps, then 
\begin{equation*}
Hom_\mathcal{D}(X,-)\cong Hom_\mathcal{E}(X,-) \otimes C \in {_\mathcal{E}}\mathscr{M}^C 
\end{equation*}
for each $X \in Ob(\mathcal{E})=Ob(\mathcal{D})$.
\end{prop}
\begin{proof}
Let $\Phi'$ be the convolution inverse of $\Phi$. Given $f\in Hom_{\mathcal D}(X,Y)$, it follows from Theorem \ref{th4.11} that $f_0 \circ \Phi'_{ZX}(f_1) \in Hom_\mathcal{E}(Z,Y)$ for every $Z \in Ob(\mathcal{D})$. We define 
\begin{equation*}
\eta:Hom_\mathcal{D}(X,-)\longrightarrow Hom_\mathcal{E}(X,-) \otimes C \qquad \eta(Y)(f):=f_0 \circ \Phi'_{XX}(f_1) \otimes f_2
\end{equation*}
Using Definition \ref{coinv*}, we see that $\rho_{XY'}(gf)=gf_0 \otimes f_1$ for any $g \in Hom_\mathcal{E}(Y,Y')$. Hence, we have
\begin{equation}\label{eq4.16l}
(gf)_0 \otimes (gf)_1 \otimes (gf)_2 =(id \otimes \Delta_C)((gf)_0 \otimes (gf)_1)=(\rho_{XY'} \otimes id_C)(gf_0 \otimes f_1)=gf_0 \otimes f_1 \otimes f_2
\end{equation}
Using \eqref{eq4.16l}, it may be easily seen that $\eta$ is a morphism of left $\mathcal{E}$-modules. Using the coassociativity of the $C$-coactions $\{\rho_{XY}\}_{X,Y \in Ob(\mathcal{D})}$, it is also clear that $\eta$ is objectwise $C$-colinear. Therefore, $\eta$ is a morphism in ${_\mathcal{E}}\mathscr{M}^C$. 

Conversely, we define $\zeta:Hom_\mathcal{E}(X,-) \otimes C \longrightarrow Hom_\mathcal{D}(X,-)$ given by $\zeta(Y)(f' \otimes c):=f' \circ \Phi_{XX}(c)$ for $Y \in Ob(\mathcal{E})$. It is immediate that $\zeta$ is a morphism of left $\mathcal{E}$-modules. Moreover,
\begin{equation*}
\rho_{XY}(f' \circ \Phi_{XX}(c))=(f' \circ \Phi_{XX}(c))_0 \otimes (f' \circ \Phi_{XX}(c))_1=f'\circ  (\Phi_{XX}(c))_0 \otimes  (\Phi_{XX}(c))_1=f' \circ  \Phi_{XX}(c_1) \otimes c_2
\end{equation*} where the last equality follows from the fact that $\Phi_{XX}$ is $C$-colinear.
It follows that $\zeta(Y)$ is $C$-colinear for each $Y \in Ob(\mathcal{E})$ and hence $\zeta$ is a morphism in ${_\mathcal{E}}\mathscr{M}^C$. It may be verified that $\zeta$ is the inverse of $\eta$.
\end{proof}

\begin{definition}\label{grouplike}
Let $\mathcal{D}$ be a small $K$-linear category and $\mathcal{E}$ be a $K$-subcategory. Let $(\mathscr{C},\Delta_\mathscr{C},\varepsilon_\mathscr{C})$ be a $\mathcal{D}$-coring. Then, a collection 
\begin{equation*}
G(\mathscr{C},\mathcal{E})=\{s_{X} \in \mathscr{C}(X,X)\}_{X \in Ob(\mathcal{E})}
\end{equation*}
is said to be group-like for  $\mathscr C$ with respect to $\mathcal E$ if 
\begin{itemize}
\item[(i)] $\Delta_\mathscr{C}(X,X)(s_{X})=s_{X} \otimes s_{X}$  and $\varepsilon_\mathscr{C}(s_X)=id_X$ for any $X \in Ob(\mathcal{E})$,
\item[(ii)] For any $f\in Hom_{\mathcal E}(X,Y)$, we have
\begin{equation}\label{glk}
f \cdot s_X=\mathscr{C}(-,f)(X)(s_X)= \mathscr{C}(f,-)(Y)(s_Y)=s_Y \cdot f
\end{equation} 
\end{itemize}
\end{definition}

\begin{example}
(i) If $\mathcal E$ is a subcategory of $\mathcal{D}$, then the collection $\{id_X \otimes id_X \in {\bf h}_X \otimes_\mathcal{E} {_X}{\bf h}\}_{X \in Ob(\mathcal{E})}$ is group-like for $h \otimes_\mathcal{E} h$ with respect to $\mathcal{E}$.

\smallskip
(ii) Let $\mathcal{D}$ be a $C$-Galois extension of $\mathcal{E}$. Then $h \otimes C$ is a $\mathcal{D}$-coring (by Theorem \ref{Gal-ent} and Lemma \ref{entcoring}) and the collection $\{{id_X}_0 \otimes {id_X}_1 \in Hom_\mathcal{D}(X,X) \otimes C\}_{X \in Ob(\mathcal{E})}$ is group-like for $h \otimes C$ with respect to $\mathcal{E}$. Since ${\bf h}_Y \in  {\mathscr{M}(\psi)}_\mathcal{D}^C$ for each $Y \in Ob(\mathcal{D})$, we have  
\begin{equation*}
\rho_{XY}(f)=\rho_{XY}({\bf h}_Y(f)(id_Y))={\bf h}_Y(f_\psi)({id_Y}_0)\otimes {{id_Y}_1}^\psi ={id_Y}_0f_\psi \otimes  {{id_Y}_1}^\psi=({id_Y}_0\otimes {id_Y}_1)\cdot f
\end{equation*}
for any $f \in Hom_\mathcal{D}(X,Y)$. But, if $f \in Hom_\mathcal{E}(X,Y)$, then we also have
\begin{equation*}
\rho_{XY}(f)=\rho_{XY}(f\circ id_X)=f \cdot \rho_{XX}(id_X)=f\circ {id_X}_0 \otimes {id_X}_1=f\cdot ( {id_X}_0 \otimes {id_X}_1)
\end{equation*}
\end{example}
 
 \begin{prop}\label{lem4.16v}
Let $\mathcal{E} \subseteq \mathcal{D}$ be a subcategory and $\mathscr{C}$ be a $\mathcal{D}$-coring. Let $\{s_X\}_{X \in Ob(\mathcal{E})}$ be a group-like collection for $\mathscr{C}$ with respect to $\mathcal{E}$. For a right $\mathscr{C}$-comodule $(\mathcal{N},\rho_{\mathcal N})$, the 
$\mathcal{E}$-submodule $\mathcal{N}^{co\mathscr{C}}:\mathcal{E}^{op} \longrightarrow Vect_K$ of coinvariants of $\mathcal{N}$ is given by:
\begin{equation*}
\begin{array}{l}
\mathcal{N}^{co\mathscr{C}}(X):=\{n \in \mathcal{N}(X) ~|~ \rho_{\mathcal N}(X)(n)=n \otimes s_{X}\}\\
\mathcal{N}^{co\mathscr{C}}(f)(n'):=\mathcal{N}(f)(n')\\
\end{array}
\end{equation*}
for any $X \in Ob(\mathcal{E})$, $f \in Hom_\mathcal{E}(X,Y)$ and $n' \in \mathcal{N}^{co\mathscr{C}}(Y)$. 
\end{prop}
\begin{proof}
We will show that for any $f \in Hom_\mathcal{E}(X,Y)$, the morphism $\mathcal{N}^{co\mathscr{C}}(f): \mathcal{N}^{co\mathscr{C}}(Y) \longrightarrow \mathcal{N}^{co\mathscr{C}}(X)$ is well-defined. Since $\rho_\mathcal{N}: \mathcal{N} \longrightarrow \mathcal{N} \otimes_\mathcal{D} \mathscr{C}$ is a morphism of right $\mathcal{D}$-modules, we have the following commutative diagram:
\begin{equation*}
\begin{CD}
\mathcal{N}(Y) @>\rho_\mathcal{N}(Y)>> \mathcal{N} \otimes_\mathcal{D} \mathscr{C}(Y,-) \\
@ V\mathcal{N}(f)VV        @VV (\mathcal{N} \otimes_\mathcal{D} \mathscr{C})(f)=id_\mathcal{N} \otimes \mathscr{C}(f,-) V\\
\mathcal{N}(X) @>\rho_\mathcal{N}(X)>>\mathcal{N} \otimes_\mathcal{D} \mathscr{C}(X,-) 
\end{CD}
\end{equation*}
Let $n' \in \mathcal{N}^{co\mathscr{C}}(Y)$ so that $\rho_\mathcal{N}(Y)(n')=n' \otimes s_Y$. Since $f \in Hom_\mathcal{E}(X,Y)$, using \eqref{glk} we have
\begin{equation*}
\rho_\mathcal{N}(X)\left(\mathcal{N}(f)(n')\right)= \left(id_\mathcal{N} \otimes \mathscr{C}(f,-)\right) (n' \otimes_\mathcal{D} s_Y)=n' \otimes_\mathcal{D} s_Y \cdot f=n' \otimes_\mathcal{D} f \cdot s_X=\mathcal{N}(f)(n') \otimes_\mathcal{D} s_X.
\end{equation*}
This shows that $\mathcal{N}(f)(n')=\mathcal{N}^{co\mathscr{C}}(f)(n') \in \mathcal{N}^{co\mathscr{C}}(X)$. The result follows.
\end{proof}

The next result shows that in the case of a $C$-Galois extension $\mathcal{E} \subseteq \mathcal{D}$, we recover the notion of coinvariants as in Definition \eqref{coinv*}.
 
\begin{lemma}
Let $\mathcal{D}$ be a $C$-Galois extension of $\mathcal{E}$. Consider the collection $\{{id_X}_0 \otimes {id_X}_1 \in Hom_\mathcal{D}(X,X) \otimes C\}_{X \in Ob(\mathcal{D})}$ which is group-like for $h\otimes C$ with respect to $\mathcal E$. Then, $\left(Hom_\mathcal{D}(-,Y)\right)^{co(h \otimes C)}(X) = Hom_\mathcal{E}(X,Y)$ for any  $X$, $Y\in 
Ob(\mathcal D)=Ob(\mathcal E)$.
\end{lemma}
\begin{proof} Since $\mathcal D$ is a $C$-Galois extension of $\mathcal E$, we know that there is a canonical entwining $(\mathcal D,C,\psi)$ such that  ${\bf h}_Y \in  {\mathscr{M}(\psi)}_\mathcal{D}^C$. Using Proposition \ref{comodovercoring}, ${\bf h}_Y$ may be treated as an object of $Comod\text{-}(h\otimes C)$. 
Let $g \in \left(Hom_\mathcal{D}(-,Y)\right)^{co(h \otimes C)}(X)$. Then, $\rho_{XY}(g)=g \circ {id_X}_0 \otimes {id_X}_1$. Using the fact that ${\bf h}_Y \in {\mathscr{M}(\psi)}_\mathcal{D}^C $ we have 
\begin{equation*}
\begin{array}{ll}
\rho_{ZY}(gf)&=({\bf h}_Y \otimes C)(f)(\rho_{XY}(g))=({\bf h}_Y \otimes C)(f)\left(g \circ {id_X}_0 \otimes {id_X}_1\right)\\
&=g \circ {id_X}_0 \circ f_\psi \otimes{ {id_X}_1}^\psi=({_Z}{\bf h} \otimes C)(g)({id_X}_0 \circ f_\psi \otimes{ {id_X}_1}^\psi)=({_Z}{\bf h} \otimes C)(g)\rho_{ZX}(f)
\end{array}
\end{equation*}
for any $f \in Hom_\mathcal{D}(Z,X)$. Therefore, $g \in Hom_\mathcal{E}(X,Y)$. The converse follows directly using the Definition \eqref{coinv*}.
\end{proof}

\begin{lemma}\label{4/18i}
Let $\mathcal{D}$ be a $C$-Galois extension of $\mathcal{E}$ and let $(\mathcal{D},C,\psi)$ be the canonical entwining structure associated to it. We denote by 
$\rho_{XY}:Hom_\mathcal{D}(X,Y)\longrightarrow Hom_\mathcal{D}(X,Y)\otimes C$ the right $C$-comodule structure maps. Then, for any $\mathcal{M}\in $  $Mod\text{-}\mathcal{E}$, we may obtain an object $\mathcal{M}\otimes_\mathcal{E} h\in {\mathscr{M}(\psi)}_\mathcal{D}^C $ by setting
\begin{equation*}
\begin{array}{l}
(\mathcal{M}\otimes_\mathcal{E} h)(Y):=\mathcal{M}\otimes_\mathcal{E} {_Y{\bf h}}\qquad 
(\mathcal{M}\otimes_\mathcal{E} h)(f)(m \otimes g):=m \otimes gf
\end{array}
\end{equation*}
for   $f \in Hom_\mathcal{D}(X,Y)$ and $m\otimes g\in \mathcal M(Z)\otimes   {_Y{\bf h}}(Z)$. In fact, this determines a functor
from $Mod\text{-}\mathcal E$ to $ {\mathscr{M}(\psi)}_\mathcal{D}^C$.

\end{lemma}

\begin{proof}
Clearly, $\mathcal{M}\otimes_\mathcal{E} h \in$ $Mod\text{-}\mathcal{D}$. For each $Y \in Ob(\mathcal{D})$, it may be verified that $\mathcal{M} \otimes_\mathcal{E} {_Y{\bf h}}$ has a right $C$-comodule structure given by
$$\mathcal{M} \otimes_\mathcal{E} {_Y{\bf h}}\xrightarrow{id\otimes \rho} \mathcal{M} \otimes_\mathcal{E} {_Y{\bf h}}\otimes C~~~~~~~~~~~~m\otimes g\mapsto m\otimes \rho_{YZ}(g)$$
for any $g\in Hom_\mathcal{D}(Y,Z)$ and $m\in \mathcal{M}(Z)$.  By Theorem \ref{Gal-ent}, ${\bf h}_Z$ is an object in ${\mathscr{M}(\psi)}_\mathcal{D}^C$ for every $Z\in Ob(\mathcal{D})$ with its canonical $\mathcal{D}$-module structure and right $C$-coactions $\{\rho_{XZ}\}_{X\in Ob(\mathcal{D})}$. Therefore, we have 
$ \rho_{XZ}({\bf h}_Z(f)(g)))= (gf)_0\otimes (gf)_1= g_0f_\psi \otimes g_1^\psi
$ for any $f \in Hom_\mathcal{D}(X,Y)$. Consequently, we have
\begin{equation} (id\otimes \rho_{XZ})\left((\mathcal{M}\otimes_\mathcal{E} h)(f)(m \otimes g)\right)= m\otimes (gf)_0\otimes (gf)_1=m \otimes g_0f_\psi \otimes g_1^\psi
\end{equation}  This shows that $\mathcal{M}\otimes_\mathcal{E} h\in {\mathscr{M}(\psi)}_\mathcal{D}^C $.
\end{proof}

\begin{lemma}\label{Cor4.15j}
Let $\mathcal{D}$ be a $C$-Galois extension of $\mathcal{E}$. If there exists a convolution invertible collection $\Phi=\{\Phi_{XY}:C\longrightarrow Hom_\mathcal{D}(X,Y)\}_{X,Y\in Ob(\mathcal{D})}$ of right
$C$-comodule maps, then 

\smallskip
(i) $Hom_\mathcal{D}(X,-)$ is flat as a left $\mathcal{E}$-module.

\smallskip
(ii) $\underset{X\in Ob(\mathcal{D})}{\bigoplus}Hom_\mathcal{D}(X,-)$  is faithfully flat as a left $\mathcal{E}$-module.

\smallskip
(iii) For any $\mathcal{M}\in $ $Mod\text{-}\mathcal{E}$, there is a monomorphism $\mathcal{M}\hookrightarrow \mathcal{M}\otimes_\mathcal{E} h$ in  $Mod\text{-}\mathcal{E}$ given by $m \mapsto m\otimes id_X$ for any $m \in \mathcal{M}(X)$.
\end{lemma}

\begin{proof}
(i) Let $i: \mathcal{M}_1\hookrightarrow \mathcal{M}_2$ be a monomorphism of right $\mathcal{E}$-modules. By Proposition \ref{DEC}, it follows that the induced map $\mathcal M_1\otimes_{\mathcal E} Hom_\mathcal{D}(X,-)\longrightarrow \mathcal M_2\otimes_{\mathcal E} Hom_\mathcal{D}(X,-)$ coincides with the map $\mathcal{M}_1(X)\otimes C\xrightarrow{i(X)\otimes id_C} \mathcal{M}_2(X)\otimes C$ for each $X\in Ob(\mathcal{E})= Ob(\mathcal{D})$. Since $i(X)\otimes id_C$ is clearly a monomorphism, it follows that $Hom_\mathcal{D}(X,-)$ is flat as a left $\mathcal{E}$-module. 

\smallskip
(ii) This is clear from the fact that $\mathcal{M}(X) \otimes C=\mathcal{M} \otimes_{\mathcal{E}} Hom_\mathcal{D}(X,-)=0$ $\Rightarrow$ $\mathcal{M}(X)=0$.

\smallskip
(iii) Since $\underset{Y\in Ob(\mathcal{D})}{\bigoplus}Hom_\mathcal{D}(Y,-)$   is faithfully flat as a left $\mathcal{E}$-module, it is enough to prove that for each $Y\in Ob(\mathcal D)$, we have a monomorphism 
\begin{equation}\label{monol}\mathcal{M}\otimes_\mathcal{E} Hom_\mathcal{D}(Y,-)\longrightarrow \mathcal{M}\otimes_\mathcal{E} h \otimes_\mathcal{E} Hom_\mathcal{D}(Y,-)
\quad \mathcal M(X)\otimes {_Y{\bf h}}(X)\ni m\otimes f \mapsto m \otimes id_X \otimes f
\end{equation} This is true because the morphism in \eqref{monol} has a  section 
\begin{equation}\mathcal{M}\otimes_\mathcal{E} h \otimes_\mathcal{E} Hom_\mathcal{D}(Y,-)\longrightarrow \mathcal{M}\otimes_\mathcal{E} Hom_\mathcal{D}(Y,-) \qquad m' \otimes g' \otimes f'\mapsto m \otimes g'f'
\end{equation} for any $m'\in \mathcal M(Z)$ and $g'\otimes f'\in {_X{\bf h}}(Z)\otimes {_Y{\bf h}}(X)$.
\end{proof}

\begin{theorem}\label{entlast}
Let $\mathcal{D}$ be a $C$-Galois extension of $\mathcal{E}$ and let $(\mathcal{D},C,\psi)$ be the canonical entwining structure associated to it. 
Suppose there exists a convolution invertible collection $\Phi=\{\Phi_{XY}:C\longrightarrow Hom_\mathcal{D}(X,Y)\}_{X,Y\in Ob(\mathcal{D})}$ of right
$C$-comodule maps. Then, the categories ${\mathscr{M}(\psi)}_\mathcal{D}^C $ and Mod-$\mathcal{E}$ are equivalent.
\end{theorem}
\begin{proof} We consider the collection $\{{id_X}_0 \otimes {id_X}_1 \in Hom_\mathcal{D}(X,X) \otimes C\}_{X \in Ob(\mathcal{E})}$ which is group-like for the coring $h\otimes C$ with respect to $\mathcal E$.
We define 
\begin{equation*}
\begin{array}{ll}
\mathscr{F}: {Mod\text{-}}\mathcal{E}&\longrightarrow {\mathscr{M}(\psi)}_\mathcal{D}^C \qquad \mathcal{M}\mapsto \mathcal{M}\otimes_\mathcal{E} h\\
\mathscr{G}: {\mathscr{M}(\psi)}_\mathcal{D}^C &\longrightarrow  {Mod\text{-}}\mathcal{E}\qquad \mathcal{N}\mapsto \mathcal{N}^{co(h\otimes C)}
\end{array}
\end{equation*}
Using Lemma \ref{4/18i} and Proposition \ref{lem4.16v}, we see that the functors $\mathscr{F}$ and $\mathscr{G}$ are well-defined. We now verify that $\mathscr{G}\circ \mathscr{F}\cong id_{{Mod\text{-}}\mathcal{E}}$ i.e., $( \mathcal{M}\otimes_\mathcal{E} h)^{co(h\otimes C)}\cong \mathcal{M}$ for any $\mathcal{M}\in  Mod\text{-}\mathcal{E}$.

\smallskip
From Lemma \ref{L4.19cop}, we know that $h \otimes C \cong h \otimes_\mathcal{E} h$ as $\mathcal{D}$-corings. Under this isomorphism, the collection $\{{id_X}_0 \otimes {id_X}_1 \in Hom_\mathcal{D}(X,X) \otimes C\}_{X \in Ob(\mathcal{E})}$ maps to the collection $\{{id_X} \otimes {id_X} \in {\bf h}_X  \otimes {_X}{\bf h} \}_{X \in Ob(\mathcal{E})}$ which is group-like for $h\otimes_\mathcal{E} h$ with respect to $\mathcal E$. Therefore, it suffices to show that $\mathcal{M} \cong ( \mathcal{M}\otimes_\mathcal{E} h)^{co(h\otimes_\mathcal{E} h)}$.

\smallskip
By Lemma \eqref{Cor4.15j}(iii), we have an inclusion $i:\mathcal{M} \longrightarrow \mathcal{M} \otimes_\mathcal{E} h$ of right $\mathcal{E}$-modules. It is clear that $i(M) \subseteq  ( \mathcal{M}\otimes_\mathcal{E} h)^{co(h\otimes_\mathcal{E} h)}$.
 By definition, $\tilde{\rho}=\rho_{\mathcal M \otimes_\mathcal{E} h}: \mathcal{M}\otimes_\mathcal{E} h\longrightarrow (\mathcal{M}\otimes_\mathcal{E} h)\otimes_\mathcal{D} (h\otimes_\mathcal{E} h )$  is determined by 
$$\tilde{\rho}(X)(m \otimes f) = m  \otimes_\mathcal{E} id_Y\otimes_\mathcal{E} f \qquad \forall~ m \otimes f \in  \mathcal{M}(Y)\otimes {_X{\bf h}}(Y)$$
for each $X\in Ob(\mathcal{D})$.
The coinvariants $( \mathcal{M}\otimes_\mathcal{E} h)^{co(h\otimes_\mathcal{E} h)}: \mathcal{E}^{op}  \longrightarrow Vect_K$ are given by
\begin{equation*}
\begin{array}{ll}
(\mathcal{M} \otimes_\mathcal{E} h)^{co(h \otimes_\mathcal{E} h)}(X)=\{\sum_{Y \in Ob(\mathcal E)}  m_Y \otimes f_Y \in \mathcal{M} \otimes {_X}{\bf h} ~|~ \tilde{\rho}(X)(\sum m_Y \otimes f_Y)=\sum m_Y  \otimes_\mathcal{E} f_Y \otimes_\mathcal{E} id_X\}\\
\end{array}
\end{equation*}

 For ${\sum} m_Y \otimes f_Y \in (\mathcal{M} \otimes_\mathcal{E} h)^{co(h \otimes_\mathcal{E} h)}(X)$, we now have
\begin{equation}\label{4.16.1}
\tilde{\rho}(X)(\sum m_Y \otimes f_Y )=\sum m_Y\otimes_\mathcal{E} f_Y \otimes_\mathcal{E} id_X=\sum m_Y  \otimes_\mathcal{E} id_Y\otimes_\mathcal{E} f_Y \in  (\mathcal{M} \otimes_\mathcal{E} h)\otimes_\mathcal{E} {_X {\bf h}}
\end{equation}
We set $ \mathcal{P}:= (\mathcal{M}\otimes_\mathcal{E} h)/\mathcal{M} \in Mod\text{-}\mathcal{E}$ and consider the following short exact sequence:
$$0\longrightarrow \mathcal{M}\overset{i}{\longrightarrow} \mathcal{M}\otimes_\mathcal{E} h \overset{\eta}{\longrightarrow} \mathcal{P}\longrightarrow 0$$
Then $\eta$ induces the morphism $\eta\otimes id_h: (\mathcal{M} \otimes_\mathcal{E} h)\otimes_\mathcal{E} h\longrightarrow \mathcal{P}\otimes_\mathcal{E} h$ of right $\mathcal{E}$-modules which for each $X\in Ob(\mathcal{D})$ is given by

\begin{equation*}
(\eta\otimes id_h)(X): (\mathcal{M} \otimes_\mathcal{E} h)\otimes_\mathcal{E} {_X {\bf h}}\longrightarrow \mathcal{P}\otimes_\mathcal{E} {_X {\bf h}} \qquad m' \otimes f'\otimes g' \mapsto \eta(Y)(m' \otimes f')\otimes g'
\end{equation*}
where $m' \in \mathcal{M}(Z)$, $f'\in Hom_\mathcal{D}(Y,Z)$, $g'\in Hom_\mathcal{D}(X,Y)$ and $Y,Z\in Ob(\mathcal{E})$.
Applying $(\eta\otimes id_h)(X)$ to \eqref{4.16.1}, we obtain 
\begin{equation}\label{eq4.19}
\sum \eta(X)(m_Y \otimes_\mathcal{E} f_Y)\otimes_\mathcal{E} id_X=\sum \eta(Y)(m_Y \otimes_\mathcal{E}id_Y)\otimes_\mathcal{E} f_Y= \sum \eta(Y)(i(Y)(m_Y))\otimes_\mathcal{E} f_Y=0
\end{equation}
Applying Lemma \eqref{Cor4.15j}(iii) to the inclusion $\mathcal P \hookrightarrow \mathcal P \otimes_\mathcal E h$, it follows from \eqref{eq4.19} that $\sum\eta(X)(m_Y \otimes_\mathcal{E} f_Y)=0$ for every $X\in Ob(\mathcal{E})$. Therefore, ${\sum} m_Y \otimes f_Y \in i(\mathcal{M})(X)$. This proves that $\mathcal{M} \cong ( \mathcal{M}\otimes_\mathcal{E} h)^{co(h\otimes_\mathcal{E} h)}$. 

\smallskip
It remains to show that $\mathscr{F}\circ \mathscr{G}\cong id_{{\mathscr{M}(\psi)}_\mathcal{D}^C}$. Let $\mathcal{N}\in {\mathscr{M}(\psi)}_\mathcal{D}^C \cong$ $Comod\text{-}(h\otimes C)$. Then, $\mathcal{N}$ is a right $\mathcal{D}$-module with a given morphism 
$$\rho_\mathcal{N}:\mathcal{N} \longrightarrow \mathcal{N} \otimes_\mathcal{D} (h\otimes C)\cong \mathcal{N}\otimes_\mathcal{D} (h\otimes_\mathcal{E}h)\cong \mathcal{N}\otimes_\mathcal{E}h$$
in ${\mathscr{M}(\psi)}_\mathcal{D}^C$. By definition, $\mathcal{N}^{co(h\otimes C)}$ is the equalizer of the following morphisms

\begin{equation}
0\longrightarrow \mathcal{N}^{co(h \otimes C)} \longrightarrow \xymatrix{\mathcal N\ar@<-.5ex>[r]_(0.4){j} \ar@<.5ex>[r]^(0.4){\rho_\mathcal N}&\mathcal N\otimes_{\mathcal E}h}
\end{equation} where $j$ is given by 
\begin{equation*} j(X):\mathcal N(X)\longrightarrow \mathcal N\otimes_{\mathcal E}{_X{\bf h}} \qquad n\mapsto n\otimes id_X
\end{equation*} for every $X\in Ob(\mathcal D)$. By Lemma \ref{Cor4.15j}(i), it follows that $\mathcal{N}^{co(h\otimes C)}\otimes_\mathcal{E}{_X}{\bf h}$ is the equalizer of the following morphisms
\begin{equation}\label{eq4.21}
0\longrightarrow \mathcal{N}^{co(h \otimes C)}\otimes_{\mathcal E}{_X{\bf h}} \longrightarrow \xymatrix{\mathcal N\otimes_{\mathcal E}{_X{\bf h}}\ar@<-.5ex>[r]_(0.4){j\otimes id} \ar@<.5ex>[r]^(0.4){\rho_\mathcal N\otimes id}&\mathcal N\otimes_{\mathcal E}h\otimes_{\mathcal E}{_X{\bf h}}}
\end{equation}
Comparing with \eqref{deltaXY}, we observe that $j \otimes id=id_{\mathcal N} \otimes_{\mathcal E} \Delta_{h \otimes_{\mathcal E} h}(X,-)$.
Using the coassociativity of $\rho_\mathcal{N}: \mathcal{N}\longrightarrow \mathcal{N}\otimes_\mathcal{E}{h}$, it follows from \eqref{eq4.21} that $\rho_\mathcal{N}(X)$ factorises through $\mathcal{N}^{co(h\otimes C)}\otimes_\mathcal{E}{_X {\bf h}}$, which is denoted by $\rho'_\mathcal{N}(X):\mathcal{N}(X)\longrightarrow \mathcal{N}^{co(h\otimes C)}\otimes_\mathcal{E}{_X {\bf h}}\subseteq \mathcal{N}\otimes_\mathcal{E}{_X {\bf h}}$.

\smallskip
 We claim that $\rho'_\mathcal{N}:\mathcal{N}\longrightarrow \mathcal{N}^{co(h\otimes C)}\otimes_\mathcal{E}{h}$ is an isomorphism in ${\mathscr{M}(\psi)}_\mathcal{D}^C$. From the counit property, we know that $(id_\mathcal N \otimes_{\mathcal D} \varepsilon_{h \otimes_\mathcal E h})\circ \rho_{\mathcal N}=id_{\mathcal N}$. Hence, $\rho_\mathcal{N}$ is a monomorphism and so is $\rho'_{\mathcal N}$. It remains  to show that $\rho'_\mathcal{N}(X)$ is an epimorphism for each $X\in Ob(\mathcal{D})$. For each $X\in Ob(\mathcal{D})$, we define 
\begin{equation*}\zeta(X):\mathcal{N}^{co(h\otimes C)}\otimes_\mathcal{E}{{_X{\bf h}}} \longrightarrow \mathcal{N}(X)\qquad \sum_{Y\in Ob(\mathcal D)}n_Y\otimes f_Y\mapsto \sum_{Y\in Ob(\mathcal D)}\mathcal{N}(f_Y)(n_Y)
\end{equation*}
Since $\rho'_\mathcal{N}$ is a morphism of right $\mathcal{D}$-modules, we now have
\begin{align*}
\rho'_\mathcal{N}(X)\left(\zeta(X)\left(n_Y\otimes f_Y\right)\right)= \rho'_\mathcal{N}(X)(\mathcal{N}(f_Y)(n_Y))&=(\mathcal{N}^{co(h\otimes C)}\otimes_\mathcal{E}{{h}})(f_Y)(\rho'_\mathcal{N}(Y)(n_Y))\\
&= (\mathcal{N}^{co(h\otimes C)}\otimes_\mathcal{E}{{h}})(f_Y)(n_Y\otimes id_Y)= n_Y\otimes f_Y
\end{align*}
This shows that $\mathscr{F}\circ \mathscr{G}\cong id_{{\mathscr{M}(\psi)}_\mathcal{D}^C}$.

\end{proof}

\end{document}